\documentclass[11pt,letterpaper,reqno]{amsart}
\usepackage{amssymb}
\usepackage{mathtools}
\usepackage{mathrsfs}
\usepackage{empheq}
\usepackage[left=.75 in, right=.75 in,top=.75 in, bottom=.75 in]{geometry}
\usepackage{stackengine}
\usepackage{bm}
\usepackage{nameref,hyperref,cleveref}

\newtheorem{theorem}{Theorem}[section]
\newtheorem{lemma}[theorem]{Lemma}
\newtheorem{definition}[theorem]{Definition}
\newtheorem{remark}[theorem]{Remark}
\numberwithin{equation}{section}

\newcommand{\abs}[1]{\left\lvert#1\right\rvert}

\newcommand{\norm}[1]{\left\lVert#1\right\rVert}

\stackMath
\newcommand\tsup[2][2]{%
 \def\useanchorwidth{T}%
  \ifnum#1>1%
    \stackon[-.5pt]{\tsup[\numexpr#1-1\relax]{#2}}{\scriptscriptstyle\sim}%
  \else%
    \stackon[.5pt]{#2}{\scriptscriptstyle\sim}%
  \fi%
}

\newcommand{\wiener}[2][\nu]{\dot{\mathcal{F}}_{#1}^{#2,1}(\mathbb{T})}
\newcommand{\wieners}[2][\nu]{\mathcal{F}_{#1}^{#2,1}(\mathbb{T})}
\newcommand{\wienersn}[1]{\mathcal{F}^{#1,1}(\mathbb{T})}
\newcommand{\norms}[3][\nu]{\left\lVert#2\right\rVert_{\dot{\mathcal{F}}_{#1}^{#3,1}}}

\newcommand{\T}{\mathbb T}

\newcommand{\avec}{{\bf A}}
\newcommand{\fmat}{{\bf F}}
\newcommand{\kerk}{{\bf K}}
\newcommand{\dmat}{{\bf S }}

\newcommand{\gmat}{\fmat^\gamma}

\newcommand{\mc}[1]{\mathcal{#1}}
\newcommand{\ident}{\mc{I}}
\newcommand{\compt}{\mc{U}}
\newcommand{\mompt}{\mc{V}}
\newcommand{\bobpt}{\mc{Z}}

\newcommand{\nrmv}{\bm{n}}
\newcommand{\tgtv}{\bm{\tau}}

\newcommand{\SB}{2\sin{\left(\beta/2\right)}}
\newcommand{\DBL}{D_\beta}
\newcommand{\DBTF}{\DBL f(\alpha)}
\newcommand{\DBTZ}{\DBL z(\alpha)}
\newcommand{\DBB}{\delta_\beta}

\newcommand{\DBZA}{\DBB z(\alpha)}

\newcommand{\eqdef}{\overset{\mbox{\tiny{def}}}{=}}
\newcommand{\pv}{\text{pv}}

\newcommand{\diffEXT}{\mathcal{H}}
\newcommand{\noCANCEL}{\mathcal{G}}

\begin{document}
	
\title[The Two-Phase Brinkman Problem]{Critical global well-posedness for the two-phase Brinkman problem with surface tension}
\author{Jae Ho Choi}
\address{
	Department of Mathematical Sciences\\
	Carnegie Mellon University\\
	Pittsburgh, PA 15213, USA
}
\email[J.H. Choi]{jaechoi@andrew.cmu.edu}




	
\begin{abstract} 
	We study a system in which a fluid occupying a bounded simply connected region in \(\mathbb{R}^{2}\) is surrounded by another fluid with sharp boundary. They are incompressible Brinkman flows of the same viscosity saturating a porous medium with constant permeability. They interact via surface tension on their interface. We assume that the velocity has no jump across the interface and decays at infinity. We establish the asymptotic stability of the circular interface, which is a steady-state solution to our system. The technical threshold for the size of the initial perturbation for asymptotic stability can be explicitly calculated. We further show that the initial perturbation decays exponentially. We prove the existence, uniqueness, and continuous dependence on initial data of highly regular solutions by containing the perturbation variable in a Wiener-type algebra with a time-activated exponential weight, which allows for analytic solutions with much coarser initial data. The solution is contained in the Wiener-type algebra with the critical scaling exponent for our problem.
\end{abstract}

\maketitle

\section{Introduction}
In 1856, Henry Darcy published a study of the public water supply system for the city of Dijon, France, in which he famously formulated a law governing the flow of water through sands. Known as Darcy's law today, it has since been employed to simulate fluid flow in a porous medium for a variety of situations. In one of its modern forms, Darcy's law is written as
\begin{align}
    &\bm{v}=-\frac{\kappa}{\mu}(\nabla p + \rho \bm{g}), \quad \nabla \cdot \bm{v}=0,
\end{align}
where \(\bm{v}\) is the velocity of fluid through the porous medium; \(p\) is the intrinsic fluid pressure exceeding that of the atmosphere; \(\kappa\) is the intrinsic permeability of the porous medium; \(\mu\) is the intrinsic shear viscosity of the fluid; \(\rho\) is the intrinsic fluid density, which is a constant due to the incompressibility condition; and \(\bm{g}\) is the gravitational vector. In 1949, Henri Brinkman modified Darcy's law to incorporate viscous effects of fluid flow. Brinkman's equations are given by
\begin{align} \label{brinkmaneqn}
    \nabla p = -\frac{\mu}{\kappa}\bm{v}+\mu\Delta \bm{v}, \quad \nabla \cdot \bm{v}=0,
\end{align}
which approximate steady incompressible Stokes equations in the highly pervious regime (where \(\kappa \gg 1\)) and Darcy's law without gravity in the highly impervious regime (where \(\kappa \ll 1\)). In this paper, we consider a system of two immiscible fluids saturating a porous medium with constant permeability \(\kappa\). There is a fluid occupying a bounded simply connected region in \(\mathbb{R}^{2}\), which is surrounded by another fluid with sharp boundary. Both fluids are governed by Brinkman's equations and interact with one another via surface tension along the interface. We assume that the velocity has no jump across the interface and decays at infinity. For the sake of tractability, we restrict our attention to the equal viscosity case where the common viscosity is denoted by \(\mu\).

Well-posedness has been studied in numerous contexts for Brinkman's equations and its variants. Instead of attempting to compile a comprehensive list of such works, we highlight only a few papers related to our work. Kundu, Maharana, and Mishra \cite{MR4756023} proved the existence and uniqueness of a weak solution to a model describing the motion of unsteady incompressible Brinkman flow with concentration-dependent permeability under the Korteweg stress in a bounded domain in \(\mathbb{R}^{2}\). Mohan \cite{MR4314117} showed the existence and uniqueness of a local mild solution to deterministic convective Brinkman-Forchheimer equations on the whole space \(\mathbb{R}^{d}\), \(d \geq 2\), which is a nonlinear variant of Brinkman's equations. Kumankat and Wuttanachamsri \cite{MR4545960} established the well-posedness of a discretized form of the weak formulation of Brinkman's equations in a fixed domain, which is derived from a mixed finite element scheme to model fluid flow propelled by moving solid phases. Howell, Neilan, and Walkington \cite{MR3633543} developed a mixed variational formulation of a Brinkman model in a bounded, connected domain in \(\mathbb{R}^{d}\), \(d \in \{2,3\}\), which is uniformly well-posed for degenerate coefficients under the hypothesis that a generalized Poincar\'e inequality holds. To the best of our knowledge, however, there are no prior works on the well-posedness of Brinkman's equations in a free boundary setting. In this paper, we establish the asymptotic stability of the circular interface, which is a steady-state solution to our system. The technical threshold for the size of the initial perturbation for asymptotic stability can be explicitly calculated. We further show that the initial perturbation decays exponentially. Unlike the works cited earlier, we establish the existence, uniqueness, and continuous dependence on initial data of highly regular solutions by containing the perturbation variable in a Wiener-type algebra with a time-activated exponential weight, which allows for analytic solutions with much coarser initial data.

The main idea behind our proof is to linearize the evolution equation for the perturbation variable around the steady state solution \(0\), separating it into a well understood linear operator acting on the perturbation variable, plus the remainder which is “small” in some appropriate sense that depends on the clever selection of the solution space. Since this linearization is valid only near \(0\), our main result requires a small initial data condition. This proof strategy had been used by Gancedo, Garc\'ia--Ju\'arez, Patel, and Strain \cite{MR4679708}, for example, to establish the well-posedness of Muskat bubbles under gravity. Their evolution equation is of the form
\begin{align} \label{hilbert}
	\partial_{t}g + (-\Delta)^{3/2}g = \mathfrak{R}
\end{align}
in the high-frequency regime, where \(\mathfrak{R}\) denotes the nonlinear remainder in \(g\). To ensure that \(\mathfrak{R}\) is “small,” they employ the following Wiener-type algebra with a time-activated exponential weight
\begin{align} \label{keyspaces}
	\dot{\mathcal{F}}_{\nu(t)}^{s,1} = \biggl\{f: \mathbb{T} \to \mathbb{R} \mid \sum_{k \neq 0} e^{\nu(t)\abs{k}}\abs{k}^{s}\abs{\hat{f}(k)} < \infty \biggr\},
\end{align}
where \(\nu(t) = \frac{t}{1+t}\nu_{0}\) for some \(\nu_{0}>0\). The first superscript, \(s \geq 0\), measures the regularity of functions in the space, while the second superscript, \(1\), indicates that the \(l^{1}\) norm is taken with respect to the wave number \(k\). The time-activated exponential weight inside the norm associated with (\ref{keyspaces}) leads to the remarkable property that while the solution is analytic for any positive time, the initial perturbation has low regularity. We adopt this proof strategy to establish global-in-time well-posedness for the perturbation variable. In our setting, the evolution equation takes the form (\ref{hilbert}), except that the linear operator \((-\Delta)^{3/2}\) is replaced by \((-\Delta)^{1/2}\). The lower fractional order of the negative Laplacian makes our analysis more challenging due to the relative lack of regularity. On the physical side, \((-\Delta)^{1/2}\) can be characterized as the Hilbert transform acting on the spatial differential operator.

\section{Notational Conventions}
In this section, we establish notational conventions used throughout the paper. For \(\bm{x}=(x_1,x_2)^{T}\in \mathbb{R}^{2}\), we let \(|\bm{x}|\eqdef\sqrt{\bm{x}^T\overline{\bm{x}}}=\sqrt{|x_1|^2+|x_2|^2}\) and \(\bm{x}^\perp \eqdef (-x_{2},x_{1})\). In addition, if \(\bm{y}=(y_{1}, y_{2})^{T} \in \mathbb{R}^{2}\), then
$$
\bm{x}\otimes\bm{y}\eqdef\begin{bmatrix}x_1 y_1 & x_1 y_2 \\ x_2y_1  & x_2 y_2\end{bmatrix}.
$$
The one-dimensional torus with period \(2\pi\) is denoted by \(\mathbb{T}\eqdef\mathbb{R}/ (2\pi\mathbb{Z})\). Given a function \(\avec: \mathbb{T} \to \mathbb{R}^{n}\) where \(n \in \mathbb{N}\), we define \(\DBB \avec(\alpha) \eqdef \avec(\alpha+\beta)- \avec(\alpha)\) and \(D_{\beta}\avec(\alpha) \eqdef \frac{\delta_\beta  \avec(\alpha)}{ \SB}\). The Hilbert transform of a function \(f:\mathbb{T} \to \mathbb{R}\) is denoted by
\begin{align}
    \mathcal{H}(f)(\alpha)&\eqdef \frac{1}{\pi}\pv\int_{\mathbb{T}}\frac{f(\alpha-\beta)}{2\tan{(\beta/2)}}d\beta \label{Hilbert_def}=\frac{1}{2\pi}\pv\int_{\mathbb{T}}\frac{f(\alpha-\beta)-f(\alpha+\beta)}{2\tan{(\beta/2)}}d\beta \nonumber
\end{align}
and the square root of the negative Laplacian of \(f:\mathbb{T} \to \mathbb{R}\) is denoted by
\begin{equation} \label{lambda_op}
    \mathcal{H}(\partial_\alpha f)(\alpha)=\frac{1}{\pi}\pv\int_{\mathbb{T}}\frac{f(\alpha)-f(\alpha-\beta)}{(\SB)^2}d\beta.
\end{equation}
For any \(k \in \mathbb{Z}\), we define the \(k\)th Fourier mode of \(f:\mathbb{T} \to \mathbb{R}\) via
\begin{equation}\label{Fourier_coef}
    \begin{aligned}
        \widehat{f}(k)=\mathcal{F}(f)(k) \eqdef \frac{1}{2\pi}\int_{\mathbb{T}}f(\alpha)e^{-ik\alpha}d\alpha.
    \end{aligned}
\end{equation}
Moreover,
\begin{equation}\label{vectors}
    \bm{n}(\alpha)\eqdef
    \begin{bmatrix}
        \cos{\alpha} \\
        \sin{\alpha}
    \end{bmatrix}
    , \quad \tgtv(\alpha) \eqdef
    \begin{bmatrix}
        -\sin{\alpha}\\
        \cos{\alpha}
    \end{bmatrix}
    .
\end{equation}
For \(n \in \mathbb{N}\), we denote the closure and the boundary of a set \(S \subseteq \mathbb{R}^{n}\) by \(\overline{S}\) and \(\partial S\), respectively, the transpose of a matrix \(M \in \mathbb{R}^{n \times n}\) by \(M^{T}\), and the \(n \times n\) identity matrix by \(\ident_{n}\). Lastly, if \(M=M(g)\) is an expression involving a function \(g:\mathbb{T} \to \mathbb{R}\), then
\begin{align*}
	M_{0}\eqdef&M(0), \quad M_{L}=M_{L}(g)\eqdef\frac{d}{d\epsilon}[M(\epsilon g)]\mid_{\epsilon=0}, \mbox{ and }	M_{N}\eqdef M-M_{0}-M_{L}.
\end{align*}

\section{Problem Formulation} \label{pformulation}
In this paper, we consider the system in which a fluid occupying a bounded simply connected open region \(\Omega_{1}=\Omega_{1}(t) \subseteq \mathbb{R}^{2}\) is surrounded by another fluid in \(\Omega_{2}=\Omega_{2}(t)=\mathbb{R}^{2} \setminus \overline{\Omega_{1}(t)}\). If we denote the fluid interface by \(\Gamma =\Gamma(t)= \partial \Omega_{1}\), then the two fluids satisfy the equations
\begin{align*}
    \nabla p &= -\frac{\mu_{1}}{\kappa}\bm{u}+\mu_{1}\Delta \bm{u} \quad \mbox{in \(\Omega_{1}\)}, \quad \nabla p = -\frac{\mu_{2}}{\kappa}\bm{u}+\mu_{2}\Delta \bm{u} \quad \mbox{in \(\Omega_{2}\)}, \quad \nabla \cdot \bm{u} = 0 \quad \mbox{in \(\mathbb{R}^{2} \setminus \Gamma\)},
\end{align*}
where \(\bm{u}: \mathbb{R}^{2} \setminus \Gamma \times [0,\infty) \to \mathbb{R}^{2}\) is the fluid velocity and \(p: \mathbb{R}^{2} \setminus \Gamma \times [0,\infty) \to \mathbb{R}\) is the pressure. We assume that the two fluids saturate a porous medium with constant permeability \(\kappa>0\) and that \(\mu_{1}=\mu_{2}=\mu>0\).

Let us specify the coordinate position of \(\Gamma\) at time \(t \in [0,\infty)\) with a parametrization \(z(\cdot,t): \mathbb{T} \to \mathbb{R}^{2}\) in the counter-clockwise direction. For any function \(w\) defined on \(\Omega_{1} \cup \Omega_{2}\), we define the jump \([w] \eqdef w|_{\Gamma_{1}}-w|_{\Gamma_{2}}\) where \(w|_{\Gamma_{1}}\) and \(w|_{\Gamma_{2}}\) denote the traces of \(w\) onto \(\Gamma\) from within \(\Omega_{1}\) and \(\Omega_{2}\), respectively. The boundary conditions of the problem are then given by
\begin{align}
    [\bm{u}] &= 0, \label{nojump} \\
    [\Sigma\bm{\nu}] &= \gamma \frac{\partial}{\partial \alpha}\biggl(\frac{z_{\alpha}(\alpha,t)}{\abs{z_{\alpha}(\alpha,t)}}\biggr)\abs{z_{\alpha}(\alpha,t)}^{-1}, \label{imbal}
\end{align}
where \(\Sigma\) is the viscous stress tensor
\begin{align*}
    \Sigma =
    \begin{cases}
        \mu_{1}(\nabla \bm{u} + (\nabla \bm{u})^{T}) - p\ident_{2} \quad \mbox{in \(\Omega_{1}\)} \\
        \mu_{2}(\nabla \bm{u} + (\nabla \bm{u})^{T}) - p\ident_{2} \quad \mbox{in \(\Omega_{2}\)};
    \end{cases}
\end{align*}
\(\bm{\nu}\) is the outward-pointing unit normal vector on \(\Gamma\); and \(\gamma>0\) is the surface tension coefficient. The condition (\ref{nojump}) ensures that the normal component of the fluid velocity on the interface is well-defined. The condition (\ref{imbal}) states that surface tension is responsible for the interfacial stress imbalance driving the system. To complete the specification of the problem, we require that \(\bm{u} \to \bm{0}\) and \(p \to 0\) at infinity.

We dedicate the rest of this section to formulating a kinematic condition that forms the basis of our study of the problem.

\subsection{Boundary integral formulation of the fluid velocity}

Given a function \(\fmat\eqdef (F_{1}, F_{2})^T\) on \(\Gamma\), we impose that the fluid velocity \(\bm{u}=(u_{1}, u_{2})^{T} \in \mathbb{R}^{2} \setminus \Gamma\) be given by the single layer potential
\begin{equation} \label{slp}
    u_{j}(\bm{x}) = \frac{1}{2\pi}\int_{\Gamma}\sum_{i=1}^{2}F_{i}(s)G_{ij}(\bm{x}-\bm{y}(s))ds, \quad \bm{x} \in \mathbb{R}^{2} \setminus \Gamma, \quad j \in \{1,2\},
\end{equation}
where
\begin{align}\label{StokesFundamental}
    G_{ij}(\bm{w}) =
    &\delta_{ij}\frac{-1+\lambda \abs{\bm{w}} K_{1}(\lambda \abs{\bm{w}})+\lambda^{2}\abs{\bm{w}}^{2}K_{0}(\lambda \abs{\bm{w}})}{\lambda^{2} \abs{\bm{w}}^{2}}+\frac{w_{i}w_{j}(2-\lambda^{2}\abs{\bm{w}}^{2}K_{2}(\lambda \abs{\bm{w}}))}{\lambda^{2} \abs{\bm{w}}^{4}}
\end{align}
is the \((i,j)\)-component of the Green's function \(\bm{G}\) for incompressible Brinkman flow in \(\mathbb{R}^{2}\). Here, \(K_{n}\) is the modified Bessel function of the second kind of order \(n\) and \(\lambda\eqdef \sqrt{1/\kappa}\). Since the representation formula (\ref{slp}) exhibits no jump across the interface \(\Gamma\), the condition (\ref{nojump}) is automatically satisfied.

\subsection{Stress imbalance formulation}
Using potential theory, one can reformulate the stress imbalance condition \eqref{imbal} more conveniently. The stress \(\Sigma\) in \(\Omega_{2}\) is given by
\begin{align*}
    \Sigma_{ij}(\bm{x}) = \mu_{2}\int_{\mathbb{T}}\sum_{k=1}^{2}\mathcal{T}_{ijk}(z(\beta)-\bm{x})F_{k}(\beta)d\beta, \quad \bm{x} \in \Omega_{2}, \quad i,j \in \{1,2\},
\end{align*}
where
\begin{align}\label{Tijk}
    \mc{T}_{ijk}(\bm{w})=&\delta_{ij}\frac{w_{k}(4-\lambda^{2}\abs{\bm{w}}^{2}-2\lambda^{2}\abs{\bm{w}}^{2}K_{2}(\lambda \abs{\bm{w}}))}{2\pi \lambda^{2}\abs{\bm{w}}^{4}}+\frac{w_{i}w_{j}w_{k}(-8+\lambda^{3}\abs{\bm{w}}^{3}K_{3}(\lambda \abs{\bm{w}}))}{\pi\lambda^{2}\abs{\bm{w}}^{6}} \\
    &+\frac{(\delta_{ik}w_{j}+\delta_{jk}w_{i})(4-2\lambda^{2}\abs{\bm{w}}^{2}K_{2}(\lambda \abs{\bm{w}})-\lambda^{3}\abs{\bm{w}}^{3}K_{1}(\lambda \abs{\bm{w}}))}{2\pi \lambda^{2}\abs{\bm{w}}^{4}}.
\end{align}
Similarly, the stress \(\Sigma\) in \(\Omega_{1}\) is given by
\begin{align*}
    \Sigma_{ij}(\bm{x}) = \mu_{1}\int_{\mathbb{T}}\sum_{k=1}^{2}\mathcal{T}_{ijk}(z(\beta)-\bm{x})F_{k}(\beta)d\beta, \quad \bm{x} \in \Omega_{1}, \quad i,j \in \{1,2\}.
\end{align*}
In Einstein's summation notation,
\begin{align*}
    \left.\Sigma_{ij}(z(\alpha))\left(-\frac{z_{\alpha}(\alpha,t)^{\perp}}{\abs{z_{\alpha}(\alpha,t)}}\right)_{j}\right|_{\Gamma_{2}}=&\mu_{2}\biggl(-\frac{1}{2}F_{i}(\alpha)\abs{z_{\alpha}(\alpha)}^{-1} + \mbox{pv}\int_{\mathbb{T}}\mathcal{T}_{ijk}(z(\beta)-z(\alpha))F_{k}(\beta)\nu_{j}(\alpha)d\beta\biggr)
    \\
    \left.\Sigma_{ij}(z(\alpha))\left(-\frac{z_{\alpha}(\alpha,t)^{\perp}}{\abs{z_{\alpha}(\alpha,t)}}\right)_{j}\right|_{\Gamma_{1}}=&\mu_{1}\biggl(\frac{1}{2}F_{i}(\alpha)\abs{z_{\alpha}(\alpha)}^{-1} + \mbox{pv}\int_{\mathbb{T}}\mathcal{T}_{ijk}(z(\beta)-z(\alpha))F_{k}(\beta)\nu_{j}(\alpha)d\beta\biggr).
\end{align*}
Consequently, the stress imbalance condition (\ref{imbal}) can be written as
\begin{align}\label{viscosityjump}
    \fmat(\alpha) + 2 A_\mu\int_{\mathbb{T}} \dmat(\DBZA,z_{\alpha}(\alpha)^{\perp})\fmat(\beta+\alpha)d\beta = 2 A_\gamma \gmat(z_{\alpha}(\alpha)),
\end{align}
where
\begin{align}
    &A_\mu \eqdef \frac{\mu_{2}-\mu_{1}}{\mu_{1}+\mu_{2}}, \quad A_{\gamma} \eqdef \frac{\gamma}{\mu_1+\mu_2}, \quad \gmat(z_{\alpha}(\alpha)) \eqdef \frac{d}{d\alpha}\left(\frac{z_{\alpha}(\alpha)}{\abs{z_{\alpha}(\alpha)}}\right), \label{Amu}
\end{align}
and \(\dmat=\dmat(\DBZA,z_{\alpha}(\alpha)^\perp) \eqdef \{S_{ij}\}\), in which \(S_{ij}=\mathcal{T}_{ikj}(\delta_\beta z(\alpha))(z_{\alpha}(\alpha)^{\perp})_{k}\). In summary, the stress imbalance condition (\ref{imbal}) turns into the identity (\ref{viscosityjump}) that \(\fmat\) must satisfy. In the equal viscosity case, the identity \eqref{viscosityjump} collapses to the explicit formula for \(\fmat\):
\begin{align*}
    \fmat(\alpha)=2A_{\gamma}\frac{d}{d\alpha}\left(\frac{z_{\alpha}(\alpha)}{\abs{z_{\alpha}(\alpha)}}\right).
\end{align*}
The kinematic condition governing the evolution of our system at the boundary is given by the following no-slip condition:
\begin{equation} \label{NoSlipEQN}
     z_{t}(\alpha,t) =\bm{u}(z(\alpha,t),t)
   =  \frac{1}{2\pi}\int_{\T}\bm{G}(\DBZA) \fmat(\alpha +\beta) d\beta.
\end{equation}

For a variety of reasons, it is instructive to understand whether a partial differential equation exhibits an innate scale invariance. For one, it sheds light on the intrinsic mathematical landscape of the model, which is conducive to designing an optimal functional framework for proving the desired result. In the case of \eqref{NoSlipEQN}, there is regrettably no such scale invariance at the level of the full equation due to the presence of non-homogeneous modified Bessel functions of the second kind in the kernel \(\bm{G}\). More precisely, there exists no \(r \in \mathbb{R}\) such that if \(z(\alpha,t)\) is a solution to \eqref{NoSlipEQN}, then so too is \(z_{\epsilon}(\alpha,t)=\epsilon^{r}z(\epsilon\alpha,\epsilon t)\). While our model lacks global scale invariance, it does enjoy an asymptotic scale invariance in the following sense. In the regime where \(\abs{\bm{w}} \ll 1\), the kernel \(\bm{G}(\bm{w})\) expands so that the right hand side of \eqref{NoSlipEQN} becomes a convolutional operator acting on \(\bm{F}\), whose kernel acts like the standard free-space 2D Stokeslet. At this level, \eqref{NoSlipEQN} exhibits scale invariance with \(r=-1\). Our Brinkman flow can therefore be construed as a Stokes flow perturbed by Darcian effects.

\section{Function Spaces}
To study the well-posedness of our problem, we employ families of function spaces that generalize the Wiener algebra \(\mathcal{F}^{0,1}(\mathbb{T})\), i.e., the space of functions on \(\mathbb{T}\) whose Fourier series converge absolutely. They provide a convenient framework for formulating fractional regularity and capturing instantaneous analytic regularity. For any \(s \geq 0\), let \(\wienersn{s}\) be the space of functions on \(\mathbb{T}\) endowed with the norm
\begin{align*}
    \norm{h}_{\mathcal{F}^{s,1}}\eqdef\abs{\hat{h}(0)}+\sum_{k \in \mathbb{Z} \setminus \{0\}} \abs{k}^{s}\abs{\hat{h}(k)}
\end{align*}
and \(\wiener[]{s}\) the space of functions on \(\mathbb{T}\) equipped with the semi-norm
\begin{align*}
    \norms[]{h}{s} \eqdef \sum_{k \in \mathbb{Z} \setminus \{0\}}\abs{k}^{s}\abs{\hat{h}(k)}.
\end{align*}

For any \(s \geq 0\), \(\nu_{0}>0\), and \(t \geq 0\), let \(\mathcal{F}_{\nu,t}^{s,1}(\mathbb{T})\) be the space of functions on \(\mathbb{T}\) endowed with the norm
\begin{align*}
    \norm{h}_{\mathcal{F}_{\nu,t}^{s,1}}\eqdef \abs{\widehat{h}(0)}+\sum_{k \in \mathbb{Z} \setminus \{0\}}e^{\nu(t)\abs{k}}\abs{k}^{s}\abs{\hat{h}(k)}
\end{align*}
and \(\wiener[\nu, t]{s}\) the space of functions on \(\mathbb{T}\) equipped with the semi-norm
\begin{align} \label{expnormst}
    \norms[\nu,t]{h}{s} \eqdef\sum_{k \in \mathbb{Z} \setminus \{0\}} e^{\nu(t)\abs{k}}\abs{k}^{s}\abs{\hat{h}(k)},
\end{align}
where \(\nu(t)\eqdef\nu_{0}\frac{t}{1+t}\). We note that \(\nu'(t)=\nu_{0}/(1+t)^{2}\). For notational brevity, we drop the explicit time dependence and write \(\mathcal{F}_{\nu}^{s,1}(\mathbb{T})\) and \(\wiener{s}\) as shorthand for \(\mathcal{F}_{\nu,t}^{s,1}(\mathbb{T})\) and \(\wiener[\nu,t]{s}\), respectively.

For any \(s \geq 0\), \(\nu_{0}>0\), and \(t \geq 0\), let \(\dot{Z}_{\nu}^{1,1}(\mathbb{T})\) be the mean-zero subspace of \(\wiener{1}\), i.e., \(\dot{Z}_{\nu}^{1,1}(\mathbb{T})\eqdef\{g \in \wiener{1} : \widehat{g}(0)=0\}\). Similarly, \(\dot{Z}^{0,1}(\mathbb{T})\eqdef\{g \in \dot{\mathcal{F}}^{0,1}(\mathbb{T}) : \widehat{g}(0)=0\}\).

Lastly, we introduce the space
\begin{align}
    X_{T,\nu}(\mathbb{T})\eqdef C((0,T];\dot{Z}_{\nu}^{1,1}(\mathbb{T}))\cap L^{1}(0,T;\wiener{2}), \label{xtnu}
\end{align}
which is a Banach space when equipped with the norm
\begin{align*}
    \norm{g}_{X_{T,\nu}}=\norm{g}_{X_{T,\nu}(\mathbb{T})}\eqdef\norm{g}_{L^{\infty}(0,T;\wiener{1})}+\norm{g}_{L^{1}(0,T;\wiener{2})}.
\end{align*}

We seek a solution in the space \(C([0,T];\dot{Z}_{\nu}^{1,1}(\mathbb{T}))\) because \(s=1\) is the critical scaling exponent of the Wiener-type algebra for our problem. That is, \(\wiener{1}\) is the space whose norm is invariant under the natural scaling of our model identified at the end of Section \ref{pformulation}. 

\section{The Initial Data}
We consider an initial region \(\Omega_{1}(0)\subseteq \mathbb{R}^2\) whose boundary $\Gamma(0)$ admits a parametrization of the form
\begin{equation}\label{z0_polar}
	\begin{aligned}
		z_0(\alpha)&=\bm{x}_{c,0}+R(1+2f_0(\alpha))^{\frac12}\nrmv(\alpha),
	\end{aligned}
\end{equation}
where \(\bm{x}_{c,0}\) is the centroid of \(\Omega_{1}(0)\) defined by
\begin{equation*}
	\begin{aligned}  
		\bm{x}_{c,0}&\eqdef\frac{1}{\abs{\Omega_{1}(0)}}\int_{\Omega_{1}(0)}\bm{x} d\bm{x}, \quad \text{with} \quad \abs{\Omega_{1}(0)}\eqdef\int_{\Omega_{1}(0)}d\bm{x}.
	\end{aligned}
\end{equation*}
Here, $R=\sqrt{\abs{\Omega_{1}(0)}/\pi}$ is the effective radius, and $f_0:\mathbb{T}\to (-1,\infty)$ is a perturbation variable satisfying the constraints
\begin{equation}\label{f_cond}
	\begin{aligned}  
		\int_{-\pi}^{\pi}  f_{0}(\alpha) d\alpha =0 \quad \text{and} \quad \int_{-\pi}^{\pi}(1+2f_{0}(\alpha))^{3/2}\nrmv(\alpha)d\alpha=\bm{0}.
	\end{aligned}
\end{equation}
Due to the translational invariance of the governing system, we may assume without loss of generality that the centroid of the initial region is located at the origin, i.e., \(\bm{x}_{c,0}=\bm{0}\).

The first condition in \eqref{f_cond} ensures that the area of the region enclosed by \eqref{z0_polar} is precisely \(\abs{\Omega_{1}(0)}\), as shown by the following calculation:
\begin{align*}
	\int_{\Omega_{1}(0)}dx&=\int_{-\pi}^{\pi}\int_{0}^{R(1+2f_{0}(\alpha))^{\frac{1}{2}}}r dr d\alpha=\pi R^{2}+R^{2}\int_{-\pi}^{\pi}f_{0}(\alpha)d\alpha=\pi R^{2}=\abs{\Omega_{1}(0)}.
\end{align*}
Similarly, the second condition in \eqref{f_cond} guarantees that the centroid of the region enclosed by \eqref{z0_polar} is consistent with \(\bm{x}_{c,0}\):
\begin{align*}
	\frac{1}{\abs{\Omega_{1}(0)}}\int_{\Omega_{1}(0)}xdx=&\frac{1}{\abs{\Omega_{1}(0)}}\int_{-\pi}^{\pi}\frac{R^{3}(1+2f_{0}(\alpha))^{3/2}}{3}\nrmv(\alpha)d\alpha=\bm{0}.
\end{align*}

\section{The Parametrization Scheme}
For \(t > 0\), we parametrize \(\Gamma=\Gamma(t)\) around a time-dependent pole \(\bm{C}(t)\) via
\begin{align} \label{eveqn}
    z(\alpha,t)=\bm{C}(t)+   R(1+2f(\alpha,t))^{\frac12}\nrmv(\alpha+\vartheta(\alpha,t)),
\end{align}
where \(\bm{C}(t)\) is specified in \eqref{cCONSTdef}. We utilize this parametrization scheme to transform the kinematic condition \eqref{NoSlipEQN} into an evolution equation for \(f\). We first expand the kinematic condition \eqref{NoSlipEQN} as
\begin{align} \label{noslipcomps}
	z_t(\alpha,t) = U(\alpha, t) {\bm{n}}(\alpha+\vartheta(\alpha,t)) + T(\alpha, t) {\bm{\tau}}(\alpha+\vartheta(\alpha,t)),
\end{align}
where the normal and tangential velocities are given by
\begin{align*}
	T(\alpha, t)&=\frac{1}{2\pi}\bm{\tau}(\alpha+\vartheta(\alpha,t))\cdot\int_{\T}\bm{G}(\DBZA) \fmat(\alpha +\beta) d\beta, \\
	U(\alpha, t)&=\frac{1}{2\pi}\bm{n}(\alpha+\vartheta(\alpha,t))\cdot\int_{\T}\bm{G}(\DBZA) \fmat(\alpha +\beta) d\beta.
\end{align*}
Differentiating \eqref{eveqn} with respect to \(t\), we obtain
\begin{align*}
	z_{t}(\alpha,t)=\dot{\bm{C}}(t)+Rf_{t}(\alpha,t)(1+2f(\alpha,t))^{-\frac{1}{2}}\nrmv(\alpha+\vartheta(\alpha,t)) +R(1+2f(\alpha,t))^{\frac{1}{2}}\tgtv(\alpha+\vartheta(\alpha,t))\vartheta_{t}(\alpha,t).
\end{align*}
We plug it into the left hand side of \eqref{noslipcomps}. 
Projecting the resulting equation onto \({\bm{n}}(\alpha+\vartheta(\alpha,t))\) and \({\bm{\tau}}(\alpha+\vartheta(\alpha,t))\) yields
\begin{align*}
	f_{t}(\alpha,t)(1+2f(\alpha,t))^{-\frac{1}{2}}=&\frac{1}{2\pi R}\nrmv(\alpha+\vartheta(\alpha,t))\cdot\int_{\T}\bm{G}(\DBZA)\bm{F}(\alpha+\beta)d\beta-\frac{1}{R}\dot{\bm{C}}(t)\cdot\nrmv(\alpha+\vartheta(\alpha,t))
\end{align*}
and
\begin{align*}
	\vartheta_{t}(\alpha,t)=&\frac{1}{2\pi}\frac{\tgtv(\alpha+\vartheta(\alpha,t))}{R(1+2f(\alpha,t))^{\frac{1}{2}}}\cdot\int_{\T}\bm{G}(\DBZA)\bm{F}(\alpha +\beta)d\beta-\frac{\dot{\bm{C}}(t)\cdot\tgtv(\alpha+\vartheta(\alpha,t))}{R(1+2f(\alpha,t))^{\frac{1}{2}}},
\end{align*}
respectively.

Since the evolution of the tangential phase variable \(\vartheta\) does not change the physical geometry of the interface, we are free to impose that \(\vartheta \equiv 0\). Under this choice, the equation for \(\vartheta\) reduces to the constraint
\begin{align} \label{zeroCOND}
	\tgtv(\alpha)\cdot\int_{\T}\bm{G}(\DBZA) \bm{F}(\alpha +\beta) d\beta=2\pi\dot{\bm{C}}(t)\cdot\tgtv(\alpha),
\end{align}
while the equation for $f$ becomes
\begin{align} \label{evolutionEQ}
	f_{t}(\alpha,t)=&(1+2f(\alpha,t))^{\frac{1}{2}}\frac{\nrmv(\alpha)}{2\pi R} \cdot \int_{\T}\bm{G}(\DBZA)\bm{F}(\alpha+\beta)d\beta -\frac{1}{R}(1+2f(\alpha,t))^{\frac{1}{2}}\nrmv(\alpha)\cdot \dot{\bm{C}}(t).
\end{align}

In summary, for \(t>0\), we adopt the parametrization
\begin{align}\label{parametrization}
	&z(\alpha,t)=   \bm{C}(t)+R(1+2f(\alpha,t))^{\frac12}\nrmv(\alpha),
\end{align}
where the moving pole \(\bm{C}(t)\) is defined as
\begin{align}\label{cCONSTdef}
	\bm{C}(t) = -3R(\pi+\text{Im}(I_{1})(\lambda,R,1)+2\text{Re}(I_{2})(\lambda,R,1))\frac{A_{\gamma}}{2\pi R}\frac{1}{2\pi}\int_{0}^{t}\int_{\mathbb{T}}f(\alpha,s)\nrmv(\alpha)d\alpha ds
\end{align}
and the constants \(I_{1}(\lambda,R,1)\) and \(I_{2}(\lambda,R,1)\) originate from \eqref{I1I2}. As shown in subsequent sections, this precise choice of \(\bm{C}(t)\) dynamically couples the first Fourier mode of \(f\) with its higher Fourier modes, ultimately guaranteeing its asymptotic decay in time. Differentiating \eqref{cCONSTdef} with respect to \(t\) yields
\begin{align} \label{DERIVcCONST}
	\dot{\bm{C}}(t) =-3R(\pi+\text{Im}(I_{1})(\lambda,R,1)+2\text{Re}(I_{2})(\lambda,R,1))\frac{A_{\gamma}}{2\pi R}\frac{1}{2\pi}\int_{\mathbb{T}}f(\alpha,t)\nrmv(\alpha)d\alpha.
\end{align}
Consequently, under the parametrization \eqref{parametrization} and the geometric constraint \eqref{zeroCOND}, the kinematic condition \eqref{NoSlipEQN} reduces directly  to \eqref{evolutionEQ}.

\subsection{Revisiting the no-slip condition}
Let us suppose that, instead of \eqref{NoSlipEQN}, we consider the normal kinematic condition
\begin{align} \label{normalnoslip}
    z_{t}(\alpha,t) \cdot z_{\alpha}(\alpha,t)^{\perp} = \bm{u}(z(\alpha,t),t) \cdot z_{\alpha}(\alpha,t)^{\perp}.
\end{align}
Applying the parametrization \eqref{parametrization}, we obtain
\begin{equation} \notag
	z_{\alpha}(\alpha,t)^{\perp}={R}f_{\alpha}(\alpha,t)(1+2f(\alpha,t))^{-\frac{1}{2}}\tgtv(\alpha)-{R}(1+2f(\alpha,t))^{\frac{1}{2}}\nrmv(\alpha).
\end{equation}
Substituting this relation into \eqref{normalnoslip} yields the evolution equation
\begin{align*} 
	f_t(\alpha,t)
	=&-\frac{1}{R}(1+2f(\alpha,t))^{\frac12}\dot{\bm{C}}(t)\cdot\nrmv(\alpha)+\frac{1}{R}(1+2f(\alpha,t))^{-\frac12}f_{\alpha}(\alpha,t)\dot{\bm{C}}(t)\cdot \tgtv(\alpha) \\
	&-(1+2f(\alpha,t))^{-\frac12}f_{\alpha}(\alpha,t)\frac{\tgtv(\alpha)}{2\pi R}\cdot\int_{\mathbb{T}}\bm{G}(\DBZA)\bm{F}(\alpha+\beta)d\beta \\
	&+(1+2f(\alpha,t))^{\frac12}\frac{\nrmv(\alpha)}{2\pi R} \cdot \int_{\T}\bm{G}(\DBZA) \bm{F}(\alpha +\beta) d\beta.
\end{align*}
By virtue of \eqref{zeroCOND}, the second and third terms on the right hand side cancel, reducing this expression directly to \eqref{evolutionEQ}.

Consequently, the equation \eqref{normalnoslip} may serve as the governing kinematic condition for our model in place of \eqref{NoSlipEQN}. Under the constraint \eqref{zeroCOND}, both formulations reduce to the identical evolution equation (\ref{evolutionEQ}).

\section{The Evolution Equation}
To express the evolution equation \eqref{evolutionEQ} explicitly in terms of \(f\), it is necessary to derive an analytical representation for the kernel
\begin{equation} \label{kernelUSE}
    \kerk(\alpha,\beta) \eqdef (1+2f(\alpha,t))^{\frac{1}{2}}\nrmv(\alpha)\cdot \bm{G}(\DBZA) 
\end{equation}
In this notation, the evolution equation \eqref{evolutionEQ} takes the form
\begin{equation}\label{NEWequation}
	f_{t}(\alpha,t) = \frac{1}{2\pi R} \int_{\T}\kerk(\alpha,\beta) \fmat(\alpha +\beta) d\beta-(1+2f(\alpha,t))^{\frac{1}{2}}\frac{\dot{\bm{C}}(t)\cdot \nrmv(\alpha)}{R}.
\end{equation}
Substituting the fundamental solution \eqref{StokesFundamental} into \eqref{kernelUSE} yields
\begin{align*}
    \nrmv(\alpha) \cdot \bm{G}(\DBZA)=&\nrmv(\alpha)\cdot\biggl(\biggl(\frac{-1}{\lambda^{2}\abs{\DBZA}^{2}}+\frac{K_{1}(\lambda\abs{\DBZA})}{\lambda \abs{\DBZA}}+K_{0}(\lambda\abs{\DBZA})\biggr)I \\
    &+\biggl(\frac{1}{\abs{\DBTZ}^{2}}\frac{2}{\lambda^{2}\abs{\DBZA}^{2}}-\frac{K_{2}(\lambda\abs{\DBZA})}{\abs{\DBTZ}^{2}}\biggr)\DBTZ \otimes \DBTZ\biggr).
\end{align*}
To expand \(\kerk(\alpha,\beta)\), it suffices to express \(\abs{\DBZA}\), \(\DBTZ\), and \(\abs{\DBTZ}^{2}\) as functions of \(f\). We begin with some preliminary identities. Utilizing the algebraic relation
\begin{align}\label{sqrt.diff}
\sqrt{a}-\sqrt{b}
    =\frac{a-b}{\sqrt{a}+\sqrt{b}} \quad \text{for } \min\{a,b\}>0,
\end{align}
we find that
\begin{align*}
    (1+2f(\alpha))^{\frac12}-1 = f(\alpha) \frac{2}{1+(1+2f(\alpha))^{\frac12}}
    =f(\alpha)\left(1+\diffEXT_1\right),
\end{align*}
where the higher-order corrections are gathered in
\begin{align*}
    \diffEXT_1  \eqdef
    \frac{1-(1+2f(\alpha))^{\frac12}}{1+(1+2f(\alpha))^{\frac12}}
    =f(\alpha) \noCANCEL_1, \quad \noCANCEL_1  \eqdef
    \frac{-2}{\left(1+(1+2f(\alpha))^{\frac12}\right)^2}.
\end{align*}
Consequently, we expand \((1+2f(\alpha))^{\frac12} 
=1+f(\alpha)(1 + \diffEXT_1)
=1+f(\alpha) +f(\alpha)^2 \noCANCEL_1\).
This allows us to partition the parametrization \eqref{parametrization} into its linear and nonlinear components: \(z(\alpha) = z_0(\alpha)+z_L(\alpha)+z_N(\alpha)\), where
\begin{align*}
    z_0(\alpha) = & R \nrmv(\alpha), \quad z_L(\alpha) = R \nrmv(\alpha) f(\alpha)+\bm{C}(t), \quad z_N(\alpha) = R \nrmv(\alpha) f(\alpha) \diffEXT_1
    =
    R \nrmv(\alpha) f(\alpha)^2 \noCANCEL_1.
\end{align*}
Similarly,
\begin{align*}
   \DBB (1+2f(\alpha))^{\frac12}
   = &
   \DBB f(\alpha)\frac{2}{(1+2f(\alpha+\beta))^{\frac12}+(1+2f(\alpha))^{\frac12}}
   = 
     \DBB f(\alpha) \left(1+\diffEXT_2 \right),
\end{align*}
where
\begin{align*}
    \diffEXT_2  \eqdef &
    \frac{2-(1+2f(\alpha+\beta))^{\frac12}-(1+2f(\alpha))^{\frac12}}{(1+2f(\alpha+\beta))^{\frac12}+(1+2f(\alpha))^{\frac12}}=f(\alpha+\beta) \noCANCEL_2
     +f(\alpha)  \noCANCEL_3,
\end{align*}
in which
\begin{align*}
\noCANCEL_2
    \eqdef &
 \frac{-2}{\left((1+2f(\alpha+\beta))^{\frac12}+(1+2f(\alpha))^{\frac12}\right) \left(1+(1+2f(\alpha+\beta))^{\frac12} \right)},
        \\
 \noCANCEL_3
    \eqdef    &
 \frac{-2}{\left((1+2f(\alpha+\beta))^{\frac12}+(1+2f(\alpha))^{\frac12}\right) \left(1+(1+2f(\alpha))^{\frac12} \right)}.
\end{align*}

Using these preliminary calculations, we obtain
\begin{align}
\DBTZ=&R\nrmv(\alpha) \cos(\beta)\DBTF  \left(1+\diffEXT_2 \right)-R\nrmv(\alpha) \sin(\beta/2)\left(1+f(\alpha)\left(1+\diffEXT_1\right) \right) \label{eqDBTZzero} \\
&+R\tgtv(\alpha)\cos(\beta/2)\left(  1+f(\alpha+\beta)+\DBB f(\alpha) \diffEXT_2+ f(\alpha) \diffEXT_1\right). \notag
\end{align}
To simplify the notation, we introduce the auxiliary functions
\begin{align}
    \compt=\compt(f) \eqdef & \cos(\beta)\DBTF \left(1+\diffEXT_2 \right)  
    -\sin(\beta/2)\left(1+f(\alpha)+f(\alpha)\diffEXT_1 \right),\label{DEFcompt} \\
	\mompt=\mompt(f) \eqdef & \cos(\beta/2) \left(  1+f(\alpha+\beta)+\DBB f(\alpha) \diffEXT_2+ f(\alpha) \diffEXT_1\right). \label{DEFmompt}
\end{align}
We can then write \(\DBTZ=R\nrmv(\alpha) \compt(f)+R\tgtv(\alpha)\mompt(f)\), which implies that \(|\DBTZ |^2= R^2 \compt(f)^2+R^2 \mompt(f)^2\). Defining \(\mathcal{Z}(f) \eqdef \compt(f)^{2}+\mompt(f)^{2}\), we obtain that \(\abs{\DBTZ}^{2}=R^{2}\mathcal{Z}(f)\) and \(\abs{\DBZA}^{2}=R^{2}\abs{2\sin(\beta/2)}^{2}\mathcal{Z}(f)\). To carefully track the linear and nonlinear contributions within \(\mathcal{Z}(f)\), we expand the squares \(\compt(f)^{2}\) and \(\mompt(f)^{2}\) as follows:
\begin{align*}
 \bobpt(f)
=&
 \compt(f)^2
+
 \mompt(f)^2
 =\left( \compt_0+\compt_L+\compt_N \right)^2
 +
 \left( \mompt_0+\mompt_L+\mompt_N \right)^2
 \\
 =& \compt_0^2+\mompt_0^2 + 2\left(\compt_0 \compt_L+\mompt_0\mompt_L \right)+\compt_L^2+\mompt_L^2 
 +
 2\left( \compt_0+\compt_L \right)\compt_N
  +
 2\left( \mompt_0+\mompt_L \right)\mompt_N
 + \compt_N^2+\mompt_N^2,
\end{align*}
where the constituent terms are given by
\begin{align*}
	&\compt_0= -\sin(\beta/2), \quad \compt_L(f)=\cos(\beta)\DBTF -\sin(\beta/2)f(\alpha), \\
	&\compt_N(f) = \cos(\beta)\DBTF \diffEXT_2 -\sin(\beta/2)f(\alpha) \diffEXT_1, \\
	&\mompt_0= \cos(\beta/2), \quad \mompt_L(f) = \cos(\beta/2)  f(\alpha+\beta), \quad \mompt_N(f) = \cos(\beta/2)  \left(\DBB f(\alpha) \diffEXT_2+ f(\alpha) \diffEXT_1\right).
\end{align*}
This decomposition yields
\begin{align}
    \mathcal{Z}_{0}=&\compt_0^2+\mompt_0^2=\sin^{2}(\beta/2)+\cos^{2}(\beta/2)=1, \\
    \mathcal{Z}_{L}=&2\left(\compt_0 \compt_L+\mompt_0\mompt_L \right)=f(\alpha)+f(\alpha+\beta), \label{ZL} \\
    \mathcal{Z}_{N}=&\compt_L^2+\mompt_L^2+2\left( \compt_0+\compt_L \right)\compt_N+2\left(\mompt_0+\mompt_L \right)\mompt_N+\compt_N^2+\mompt_N^2. \label{ZN}
\end{align}
Putting these pieces together, we obtain
\begin{align}
    &\kerk(\alpha,\beta)(1+2f(\alpha))^{-\frac{1}{2}} =\nrmv(\alpha) \cdot \label{kernelKexp} \\
    &\biggl(\biggl(\frac{-1}{\lambda^{2}R^{2}\abs{2\sin(\beta/2)}^{2}\mathcal{Z}(f)}+\frac{K_{1}(\lambda R\abs{2\sin(\beta/2)}\mathcal{Z}(f)^{\frac{1}{2}})}{\lambda R\abs{2\sin(\beta/2)}\mathcal{Z}(f)^{\frac{1}{2}}}+K_{0}(\lambda R\abs{2\sin(\beta/2)}\mathcal{Z}(f)^{\frac{1}{2}})\biggr)I \notag \\
    &+\biggl(\frac{1}{R^{2}\mathcal{Z}(f)}\frac{2}{\lambda^{2}R^{2}\abs{2\sin(\beta/2)}^{2}\mathcal{Z}(f)}-\frac{K_{2}(\lambda R\abs{2\sin(\beta/2)}\mathcal{Z}(f)^{\frac{1}{2}})}{R^{2}\mathcal{Z}(f)}\biggr) \notag \\
    &\biggl(R^{2}\compt(f)^{2}\nrmv(\alpha) \otimes \nrmv(\alpha)+R^{2}\compt(f)\mompt(f)\nrmv(\alpha) \otimes \tgtv(\alpha) \notag \\
    &+R^{2}\compt(f)\mompt(f)\tgtv(\alpha)\otimes \nrmv(\alpha)+R^{2}\mompt(f)^{2}\tgtv(\alpha) \otimes \tgtv(\alpha)\biggr)\biggr). \notag
\end{align}
Therefore,
\begin{align*}
&\kerk(\alpha,\beta)  \fmat(\alpha+\beta)
\\
= &\biggl((1+2f(\alpha))^{\frac{1}{2}}\biggl(\frac{-1}{\lambda^{2}R^{2}\abs{2\sin(\beta/2)}^{2}\mathcal{Z}(f)}+\frac{K_{1}(\lambda R\abs{2\sin(\beta/2)}\mathcal{Z}(f)^{\frac{1}{2}})}{\lambda R\abs{2\sin(\beta/2)}\mathcal{Z}(f)^{\frac{1}{2}}} \\
&+K_{0}(\lambda R\abs{2\sin(\beta/2)}\mathcal{Z}(f)^{\frac{1}{2}})\biggr)\nrmv(\alpha) \cdot\fmat(\alpha+\beta) \\
    &+(1+2f(\alpha))^{\frac{1}{2}}\biggl(\frac{1}{R^{2}\mathcal{Z}(f)}\frac{2}{\lambda^{2}R^{2}\abs{2\sin(\beta/2)}^{2}\mathcal{Z}(f)}-\frac{K_{2}(\lambda R\abs{2\sin(\beta/2)}\mathcal{Z}(f)^{\frac{1}{2}})}{R^{2}\mathcal{Z}(f)}\biggr) \\
    &R^{2}\compt(f)^{2}\nrmv(\alpha)\cdot \fmat(\alpha+\beta) \\
    &+(1+2f(\alpha))^{\frac{1}{2}}\biggl(\frac{1}{R^{2}\mathcal{Z}(f)}\frac{2}{\lambda^{2}R^{2}\abs{2\sin(\beta/2)}^{2}\mathcal{Z}(f)}-\frac{K_{2}(\lambda R\abs{2\sin(\beta/2)}\mathcal{Z}(f)^{\frac{1}{2}})}{R^{2}\mathcal{Z}(f)}\biggr) \\
    &R^{2}\compt(f)\mompt(f)\tgtv(\alpha) \cdot \fmat(\alpha+\beta)\biggr).
\end{align*}

\section{The forcing term \texorpdfstring{$\gmat$}{gmat}}
Since
\begin{align*}
	&z_{\alpha}(\alpha) = R(1+2f(\alpha))^{-\frac12} f_{\alpha}(\alpha) \nrmv(\alpha)+R(1+2f(\alpha))^{\frac12}\tgtv(\alpha),
	\\
	&| z_{\alpha}(\alpha) |^2 = R^2 \frac{(f_\alpha(\alpha))^2 }{1+2f(\alpha)} +R^2(1+2f(\alpha)),
\end{align*}
we obtain
\begin{align}\label{gmatWF}
	\gmat(f) 
	=& \frac{\partial}{\partial\alpha}\left(\frac{z_{\alpha}(\alpha)}{\abs{z_{\alpha}(\alpha)}}\right)
	=\frac{\partial}{\partial\alpha}\biggl(\frac{f_{\alpha}(\alpha)\nrmv(\alpha)+(1+2f(\alpha))\tgtv(\alpha)}{\sqrt{f_{\alpha}(\alpha)^{2}+(1+2f(\alpha))^{2}}}\biggr).
\end{align}
Therefore,
\begin{align}
	&\gmat_0=\gmat(0)=-\nrmv(\alpha), \quad \gmat_L(f)=\frac{d}{d\epsilon}[\gmat(\epsilon f)]\mid_{\epsilon=0} =f_{\alpha\alpha}(\alpha) \nrmv(\alpha)+f_\alpha(\alpha) \tgtv(\alpha). \label{gmatLcalc}
\end{align}
Since \(\fmat=2A_{\gamma}\fmat^{\gamma}\), it follows that
\begin{align*}
	\fmat_{0}(\alpha)&\eqdef \fmat(0)=2A_{\gamma}\gmat_{0}=-2A_{\gamma}\nrmv(\alpha), \\
	\fmat_{L}(f) &\eqdef \frac{d}{d\epsilon}[\fmat(\epsilon f)]\mid_{\epsilon=0}=
	2 A_\gamma \gmat_{L}(f)=2A_{\gamma}(f_{\alpha\alpha} \nrmv(\alpha)+f_\alpha(\alpha) \tgtv(\alpha))=2A_{\gamma}\partial_{\alpha}(f_{\alpha}(\alpha)\nrmv(\alpha)).
\end{align*}

\section{The zeroth-order system}
Setting \(f=0\) in \eqref{kernelUSE}, we obtain
\begin{align}
	\kerk_{0}=&\nrmv(\alpha) \cdot \biggl(\biggl(\frac{-1}{\lambda^{2}R^{2}\abs{2\sin(\beta/2)}^{2}}+\frac{K_{1}(\lambda R\abs{2\sin(\beta/2)})}{\lambda R\abs{2\sin(\beta/2)}}+K_{0}(\lambda R\abs{2\sin(\beta/2)})\biggr)I \nonumber \\
	&+\biggl(\frac{2}{\lambda^{2}R^{2}\abs{2\sin(\beta/2)}^{2}}-K_{2}(\lambda R\abs{2\sin(\beta/2)})\biggr) \nonumber \\
	&\biggl(\sin^{2}(\beta/2)\nrmv(\alpha) \otimes \nrmv(\alpha)-\sin(\beta/2)\cos(\beta/2)\nrmv(\alpha) \otimes \tgtv(\alpha) \nonumber \\
	&-\sin(\beta/2)\cos(\beta/2)\tgtv(\alpha)\otimes \nrmv(\alpha)+\cos^{2}(\beta/2)\tgtv(\alpha) \otimes \tgtv(\alpha)\biggr)\biggr). \label{K0}
\end{align}
It follows that
\begin{align}
	&\int_{\mathbb{T}}\kerk_{0}\fmat_{0}(\alpha+\beta)d\beta \notag \\
	=&\int_{\mathbb{T}}\biggl(\frac{-1}{\lambda^{2}R^{2}\abs{2\sin(\beta/2)}^{2}}+\frac{K_{1}(\lambda R\abs{2\sin(\beta/2)})}{\lambda R\abs{2\sin(\beta/2)}}+K_{0}(\lambda R\abs{2\sin(\beta/2)})\biggr)\nrmv(\alpha) \cdot\fmat_{0}(\alpha+\beta)d\beta \label{int1} \\
	&+\int_{\mathbb{T}}\biggl(\frac{2}{\lambda^{2}R^{2}\abs{2\sin(\beta/2)}^{2}}-K_{2}(\lambda R\abs{2\sin(\beta/2)})\biggr) \sin^{2}(\beta/2)\nrmv(\alpha)\cdot\fmat_{0}(\alpha+\beta)d\beta \label{int2} \\
	&-\int_{\mathbb{T}}\biggl(\frac{2}{\lambda^{2}R^{2}\abs{2\sin(\beta/2)}^{2}}-K_{2}(\lambda R\abs{2\sin(\beta/2)})\biggr)\sin(\beta/2)\cos(\beta/2)\tgtv(\alpha)\cdot\fmat_{0}(\alpha+\beta)d\beta. \label{int3}
\end{align}
In fact, this entire integral vanishes because the combined integrands in \eqref{int1}--\eqref{int3} constitute a total derivative with respect to \(\beta\). Consequently, \(f=0\) is a stationary solution to the evolution equation \eqref{NEWequation}.

\section{Linearization of the Equation for \texorpdfstring{\(f\)}{f}}
The linearization of the evolution equation \eqref{NEWequation} about the steady state \(f=0\) is given by
\begin{align}
	f_{t}(\alpha,t) = &\frac{1}{2\pi R} \int_{\T}\kerk_{0}(\alpha,\beta) \fmat_{L}(\alpha +\beta) d\beta+\frac{1}{2\pi R} \int_{\T}\kerk_{L}(\alpha,\beta) \fmat_{0}(\alpha +\beta) d\beta-\frac{\dot{\bm{C}}(t)\cdot \nrmv(\alpha)}{R}. \label{linevolf}
\end{align}

\subsection{The first term} \label{firstterm}
Using \eqref{K0}, we obtain
\begin{align*}
	&\int_{\mathbb{T}}\kerk_{0}\fmat_{L}(\alpha+\beta)d\beta \\
	=&\int_{\mathbb{T}} \biggl(\biggl(\frac{-1}{\lambda^{2}R^{2}\abs{2\sin(\beta/2)}^{2}}+\frac{K_{1}(\lambda R\abs{2\sin(\beta/2)})}{\lambda R\abs{2\sin(\beta/2)}}+K_{0}(\lambda R\abs{2\sin(\beta/2)})\biggr)\nrmv(\alpha) \cdot\fmat_{L}(\alpha+\beta) \\
	&+\biggl(\frac{2}{\lambda^{2}R^{2}\abs{2\sin(\beta/2)}^{2}}-K_{2}(\lambda R\abs{2\sin(\beta/2)})\biggr) \\
	&\biggl(\sin^{2}(\beta/2)\nrmv(\alpha) \cdot\nrmv(\alpha) \otimes \nrmv(\alpha)\fmat_{L}(\alpha+\beta)-\sin(\beta/2)\cos(\beta/2)\nrmv(\alpha) \cdot\nrmv(\alpha) \otimes \tgtv(\alpha)\fmat_{L}(\alpha+\beta) \\
	&-\sin(\beta/2)\cos(\beta/2)\nrmv(\alpha) \cdot\tgtv(\alpha)\otimes \nrmv(\alpha)\fmat_{L}(\alpha+\beta)+\cos^{2}(\beta/2)\nrmv(\alpha) \cdot\tgtv(\alpha) \otimes \tgtv(\alpha)\fmat_{L}(\alpha+\beta)\biggr)\biggr)d\beta.
\end{align*}

\subsection{The second term} \label{secondterm}
From \eqref{kernelKexp}, we obtain
\begin{align*}
    \kerk_{L}=&\biggl[\nrmv(\alpha) \cdot \biggl(\biggl(\frac{-1}{\lambda^{2}R^{2}\abs{2\sin(\beta/2)}^{2}\mathcal{Z}(f)}+\frac{K_{1}(\lambda R\abs{2\sin(\beta/2)}\mathcal{Z}(f)^{\frac{1}{2}})}{\lambda R\abs{2\sin(\beta/2)}\mathcal{Z}(f)^{\frac{1}{2}}} \\
    &+K_{0}(\lambda R\abs{2\sin(\beta/2)}\mathcal{Z}(f)^{\frac{1}{2}})\biggr)I \\
    &+\biggl(\frac{1}{R^{2}\mathcal{Z}(f)}\frac{2}{\lambda^{2}R^{2}\abs{2\sin(\beta/2)}^{2}\mathcal{Z}(f)}-\frac{K_{2}(\lambda R\abs{2\sin(\beta/2)}\mathcal{Z}(f)^{\frac{1}{2}})}{R^{2}\mathcal{Z}(f)}\biggr) \\
    &\biggl(R^{2}\compt(f)^{2}\nrmv(\alpha) \otimes \nrmv(\alpha)+R^{2}\compt(f)\mompt(f)\nrmv(\alpha) \otimes \tgtv(\alpha) \\
    &+R^{2}\compt(f)\mompt(f)\tgtv(\alpha)\otimes \nrmv(\alpha)+R^{2}\mompt(f)^{2}\tgtv(\alpha) \otimes \tgtv(\alpha)\biggr)\biggr)\biggr]_{L} \\
    &+f(\alpha) \nrmv(\alpha) \cdot \biggl(\biggl(\frac{-1}{\lambda^{2}R^{2}\abs{2\sin(\beta/2)}^{2}}+\frac{K_{1}(\lambda R\abs{2\sin(\beta/2)})}{\lambda R\abs{2\sin(\beta/2)}}+K_{0}(\lambda R\abs{2\sin(\beta/2)})\biggr)I \\
    &+\biggl(\frac{2}{\lambda^{2}R^{2}\abs{2\sin(\beta/2)}^{2}}-K_{2}(\lambda R\abs{2\sin(\beta/2)})\biggr) \\
    &\biggl(\sin^{2}(\beta/2)\nrmv(\alpha) \otimes \nrmv(\alpha)-\sin(\beta/2)\cos(\beta/2)\nrmv(\alpha) \otimes \tgtv(\alpha) \\
    &-\sin(\beta/2)\cos(\beta/2)\tgtv(\alpha)\otimes \nrmv(\alpha)+\cos^{2}(\beta/2)\tgtv(\alpha) \otimes \tgtv(\alpha)\biggr)\biggr).
\end{align*}
To simplify the first term in this expression, we compute that
\begin{align*}
    &\biggl[\frac{-1}{\lambda^{2}R^{2}\abs{2\sin(\beta/2)}^{2}\mathcal{Z}(f)}+\frac{K_{1}(\lambda R\abs{2\sin(\beta/2)}\mathcal{Z}(f)^{\frac{1}{2}})}{\lambda R\abs{2\sin(\beta/2)}\mathcal{Z}(f)^{\frac{1}{2}}}+K_{0}(\lambda R\abs{2\sin(\beta/2)}\mathcal{Z}(f)^{\frac{1}{2}})\biggr]_{L} \\
    =&\frac{-1}{\lambda^{2}R^{2}\abs{2\sin(\beta/2)}^{2}}\biggl[\frac{1}{\mathcal{Z}(f)}\biggr]_{L}+\biggl([K_{1}(\lambda R\abs{2\sin(\beta/2)}\mathcal{Z}(f)^{\frac{1}{2}})]_{0}\biggl[\frac{1}{\lambda R\abs{2\sin(\beta/2)}}\frac{1}{\mathcal{Z}(f)^{1/2}}\biggr]_{L} \\
    &+[K_{1}(\lambda R\abs{2\sin(\beta/2)}\mathcal{Z}(f)^{\frac{1}{2}}]_{L}\biggl[\frac{1}{\lambda R\abs{2\sin(\beta/2)}}\frac{1}{\mathcal{Z}(f)^{1/2}}\biggr]_{0}\biggr)+\biggl[K_{0}(\lambda R\abs{2\sin(\beta/2)}\mathcal{Z}(f)^{\frac{1}{2}})\biggr]_{L} \\
    =&\frac{\mathcal{Z}_{L}}{\lambda^{2}R^{2}\abs{2\sin(\beta/2)}^{2}}+\biggl(-\frac{1}{2}\frac{K_{1}(\lambda R\abs{2\sin(\beta/2)})}{\lambda R\abs{2\sin(\beta/2)}}\mathcal{Z}_{L} \\
    &+\frac{1}{4}\biggl(-K_{0}(\lambda R \abs{2\sin(\beta/2)})-K_{2}(\lambda R \abs{2\sin(\beta/2)})\biggr)\mathcal{Z}_{L}\biggr) \\
    &-\frac{1}{2}K_{1}(\lambda R\abs{2\sin(\beta/2)})\lambda R \abs{2\sin(\beta/2)}\mathcal{Z}_{L}
\end{align*}
and
\begin{align*}
    &\biggl[\biggl(\frac{1}{R^{2}\mathcal{Z}(f)}\frac{2}{\lambda^{2}R^{2}\abs{2\sin(\beta/2)}^{2}\mathcal{Z}(f)}-\frac{K_{2}(\lambda R\abs{2\sin(\beta/2)}\mathcal{Z}(f)^{\frac{1}{2}})}{R^{2}\mathcal{Z}(f)}\biggr) \\
    &\biggl(R^{2}\compt(f)^{2}\nrmv(\alpha) \otimes \nrmv(\alpha)+R^{2}\compt(f)\mompt(f)\nrmv(\alpha) \otimes \tgtv(\alpha) \\
    &+R^{2}\compt(f)\mompt(f)\tgtv(\alpha)\otimes \nrmv(\alpha)+R^{2}\mompt(f)^{2}\tgtv(\alpha) \otimes \tgtv(\alpha)\biggr)\biggr]_{L} \\
    =&\biggl(\compt_{0}^{2}\biggl[\frac{2}{\lambda^{2}R^{2}\abs{2\sin(\beta/2)}^{2}\mathcal{Z}(f)^{2}}-\frac{K_{2}(\lambda R\abs{2\sin(\beta/2)}\mathcal{Z}(f)^{\frac{1}{2}})}{\mathcal{Z}(f)}\biggr]_{L}
\end{align*}
\begin{align*}
    &+2\compt_{0}\compt_{L}(f)\biggl[\frac{2}{\lambda^{2}R^{2}\abs{2\sin(\beta/2)}^{2}\mathcal{Z}(f)^{2}}-\frac{K_{2}(\lambda R\abs{2\sin(\beta/2)}\mathcal{Z}(f)^{\frac{1}{2}})}{\mathcal{Z}(f)}\biggr]_{0}\biggr)\nrmv(\alpha) \otimes \nrmv(\alpha) \\
    &+\biggl(\compt_{0}\mompt_{0}\biggl[\frac{2}{\lambda^{2}R^{2}\abs{2\sin(\beta/2)}^{2}\mathcal{Z}(f)^{2}}-\frac{K_{2}(\lambda R\abs{2\sin(\beta/2)}\mathcal{Z}(f)^{\frac{1}{2}})}{\mathcal{Z}(f)}\biggr]_{L} \nonumber \\
    &+(\compt_{0}\mompt_{L}(f)+\compt_{L}(f)\mompt_{0})\biggl[\frac{2}{\lambda^{2}R^{2}\abs{2\sin(\beta/2)}^{2}\mathcal{Z}(f)^{2}}-\frac{K_{2}(\lambda R\abs{2\sin(\beta/2)}\mathcal{Z}(f)^{\frac{1}{2}})}{\mathcal{Z}(f)}\biggr]_{0}\biggr)\nrmv(\alpha) \otimes \tgtv(\alpha) \nonumber \\
    &+\biggl(\compt_{0}\mompt_{0}\biggl[\frac{2}{\lambda^{2}R^{2}\abs{2\sin(\beta/2)}^{2}\mathcal{Z}(f)^{2}}-\frac{K_{2}(\lambda R\abs{2\sin(\beta/2)}\mathcal{Z}(f)^{\frac{1}{2}})}{\mathcal{Z}(f)}\biggr]_{L} \nonumber \\
    &+(\compt_{0}\mompt_{L}(f)+\compt_{L}(f)\mompt_{0})\biggl[\frac{2}{\lambda^{2}R^{2}\abs{2\sin(\beta/2)}^{2}\mathcal{Z}(f)^{2}}-\frac{K_{2}(\lambda R\abs{2\sin(\beta/2)}\mathcal{Z}(f)^{\frac{1}{2}})}{\mathcal{Z}(f)}\biggr]_{0}\biggr)\tgtv(\alpha)\otimes \nrmv(\alpha) \nonumber \\
    &+\biggl(\mompt_{0}^{2}\biggl[\frac{2}{\lambda^{2}R^{2}\abs{2\sin(\beta/2)}^{2}\mathcal{Z}(f)^{2}}-\frac{K_{2}(\lambda R\abs{2\sin(\beta/2)}\mathcal{Z}(f)^{\frac{1}{2}})}{\mathcal{Z}(f)}\biggr]_{L} \nonumber \\
    &+2\mompt_{0}\mompt_{L}(f)\biggl[\frac{2}{\lambda^{2}R^{2}\abs{2\sin(\beta/2)}^{2}\mathcal{Z}(f)^{2}}-\frac{K_{2}(\lambda R\abs{2\sin(\beta/2)}\mathcal{Z}(f)^{\frac{1}{2}})}{\mathcal{Z}(f)}\biggr]_{0}\biggr)\tgtv(\alpha) \otimes \tgtv(\alpha) \nonumber \\
    =&\biggl(\sin^{2}(\beta/2)\biggl(\frac{-4\mathcal{Z}_{L}}{\lambda^{2} R^{2}\abs{2\sin(\beta/2)}^{2}}+\mathcal{Z}_{L}K_{2}(\lambda R\abs{2\sin(\beta/2)}) \nonumber \\
    &-\frac{1}{4}\biggl(-K_{1}(\lambda R \abs{2\sin(\beta/2)})-K_{3}(\lambda R \abs{2\sin(\beta/2)})\biggr)\lambda R \abs{2\sin(\beta/2)}\mathcal{Z}_{L}\biggr) \nonumber \\
    &-2\sin(\beta/2)(\cos(\beta)\DBL f(\alpha)-\sin(\beta/2)f(\alpha))\biggl(\frac{2}{\lambda^{2}R^{2}\abs{2\sin(\beta/2)}^{2}}-K_{2}(\lambda R\abs{2\sin(\beta/2)})\biggr)\biggr) \nonumber \\
    &\nrmv(\alpha) \otimes \nrmv(\alpha) \nonumber \\
    &+\biggl(-\sin(\beta/2)\cos(\beta/2)\biggl(\frac{-4\mathcal{Z}_{L}}{\lambda^{2} R^{2}\abs{2\sin(\beta/2)}^{2}}+\mathcal{Z}_{L}K_{2}(\lambda R \abs{2\sin(\beta/2)}) \nonumber \\
    &-\frac{1}{4}\biggl(-K_{1}(\lambda R \abs{2\sin(\beta/2)})-K_{3}(\lambda R \abs{2\sin(\beta/2)})\biggr)\lambda R \abs{2\sin(\beta/2)}\mathcal{Z}_{L}\biggr) \nonumber \\
    &+\biggl(-\sin(\beta/2)\cos(\beta/2)f(\alpha+\beta)+\cos(\beta/2)\biggl(\cos(\beta)\DBL f(\alpha)-\sin(\beta/2)f(\alpha)\biggr)\biggr) \nonumber \\
    &\biggl(\frac{2}{\lambda^{2}R^{2}\abs{2\sin(\beta/2)}^{2}}-K_{2}(\lambda R\abs{2\sin(\beta/2)})\biggr)\biggr)\nrmv(\alpha) \otimes \tgtv(\alpha) \nonumber \\
    &+\biggl(-\sin(\beta/2)\cos(\beta/2)\biggl(\frac{-4\mathcal{Z}_{L}}{\lambda^{2} R^{2}\abs{2\sin(\beta/2)}^{2}}+\mathcal{Z}_{L}K_{2}(\lambda R\abs{2\sin(\beta/2)}) \nonumber \\
    &-\frac{1}{4}\biggl(-K_{1}(\lambda R \abs{2\sin(\beta/2)})-K_{3}(\lambda R \abs{2\sin(\beta/2)})\biggr)\lambda R \abs{2\sin(\beta/2)}\mathcal{Z}_{L}\biggr) \nonumber \\
    &+\biggl(-\sin(\beta/2)\cos(\beta/2)f(\alpha+\beta)+\cos(\beta/2)\biggl(\cos(\beta)\DBL f(\alpha)-\sin(\beta/2)f(\alpha)\biggr)\biggr) \nonumber \\
    &\biggl(\frac{2}{\lambda^{2}R^{2}\abs{2\sin(\beta/2)}^{2}}-K_{2}(\lambda R\abs{2\sin(\beta/2)})\biggr)\biggr)\tgtv(\alpha)\otimes \nrmv(\alpha) \nonumber \\
    &+\biggl(\cos^{2}(\beta/2)\biggl(\frac{-4\mathcal{Z}_{L}}{\lambda^{2} R^{2}\abs{2\sin(\beta/2)}^{2}}+\mathcal{Z}_{L}K_{2}(\lambda R\abs{2\sin(\beta/2)}) \\
    &-\frac{1}{4}\biggl(-K_{1}(\lambda R \abs{2\sin(\beta/2)})-K_{3}(\lambda R \abs{2\sin(\beta/2)})\biggr)\lambda R \abs{2\sin(\beta/2)}\mathcal{Z}_{L}\biggr) \nonumber \\
    &+2\cos^{2}(\beta/2)f(\alpha+\beta)\biggl(\frac{2}{\lambda^{2}R^{2}\abs{2\sin(\beta/2)}^{2}}-K_{2}(\lambda R\abs{2\sin(\beta/2)})\biggr)\biggr)\tgtv(\alpha) \otimes \tgtv(\alpha). \nonumber
\end{align*}
Plugging these expressions into \(\kerk_{L}\), we obtain
\begin{align*}
    \kerk_{L}=&\biggl(\frac{1}{\lambda^{2}R^{2}\abs{2\sin(\beta/2)}^{2}}-\frac{1}{2}\frac{K_{1}(\lambda R\abs{2\sin(\beta/2)})}{\lambda R\abs{2\sin(\beta/2)}} \\
    &+\frac{1}{4}\biggl(-K_{0}(\lambda R \abs{2\sin(\beta/2)})-K_{2}(\lambda R \abs{2\sin(\beta/2)})\biggr) \\
    &-\frac{1}{2}K_{1}(\lambda R\abs{2\sin(\beta/2)})\lambda R \abs{2\sin(\beta/2)}\biggr)(f(\alpha)+f(\alpha+\beta))\nrmv(\alpha)\cdot I \\
    &+\biggl(\sin^{2}(\beta/2)\biggl(\frac{-4}{\lambda^{2} R^{2}\abs{2\sin(\beta/2)}^{2}}+K_{2}(\lambda R \abs{2\sin(\beta/2)}) \\
    &-\frac{1}{4}\biggl(-K_{1}(\lambda R \abs{2\sin(\beta/2)})-K_{3}(\lambda R \abs{2\sin(\beta/2)})\biggr)\lambda R \abs{2\sin(\beta/2)}\biggr)(f(\alpha)+f(\alpha+\beta)) \\
    &-2\sin(\beta/2)\biggl(\cos(\beta)\DBL f(\alpha)-\sin(\beta/2)f(\alpha)\biggr) \\
    &\biggl(\frac{2}{\lambda^{2}R^{2}\abs{2\sin(\beta/2)}^{2}}-K_{2}(\lambda R\abs{2\sin(\beta/2)})\biggr)\biggr)\nrmv(\alpha)\cdot\nrmv(\alpha) \otimes \nrmv(\alpha) \\
    &+\biggl(-\sin(\beta/2)\cos(\beta/2)\biggl(\frac{-4}{\lambda^{2} R^{2}\abs{2\sin(\beta/2)}^{2}}+K_{2}(\lambda R \abs{2\sin(\beta/2)}) \\
    &-\frac{1}{4}\biggl(-K_{1}(\lambda R \abs{2\sin(\beta/2)})-K_{3}(\lambda R \abs{2\sin(\beta/2)})\biggr)\lambda R \abs{2\sin(\beta/2)}\biggr)(f(\alpha)+f(\alpha+\beta)) \\
    &+\biggl(-\sin(\beta/2)\cos(\beta/2)f(\alpha+\beta)+\cos(\beta/2)\biggl(\cos(\beta)\DBL f(\alpha)-\sin(\beta/2)f(\alpha)\biggr)\biggr) \\
    &\biggl(\frac{2}{\lambda^{2}R^{2}\abs{2\sin(\beta/2)}^{2}}-K_{2}(\lambda R\abs{2\sin(\beta/2)})\biggr)\biggr)\nrmv(\alpha)\cdot\nrmv(\alpha) \otimes \tgtv(\alpha) \\
    &+\biggl(-\sin(\beta/2)\cos(\beta/2)\biggl(\frac{-4}{\lambda^{2} R^{2}\abs{2\sin(\beta/2)}^{2}}+K_{2}(\lambda R\abs{2\sin(\beta/2)}) \\
    &-\frac{1}{4}\biggl(-K_{1}(\lambda R \abs{2\sin(\beta/2)})-K_{3}(\lambda R \abs{2\sin(\beta/2)})\biggr)\lambda R \abs{2\sin(\beta/2)}\biggr)(f(\alpha)+f(\alpha+\beta)) \\
    &+\biggl(-\sin(\beta/2)\cos(\beta/2)f(\alpha+\beta)+\cos(\beta/2)\biggl(\cos(\beta)\DBL f(\alpha)-\sin(\beta/2)f(\alpha)\biggr)\biggr) \\
    &\biggl(\frac{2}{\lambda^{2}R^{2}\abs{2\sin(\beta/2)}^{2}}-K_{2}(\lambda R\abs{2\sin(\beta/2)})\biggr)\biggr)\nrmv(\alpha)\cdot\tgtv(\alpha)\otimes \nrmv(\alpha) \\
    &+\biggl(\cos^{2}(\beta/2)\biggl(\frac{-4}{\lambda^{2} R^{2}\abs{2\sin(\beta/2)}^{2}}+K_{2}(\lambda R\abs{2\sin(\beta/2)}) \\
    &-\frac{1}{4}\biggl(-K_{1}(\lambda R \abs{2\sin(\beta/2)})-K_{3}(\lambda R \abs{2\sin(\beta/2)})\biggr)\lambda R \abs{2\sin(\beta/2)}\biggr)(f(\alpha)+f(\alpha+\beta)) \\
    &+2\cos^{2}(\beta/2)f(\alpha+\beta)\biggl(\frac{2}{\lambda^{2}R^{2}\abs{2\sin(\beta/2)}^{2}}-K_{2}(\lambda R\abs{2\sin(\beta/2)})\biggr)\biggr)\nrmv(\alpha)\cdot\tgtv(\alpha) \otimes \tgtv(\alpha) \\
    &+f(\alpha)\biggl(\frac{-1}{\lambda^{2}R^{2}\abs{2\sin(\beta/2)}^{2}}+\frac{K_{1}(\lambda R\abs{2\sin(\beta/2)})}{\lambda R\abs{2\sin(\beta/2)}}+K_{0}(\lambda R\abs{2\sin(\beta/2)})\biggr)\nrmv(\alpha)\cdot I \\
    &+f(\alpha)\biggl(\frac{2}{\lambda^{2}R^{2}\abs{2\sin(\beta/2)}^{2}}-K_{2}(\lambda R\abs{2\sin(\beta/2)})\biggr) \\
    &\biggl(\sin^{2}(\beta/2)\nrmv(\alpha)\cdot\nrmv(\alpha) \otimes \nrmv(\alpha)-\sin(\beta/2)\cos(\beta/2)\nrmv(\alpha)\cdot\nrmv(\alpha) \otimes \tgtv(\alpha) \\
    &-\sin(\beta/2)\cos(\beta/2)\nrmv(\alpha)\cdot\tgtv(\alpha)\otimes \nrmv(\alpha)+\cos^{2}(\beta/2)\nrmv(\alpha)\cdot\tgtv(\alpha) \otimes \tgtv(\alpha)\biggr).
\end{align*}
It follows that
\begin{align*}
    &\int_{\mathbb{T}}\kerk_{L}\fmat_{0}(\alpha+\beta)d\beta=-2A_{\gamma}\int_{\mathbb{T}}\kerk_{L}\nrmv(\alpha+\beta)d\beta \\
    =&-2A_{\gamma}\biggl(f(\alpha)\int_{\mathbb{T}}\biggl(\frac{1}{\lambda^{2}R^{2}\abs{2\sin(\beta/2)}^{2}}-\frac{1}{2}\frac{K_{1}(\lambda R\abs{2\sin(\beta/2)})}{\lambda R\abs{2\sin(\beta/2)}} -\frac{1}{4}K_{0}(\lambda R \abs{2\sin(\beta/2)}) \\
    &-\frac{1}{4}K_{2}(\lambda R \abs{2\sin(\beta/2)})-\frac{\lambda R \abs{2\sin(\beta/2)}}{2}K_{1}(\lambda R\abs{2\sin(\beta/2)})\biggr)\cos\beta d\beta \\
    &+\int_{\mathbb{T}}\biggl(\frac{1}{\lambda^{2}R^{2}\abs{2\sin(\beta/2)}^{2}}-\frac{1}{2}\frac{K_{1}(\lambda R\abs{2\sin(\beta/2)})}{\lambda R\abs{2\sin(\beta/2)}} -\frac{1}{4}K_{0}(\lambda R \abs{2\sin(\beta/2)})-\frac{1}{4}K_{2}(\lambda R \abs{2\sin(\beta/2)}) \\
    &-\frac{\lambda R \abs{2\sin(\beta/2)}}{2}K_{1}(\lambda R\abs{2\sin(\beta/2)})\biggr)f(\alpha+\beta)\cos\beta d\beta \\
    &+f(\alpha)\int_{\mathbb{T}}\sin^{2}(\beta/2)\cos\beta\biggl(\frac{-4}{\lambda^{2} R^{2}\abs{2\sin(\beta/2)}^{2}}+K_{2}(\lambda R\abs{2\sin(\beta/2)}) \\
    &+\frac{\lambda R \abs{2\sin(\beta/2)}}{4}\biggl(K_{1}(\lambda R \abs{2\sin(\beta/2)})+K_{3}(\lambda R \abs{2\sin(\beta/2)})\biggr)\biggr) d\beta \\
    &+\int_{\mathbb{T}}\sin^{2}(\beta/2)\cos\beta\biggl(\frac{-4}{\lambda^{2} R^{2}\abs{2\sin(\beta/2)}^{2}}+K_{2}(\lambda R\abs{2\sin(\beta/2)}) \\
    &+\frac{\lambda R \abs{2\sin(\beta/2)}}{4}\biggl(K_{1}(\lambda R \abs{2\sin(\beta/2)})+K_{3}(\lambda R \abs{2\sin(\beta/2)})\biggr)\biggr)f(\alpha+\beta)d\beta \\
    &+\int_{\mathbb{T}} -\sin^{2}(\beta/2)(2+\cos\beta)f(\alpha+\beta)\biggl(\frac{2}{\lambda^{2}R^{2}\abs{2\sin(\beta/2)}^{2}}-K_{2}(\lambda R\abs{2\sin(\beta/2)})\biggr)d\beta \\
    &-\frac{1}{2}f(\alpha)\int_{\mathbb{T}}\sin^{2}\beta\biggl(\frac{-4}{\lambda^{2} R^{2}\abs{2\sin(\beta/2)}^{2}}+K_{2}(\lambda R\abs{2\sin(\beta/2)}) \\
    &+\frac{\lambda R \abs{2\sin(\beta/2)}}{4}\biggl(K_{1}(\lambda R \abs{2\sin(\beta/2)})+K_{3}(\lambda R \abs{2\sin(\beta/2)})\biggr)\biggr) d\beta \\
    &-\frac{1}{2}\int_{\mathbb{T}}\sin^{2}\beta\biggl(\frac{-4}{\lambda^{2} R^{2}\abs{2\sin(\beta/2)}^{2}}+K_{2}(\lambda R\abs{2\sin(\beta/2)}) \\
    &+\frac{\lambda R \abs{2\sin(\beta/2)}}{4}\biggl(K_{1}(\lambda R \abs{2\sin(\beta/2)})+K_{3}(\lambda R \abs{2\sin(\beta/2)})\biggr)\biggr) f(\alpha+\beta)d\beta \\
    &+f(\alpha)\int_{\mathbb{T}}\biggl(\frac{-1}{\lambda^{2}R^{2}\abs{2\sin(\beta/2)}^{2}}+\frac{K_{1}(\lambda R\abs{2\sin(\beta/2)})}{\lambda R\abs{2\sin(\beta/2)}}+K_{0}(\lambda R\abs{2\sin(\beta/2)})\biggr)\cos\beta d\beta \\
    &-f(\alpha)\int_{\mathbb{T}}(1-\cos\beta)\biggl(\frac{2}{\lambda^{2}R^{2}\abs{2\sin(\beta/2)}^{2}}-K_{2}(\lambda R\abs{2\sin(\beta/2)})\biggr) d\beta\biggr).
\end{align*}

\subsection{Further simplification}
We substitute the expressions derived in Sections \ref{firstterm} and \ref{secondterm} into \eqref{linevolf} and add
\begin{align*}
	0=-\frac{1}{2\pi R}f(\alpha)\int_{\mathbb{T}}\kerk_{0}\fmat_{0}(\alpha+\beta)d\beta
\end{align*}
to the right hand side of \eqref{linevolf} to obtain
\begin{align}
    &2\pi R f_{t}(\alpha,t) \label{feqnlin} \\
    =&\int_{\mathbb{T}} \biggl(\biggl(\frac{-1}{\lambda^{2}R^{2}\abs{2\sin(\beta/2)}^{2}}+\frac{K_{1}(\lambda R\abs{2\sin(\beta/2)})}{\lambda R\abs{2\sin(\beta/2)}}+K_{0}(\lambda R\abs{2\sin(\beta/2)})\biggr)\nrmv(\alpha) \cdot\fmat_{L}(\alpha+\beta)
\end{align}
\begin{align}
    &+\biggl(\frac{2}{\lambda^{2}R^{2}\abs{2\sin(\beta/2)}^{2}}-K_{2}(\lambda R\abs{2\sin(\beta/2)})\biggr) \nonumber \\
    &\biggl(\sin^{2}(\beta/2)\nrmv(\alpha) \cdot\nrmv(\alpha) \otimes \nrmv(\alpha)\fmat_{L}(\alpha+\beta)-\sin(\beta/2)\cos(\beta/2)\nrmv(\alpha) \cdot\nrmv(\alpha) \otimes \tgtv(\alpha)\fmat_{L}(\alpha+\beta) \nonumber \\
    &-\sin(\beta/2)\cos(\beta/2)\nrmv(\alpha) \cdot\tgtv(\alpha)\otimes \nrmv(\alpha)\fmat_{L}(\alpha+\beta) \\
    &+\cos^{2}(\beta/2)\nrmv(\alpha) \cdot\tgtv(\alpha) \otimes \tgtv(\alpha)\fmat_{L}(\alpha+\beta)\biggr)\biggr)d\beta \nonumber \\
    &-2A_{\gamma}\biggl(\int_{\mathbb{T}}\biggl(\frac{1}{\lambda^{2}R^{2}\abs{2\sin(\beta/2)}^{2}}-\frac{1}{2}\frac{K_{1}(\lambda R\abs{2\sin(\beta/2)})}{\lambda R\abs{2\sin(\beta/2)}} -\frac{1}{4}K_{0}(\lambda R \abs{2\sin(\beta/2)}) \\
    &-\frac{1}{4}K_{2}(\lambda R \abs{2\sin(\beta/2)})-\frac{\lambda R \abs{2\sin(\beta/2)}}{2}K_{1}(\lambda R\abs{2\sin(\beta/2)})\biggr)\cos\beta f(\alpha)d\beta \nonumber \\
    &+\int_{\mathbb{T}}\biggl(\frac{1}{\lambda^{2}R^{2}\abs{2\sin(\beta/2)}^{2}}-\frac{1}{2}\frac{K_{1}(\lambda R\abs{2\sin(\beta/2)})}{\lambda R\abs{2\sin(\beta/2)}} -\frac{1}{4}K_{0}(\lambda R \abs{2\sin(\beta/2)}) \\
    &-\frac{1}{4}K_{2}(\lambda R \abs{2\sin(\beta/2)})-\frac{\lambda R \abs{2\sin(\beta/2)}}{2}K_{1}(\lambda R\abs{2\sin(\beta/2)})\biggr)\cos\beta f(\alpha+\beta) d\beta \nonumber \\
    &+\int_{\mathbb{T}}-\sin^{2}(\beta/2)\biggl(\frac{-4}{\lambda^{2} R^{2}\abs{2\sin(\beta/2)}^{2}}+K_{2}(\lambda R\abs{2\sin(\beta/2)}) \nonumber \\
    &+\frac{\lambda R \abs{2\sin(\beta/2)}}{4}\biggl(K_{1}(\lambda R \abs{2\sin(\beta/2)})+K_{3}(\lambda R \abs{2\sin(\beta/2)})\biggr)\biggr)f(\alpha) d\beta \nonumber \\
    &+\int_{\mathbb{T}}-\sin^{2}(\beta/2)\biggl(\frac{-4}{\lambda^{2} R^{2}\abs{2\sin(\beta/2)}^{2}}+K_{2}(\lambda R\abs{2\sin(\beta/2)}) \nonumber \\
    &+\frac{\lambda R \abs{2\sin(\beta/2)}}{4}\biggl(K_{1}(\lambda R \abs{2\sin(\beta/2)})+K_{3}(\lambda R \abs{2\sin(\beta/2)})\biggr)\biggr) f(\alpha+\beta)d\beta \nonumber \\
    &+\int_{\mathbb{T}} -\sin^{2}(\beta/2)(2+\cos\beta)\biggl(\frac{2}{\lambda^{2}R^{2}\abs{2\sin(\beta/2)}^{2}}-K_{2}(\lambda R\abs{2\sin(\beta/2)})\biggr)f(\alpha+\beta)d\beta \nonumber \\
    &+\int_{\mathbb{T}}-\sin^{2}(\beta/2)\biggl(\frac{2}{\lambda^{2}R^{2}\abs{2\sin(\beta/2)}^{2}}-K_{2}(\lambda R\abs{2\sin(\beta/2)})\biggr)f(\alpha)d\beta\biggr)-2\pi\dot{\bm{C}}(t)\cdot \nrmv(\alpha). \nonumber 
\end{align}
We dedicate the rest of this section to simplifying this expression. To begin, we draw our attention to the first term in \eqref{feqnlin}:
\begin{align*}
    \int_{\mathbb{T}}\biggl(\frac{-1}{\lambda^{2}R^{2}\abs{2\sin(\beta/2)}^{2}}+\frac{K_{1}(\lambda R\abs{2\sin(\beta/2)})}{\lambda R\abs{2\sin(\beta/2)}}+K_{0}(\lambda R\abs{2\sin(\beta/2)})\biggr)\nrmv(\alpha) \cdot\fmat_{L}(\alpha+\beta)d\beta.
\end{align*}
For \(x=\lambda R\abs{2\sin(\beta/2)}\) sufficiently small,
\begin{align*}
    &\frac{-1}{\lambda^{2}R^{2}\abs{2\sin(\beta/2)}^{2}}+\frac{K_{1}(\lambda R\abs{2\sin(\beta/2)})}{\lambda R\abs{2\sin(\beta/2)}}+K_{0}(\lambda R\abs{2\sin(\beta/2)}) \\
    =&\frac{1}{4}(-1-2\gamma_{E}+2\log 2-2\log x)+\frac{1}{64}(11-12\gamma_{E}+12\log 2-12\log x)x^{2}+\mathcal{O}(x^{4}\abs{\log x}),
\end{align*}
where \(\gamma_{E}\) is the Euler-Mascheroni constant. Isolating the contribution from the \(-\frac{1}{2}\log x\) term in the expansion, we see that
\begin{align*}
    &-\int_{\mathbb{T}}\frac{1}{2}\log(\lambda R \abs{2\sin(\beta/2)})\nrmv(\alpha)\cdot \fmat_{L}(\alpha+\beta)d\beta=\frac{A_{\gamma}}{2}\biggl(-2\pi\mathcal{H}(f_{\alpha})(\alpha)+\int_{\T}\cos\beta f(\alpha+\beta)d\beta\biggr).
\end{align*}
In view of this observation, we write
\begin{align*}
	&\int_{\mathbb{T}}\biggl(\frac{-1}{\lambda^{2}R^{2}\abs{2\sin(\beta/2)}^{2}}+\frac{K_{1}(\lambda R\abs{2\sin(\beta/2)})}{\lambda R\abs{2\sin(\beta/2)}}+K_{0}(\lambda R\abs{2\sin(\beta/2)})\biggr)\nrmv(\alpha)\cdot\fmat_{L}(\alpha+\beta)d\beta \\
	=&-\int_{\mathbb{T}}\partial_{\beta}\biggl[\frac{-1}{\lambda^{2}R^{2}\abs{2\sin(\beta/2)}^{2}}+\frac{K_{1}(\lambda R\abs{2\sin(\beta/2)})}{\lambda R\abs{2\sin(\beta/2)}}+K_{0}(\lambda R\abs{2\sin(\beta/2)})\biggr] \\
    &2A_{\gamma}f_{\beta}(\alpha+\beta)\cos\beta d\beta \\
	=&-A_{\gamma}\pi\mathcal{H}(f_{\alpha})(\alpha) \\
	&-\int_{\mathbb{T}}\biggl(\frac{A_{\gamma}}{2\tan(\beta/2)}+2A_{\gamma}\cos\beta\biggl[\frac{\lambda^{2}R^{2}4\sin(\beta/2)\cos(\beta/2)}{(\lambda^{2}R^{2}4\sin^{2}(\beta/2))^{2}} \\
	&-\frac{1}{2}(K_{0}(\lambda R(4\sin^{2}(\beta/2))^{1/2})+K_{2}(\lambda R(4\sin^{2}(\beta/2))^{1/2})) \\
    &\frac{2\lambda R}{(4\sin^{2}(\beta/2))^{1/2}}\frac{\sin(\beta/2)\cos(\beta/2)}{\lambda R(4\sin^{2}(\beta/2))^{1/2}} \\
	&-\frac{K_{1}(\lambda R(4\sin^{2}(\beta/2))^{1/2})}{(\lambda R(4\sin^{2}(\beta/2))^{1/2})^{2}}\frac{2\lambda R}{(4\sin^{2}(\beta/2))^{1/2}}\sin(\beta/2)\cos(\beta/2) \\
	&-K_{1}(\lambda R(4\sin^{2}(\beta/2))^{1/2})\frac{2\lambda R}{(4\sin^{2}(\beta/2))^{1/2}}\sin(\beta/2)\cos(\beta/2)\biggr]\biggr)f_{\beta}(\alpha+\beta)d\beta.
\end{align*}
Next, we can simplify the second term in \eqref{feqnlin} to
\begin{align*}
	&\int_{\mathbb{T}}\biggl(\frac{2}{\lambda^{2}R^{2}\abs{2\sin(\beta/2)}^{2}}-K_{2}(\lambda R\abs{2\sin(\beta/2)})\biggr) \\
	&(\sin^{2}(\beta/2)\nrmv(\alpha)\cdot\nrmv(\alpha)\otimes \nrmv(\alpha)\fmat_{L}(\alpha+\beta) \\
    &-\sin(\beta/2)\cos(\beta/2)\nrmv(\alpha)\cdot\nrmv(\alpha)\otimes \tgtv(\alpha)\fmat_{L}(\alpha+\beta) \\
	&-\sin(\beta/2)\cos(\beta/2)\nrmv(\alpha)\cdot\tgtv(\alpha)\otimes \nrmv(\alpha)\fmat_{L}(\alpha+\beta) \\
    &+\cos^{2}(\beta/2)\nrmv(\alpha)\cdot\tgtv(\alpha)\otimes\tgtv(\alpha)\fmat_{L}(\alpha+\beta))d\beta \\
	=&-\int_{\mathbb{T}}\partial_{\beta}\biggl[\biggl(\frac{2}{\lambda^{2}R^{2}\abs{2\sin(\beta/2)}^{2}}-K_{2}(\lambda R\abs{2\sin(\beta/2)})\biggr)\sin^{2}(\beta/2)\\
    &(\cos^{3}\alpha+\cos\alpha\sin^{2}\alpha)2A_{\gamma}\biggr]f_{\beta}(\alpha+\beta)\cos(\alpha+\beta)d\beta \\
	&-\int_{\mathbb{T}}\partial_{\beta}\biggl[\biggl(\frac{2}{\lambda^{2}R^{2}\abs{2\sin(\beta/2)}^{2}}-K_{2}(\lambda R\abs{2\sin(\beta/2)})\biggr)\sin^{2}(\beta/2) \\
    &(\cos^{2}\alpha\sin\alpha+\sin^{3}\alpha)2A_{\gamma}\biggr]f_{\beta}(\alpha+\beta)\sin(\alpha+\beta)d\beta \\
	&-\int_{\mathbb{T}}\partial_{\beta}\biggl[\biggl(\frac{2}{\lambda^{2}R^{2}\abs{2\sin(\beta/2)}^{2}}-K_{2}(\lambda R\abs{2\sin(\beta/2)})\biggr)\sin(\beta/2)\cos(\beta/2) \\
    &(\cos^{2}\alpha\sin\alpha+\sin^{3}\alpha)2A_{\gamma}\biggr]f_{\beta}(\alpha+\beta)\cos(\alpha+\beta)d\beta \\
	&+\int_{\mathbb{T}}\partial_{\beta}\biggl[\biggl(\frac{2}{\lambda^{2}R^{2}\abs{2\sin(\beta/2)}^{2}}-K_{2}(\lambda R\abs{2\sin(\beta/2)})\biggr)\sin(\beta/2)\cos(\beta/2)\\
    &(\cos^{3}\alpha+\cos\alpha\sin^{2}\alpha)2A_{\gamma}\biggr]f_{\beta}(\alpha+\beta)\sin(\alpha+\beta)d\beta.
\end{align*}
Substituting these expressions into \eqref{feqnlin} and taking the Fourier transform, we find that for each \(k \in \mathbb{Z}\), the \(k\)th Fourier mode \(\mathcal{F}(f)(k)\) of \(f\) satisfies the evolution equation
\begin{align*}
	&2\pi R\partial_{t}\mathcal{F}(f)(k,t) \\
	=&-4A_{\gamma}\int_{\mathbb{T}}\partial_{\beta}\biggl[\biggl(\frac{2}{\lambda^{2}R^{2}\abs{2\sin(\beta/2)}^{2}}-K_{2}(\lambda R\abs{2\sin(\beta/2)})\biggr)\sin^{2}(\beta/2)\biggr] \nonumber \\
    &\frac{1}{4}(e^{-i(-k+1)\beta}+e^{-i(-k-1)\beta})d\beta \mathcal{F}(f_{\beta})(k) \nonumber \\
	&-4A_{\gamma}\int_{\mathbb{T}}\partial_{\beta}\biggl[\biggl(\frac{2}{\lambda^{2}R^{2}\abs{2\sin(\beta/2)}^{2}}-K_{2}(\lambda R\abs{2\sin(\beta/2)})\biggr)\sin(\beta/2)\cos(\beta/2)\biggr] \nonumber \\
    &\frac{1}{4i}(e^{i(k-1)\beta}-e^{i(k+1)\beta})d\beta \mathcal{F}(f_{\beta})(k) \nonumber \\
	&-A_{\gamma}\int_{\mathbb{T}}\biggl(\frac{1}{2\tan(\beta/2)}+2\cos\beta\biggl[\frac{\lambda^{2}R^{2}4\sin(\beta/2)\cos(\beta/2)}{(\lambda^{2}R^{2}4\sin^{2}(\beta/2))^{2}} \nonumber \\
	&-\frac{1}{2}(K_{0}(\lambda R(4\sin^{2}(\beta/2))^{1/2})+K_{2}(\lambda R(4\sin^{2}(\beta/2))^{1/2})) \\
    &\frac{2\lambda R}{(4\sin^{2}(\beta/2))^{1/2}}\frac{\sin(\beta/2)\cos(\beta/2)}{\lambda R(4\sin^{2}(\beta/2))^{1/2}} \nonumber \\
	&-\frac{K_{1}(\lambda R(4\sin^{2}(\beta/2))^{1/2})}{(\lambda R(4\sin^{2}(\beta/2))^{1/2})^{2}}\frac{2\lambda R}{(4\sin^{2}(\beta/2))^{1/2}}\sin(\beta/2)\cos(\beta/2) \nonumber \\
	&-K_{1}(\lambda R(4\sin^{2}(\beta/2))^{1/2})\frac{2\lambda R}{(4\sin^{2}(\beta/2))^{1/2}}\sin(\beta/2)\cos(\beta/2)\biggr]\biggr)e^{ik\beta}d\beta\mathcal{F}(f_{\beta})(k) \nonumber \\
	&+\mathcal{F}(-A_{\gamma}\pi\mathcal{H}(f_{\alpha})(\alpha))(k) \nonumber \\
	&-2A_{\gamma}\biggl(\int_{\mathbb{T}}\biggl(\frac{1}{\lambda^{2}R^{2}\abs{2\sin(\beta/2)}^{2}}-\frac{1}{2}\frac{K_{1}(\lambda R\abs{2\sin(\beta/2)})}{\lambda R\abs{2\sin(\beta/2)}} -\frac{1}{4}K_{0}(\lambda R \abs{2\sin(\beta/2)}) \\
    &-\frac{1}{4}K_{2}(\lambda R \abs{2\sin(\beta/2)})-\frac{\lambda R \abs{2\sin(\beta/2)}}{2}K_{1}(\lambda R\abs{2\sin(\beta/2)})\biggr)\cos\beta d\beta \nonumber \\
	&+\int_{\mathbb{T}}\biggl(\frac{1}{\lambda^{2}R^{2}\abs{2\sin(\beta/2)}^{2}}-\frac{1}{2}\frac{K_{1}(\lambda R\abs{2\sin(\beta/2)})}{\lambda R\abs{2\sin(\beta/2)}} -\frac{1}{4}K_{0}(\lambda R \abs{2\sin(\beta/2)}) \\
    &-\frac{1}{4}K_{2}(\lambda R \abs{2\sin(\beta/2)})-\frac{\lambda R \abs{2\sin(\beta/2)}}{2}K_{1}(\lambda R\abs{2\sin(\beta/2)})\biggr)\cos\beta e^{ik\beta}d\beta \nonumber \\
	&+\int_{\mathbb{T}}-\sin^{2}(\beta/2)\biggl(\frac{-4}{\lambda^{2} R^{2}\abs{2\sin(\beta/2)}^{2}}+K_{2}(\lambda R\abs{2\sin(\beta/2)}) \nonumber \\
	&+\frac{\lambda R \abs{2\sin(\beta/2)}}{4}\biggl(K_{1}(\lambda R \abs{2\sin(\beta/2)})+K_{3}(\lambda R \abs{2\sin(\beta/2)})\biggr)\biggr) d\beta \nonumber \\
	&+\int_{\mathbb{T}}-\sin^{2}(\beta/2)\biggl(\frac{-4}{\lambda^{2} R^{2}\abs{2\sin(\beta/2)}^{2}}+K_{2}(\lambda R\abs{2\sin(\beta/2)}) \nonumber \\
	&+\frac{\lambda R \abs{2\sin(\beta/2)}}{4}\biggl(K_{1}(\lambda R \abs{2\sin(\beta/2)})+K_{3}(\lambda R \abs{2\sin(\beta/2)})\biggr)\biggr)e^{ik\beta}d\beta \nonumber \\
	&+\int_{\mathbb{T}} -\sin^{2}(\beta/2)(2+\cos\beta)\biggl(\frac{2}{\lambda^{2}R^{2}\abs{2\sin(\beta/2)}^{2}}-K_{2}(\lambda R\abs{2\sin(\beta/2)})\biggr)e^{ik\beta}d\beta \nonumber \\
	&+\int_{\mathbb{T}}-\sin^{2}(\beta/2)\biggl(\frac{2}{\lambda^{2}R^{2}\abs{2\sin(\beta/2)}^{2}}-K_{2}(\lambda R\abs{2\sin(\beta/2)})\biggr)d\beta\biggr)\mathcal{F}(f)(k) \\
	&-2\pi\biggl(\dot{C}_{1}(t)\frac{1}{2}(\delta_{1k}+\delta_{-1k})+\dot{C}_{2}(t)\frac{1}{2i}(\delta_{1k}-\delta_{-1k})\biggr), \nonumber
\end{align*}
where \(\bm{C}(t)\eqdef(C_{1}(t),C_{2}(t))^{T}\). For notational brevity, we express this equation as
\begin{align}
	\partial_{t}\mathcal{F}(f)(k,t)=&\frac{1}{2\pi R}\biggl[A_{\gamma}I_{1}(\lambda,R,k)ik\mathcal{F}(f)(k)-A_{\gamma}\pi\abs{k}\mathcal{F}(f)(k)-2A_{\gamma}I_{2}(\lambda,R,k)\mathcal{F}(f)(k)\biggr] \label{I1I2} \\
	&-\frac{1}{R}\biggl(\dot{C}_{1}(t)\frac{1}{2}(\delta_{1k}+\delta_{-1k})+\dot{C}_{2}(t)\frac{1}{2i}(\delta_{1k}-\delta_{-1k})\biggr) \notag,
\end{align}
where \(I_{1}\) and \(I_{2}\) are the integrals bounded uniformly in \(k\).

In Appendix \ref{appendA}, we establish that for all \(k \geq 2\),
\begin{align} \label{kgeq2}
	\pi+\text{Im}(I_{1})(\lambda,R,k)+\frac{2\text{Re}(I_{2})(\lambda,R,k)}{\abs{k}}>0.
\end{align}
In Appendix \ref{appendB}, we prove that
\begin{align} \label{keq1}
	\pi+\text{Im}(I_{1})(\lambda,R,1)+2\text{Re}(I_{2})(\lambda,R,1)<0.
\end{align}
Furthermore, an application of the Riemann-Lebesgue lemma yields
\begin{align} \label{rieleb}
	\lim_{k \to \infty}\text{Im}(I_{1})(\lambda,R,k)=0, \quad \lim_{k \to \infty}\text{Re}(I_{2})(\lambda,R,k)=0.
\end{align}
We note that
\begin{align*}
	&-\frac{1}{R}\biggl(\dot{C}_{1}(t)\frac{1}{2}(\delta_{1k}+\delta_{-1k})+\dot{C}_{2}(t)\frac{1}{2i}(\delta_{1k}-\delta_{-1k})\biggr) \\
	=&\frac{A_{\gamma}}{2\pi R}\frac{3}{2}(\pi+\text{Im}(I_{1})(\lambda,R,1)+2\text{Re}(I_{2})(\lambda,R,1))\mathcal{F}(f)(1)\delta_{1k} \\
	&+\frac{A_{\gamma}}{2\pi R}\frac{3}{2}(\pi+\text{Im}(I_{1})(\lambda,R,1)+2\text{Re}(I_{2})(\lambda,R,1))\mathcal{F}(f)(-1)\delta_{-1k}.
\end{align*}

\section{The Fixed Point Argument}
To establish the well-posedness of \eqref{NEWequation}, we must first specify the precise sense in which a function \(f\) is a solution. In this section, we formulate a notion of solution to \eqref{NEWequation} that is suitable for a contraction mapping argument based on Picard's theorem in an infinite dimensional Banach space setting.

Let \(\mathbb{P}: \mathcal{F}^{0,1}(\mathbb{T}) \to \dot{Z}^{0,1}(\mathbb{T})\) be the projection operator onto the mean-zero subspace, defined by \(\mathbb{P}(f)(\alpha)\eqdef f(\alpha)-\hat{f}(0)\). Under this projection, \eqref{NEWequation} is equivalent to the system of equations
\begin{align*}
    \hat{f}(0)_{t}&=\mathcal{F}(\mathcal{R}(f))(0), \quad (\mathbb{P}f)_{t}(\alpha)=\mathbb{P}(\mathcal{R}(f))(\alpha),
\end{align*}
where \(\mathcal{R}(f)\) denotes the right hand side of \eqref{NEWequation}.
\begin{lemma} \label{equivsys}
    For any \(T>0\) and \(\nu_{0}>0\), suppose that \(f \in C((0,T];\wiener{1}) \cap C^{1}((0,T];\mathcal{F}^{0,1}(\mathbb{T}))\).
    Then \(f\) satisfies \eqref{NEWequation} if and only if it satisfies
    \begin{align}
        \hat{f}(0,t)&=\hat{f}(0,0), \quad (\mathbb{P}f)_{t}(\alpha,t)=\mathbb{P}(\mathcal{R}(f))(\alpha,t). \label{f0idpft}
    \end{align}
\end{lemma}
\begin{proof}
    For any \(T>0\) and \(\nu_{0}>0\), suppose that \(f \in C((0,T];\wiener{1}) \cap C^{1}((0,T];\mathcal{F}^{0,1}(\mathbb{T}))\). Upon integrating the left hand side of \eqref{normalnoslip} with respect to \(\alpha\), we obtain
    \begin{align*}
    	\int_{-\pi}^{\pi}z_{t}(\alpha,t) \cdot z_{\alpha}(\alpha,t)^{\perp}d\alpha=-R^{2}\int_{-\pi}^{\pi}f_{t}(\alpha,t)d\alpha.
    \end{align*}
    Upon integrating the right hand side of \eqref{normalnoslip} with respect to \(\alpha\), we obtain
    \begin{align} \label{incomp}
    	\int_{-\pi}^{\pi}\bm{u}(z(\alpha,t),t)\cdot z_{\alpha}(\alpha,t)^{\perp}d\alpha=-\int_{\Gamma}\bm{u} \cdot \bm{\nu}ds=-\int_{\Omega_{1}}\nabla \cdot \bm{u}dx=0.
    \end{align}
    If \(f\) is a solution to \eqref{NEWequation}, then
    \begin{align*}
        -R^{2}\int_{\mathbb{T}}f_{t}(\alpha,t)d\alpha=0,
    \end{align*}
    which implies that for any \(t \in (0,T]\), \(\hat{f}(0,t)=\hat{f}(0,0)\), as needed. If \eqref{f0idpft} holds, then due to \eqref{incomp}, it is clear that \(\mathcal{F}(\mathcal{R}(f))(0)=0\), which completes the proof.
\end{proof}

The evolution equation \eqref{NEWequation} can be decomposed as
\begin{align}
    f_{t}(\alpha)=\mathcal{L}(f)(\alpha)+\mathcal{N}(f)(\alpha), \label{feqnln}
\end{align}
where \(\mathcal{L}\) represents the principal linear operator and \(\mathcal{N}\) gathers the remaining nonlinear terms. Applying the projection \(\mathbb{P}\) to both sides, we obtain \((\mathbb{P}f)_{t}(\alpha)=\mathbb{P}\mathcal{L}(f)(\alpha)+\mathbb{P}\mathcal{N}(f)(\alpha)=\mathcal{L}(\mathbb{P}f)(\alpha)+\mathbb{P}\tilde{\mathcal{N}}(\mathbb{P}f)(\alpha)\), where \(\tilde{\mathcal{N}}(\mathbb{P}f)=\mathcal{N}(\mathbb{P}f+\hat{f}(0))\).

For a sufficiently small \(\nu_{0}>0\), let \(\mathfrak{B}(\dot{\mathcal{F}}_{\nu}^{1,1})\) denote the space of bounded linear operators on \(\wiener{1}\). We define the linear solution operator \(e^{\cdot \mathcal{L}}: [0,\infty) \to \mathfrak{B}(\dot{\mathcal{F}}_{\nu}^{1,1})\) in the frequency space via 
\begin{align*}
	\mathcal{F}(e^{t\mathcal{L}}g)(k)\eqdef e^{\frac{A_{\gamma}}{2\pi R}\biggl(-\pi\abs{k}+I_{1}(\lambda,R,k)ik-2I_{2}(\lambda,R,k)+(\delta_{1k}+\delta_{-1k})\frac{3}{2}(\pi+\text{Im}(I_{1})(\lambda,R,1)+2\text{Re}(I_{2})(\lambda,R,1))\biggr)t}\mathcal{F}(g)(k)
\end{align*}
for \(k \in \mathbb{Z}\). It is straightforward to verify that \(e^{\cdot \mathcal{L}}\) is a \(C^{0}\)-semigroup generated by \(\mathcal{L}\).
\begin{definition} \label{mildsoln}
	A mild solution to \eqref{NEWequation} with initial data \(f(0)=f_{0}\) is a function \(f \in C([0,T];\wiener{1})\) for some \(T>0\) and \(\nu_{0}>0\) satisfying the first equation in \eqref{f0idpft} and
	\begin{align*}
		f(0)&=f_{0}\in \mathcal{F}^{1,1}(\mathbb{T}), \quad f(t)=e^{t\mathcal{L}}f_{0}+\int_{0}^{t}e^{(t-s)\mathcal{L}}\mathbb{P}\tilde{\mathcal{N}}(\mathbb{P}f)(s)ds.
	\end{align*}
\end{definition}
In the remainder of the paper, we present a contraction mapping argument to obtain a mild solution to \eqref{NEWequation} with \(f(0)=f_{0}\). This formulation is equivalent to finding a fixed point for the initial value problem in the infinite dimensional Banach space setting:
\begin{align}
	f(t)&=\mathcal{T}(\mathbb{P}f)(t) \label{ode} \\
	f(0)&=f_{0}\in \mathcal{F}^{1,1}(\mathbb{T}), \label{init}
\end{align}
where the nonlinear solution operator \(\mathcal{T}\) is defined by
\begin{align}
	\mathcal{T}(\mathbb{P}f)(t)\eqdef e^{t\mathcal{L}}f_{0}+\int_{0}^{t}e^{(t-s)\mathcal{L}}\mathbb{P}\tilde{\mathcal{N}}(\mathbb{P}f)(s)ds. \label{contraction}
\end{align}
\begin{lemma} \label{estimates}
	Fix \(T>0\), \(\nu_{0}>0\), \(\mu>0\), \(\gamma>0\), \(\lambda>0\), and \(R>0\). For any \(g: [0,T] \to \wiener{1}\) such that \(g(s)\) is real-valued for each \(s \in [0,T]\), let
	\begin{align}
		F(g)(\alpha,t)\eqdef\int_{0}^{t}e^{(t-s)\mathcal{L}}g(\alpha,s)ds. \label{operf}
	\end{align}
	If
	\begin{align}
		\nu_{0}<&\min\biggl[\inf_{k \geq 2}\biggl[\frac{A_{\gamma}}{2\pi R}\biggl(\pi+\text{Im}(I_{1})(\lambda,R,k)+\frac{2\text{Re}(I_{2})(\lambda,R,k)}{\abs{k}}\biggr)\biggr], \label{nuthreshold} \\
		&-\frac{A_{\gamma}}{2\pi R}\frac{1}{2}(\pi+\text{Im}(I_{1})(\lambda,R,1)+2\text{Re}(I_{2})(\lambda,R,1))\biggr], \notag
	\end{align}
	then for \(0 \leq t \leq T\),
	\begin{align*}
		&\norms{F(g)(\cdot,t)}{1} \leq \int_{0}^{t}\norms{g(\cdot,\tau)}{1}d\tau, \\
		&\int_{0}^{t}\norms{F(g)(\cdot,\tau)}{2}d\tau \lesssim_{\nu_{0}, A_{\gamma}, \lambda, R} \int_{0}^{t}\norms{g(\cdot,\tau)}{1}d\tau,
	\end{align*}
	where \(\lesssim_{\nu_{0}, A_{\gamma}, \lambda, R}\) denotes an inequality up to multiplication by a constant that depends on \(\nu_{0}\), \(A_{\gamma}\), \(R\), and \(\lambda\).
\end{lemma}
\begin{proof}
	Fix \(T>0\), \(\nu_{0}>0\), \(\mu>0\), \(\gamma>0\), \(\lambda>0\), and \(R>0\). Suppose that \(g: [0,T] \to \wiener{1}\) is a function such that \(g(s)\) is real-valued for each \(s \in [0,T]\). Then for any \(t \in [0,T]\),
	\begin{align*}
		&\norms{F(g)(\cdot,t)}{1} \\
		=&2\sum_{k \geq 1}e^{\nu(t) \abs{k}}\abs{k}\abs{\mathcal{F}(F(g)(\cdot,t))(k)} \\
		\leq&2\int_{0}^{t}e^{\nu(t)-\nu(s)}e^{\nu(s)}e^{\frac{A_{\gamma}}{2\pi R}\frac{1}{2}(\pi+\text{Im}(I_{1})(\lambda,R,1)+2\text{Re}(I_{2})(\lambda,R,1))(t-s)}\abs{\mathcal{F}(g(\cdot,s))(1)}ds \\
		&+2\sum_{k \geq 2}\abs{k}\int_{0}^{t}e^{(\nu(t)-\nu(s)) \abs{k}}e^{\nu(s)\abs{k}}e^{-\frac{A_{\gamma}}{2\pi R}\abs{k}\biggl(\pi+\text{Im}(I_{1})(\lambda,R,k)+\frac{2\text{Re}(I_{2})(\lambda,R,k)}{\abs{k}}\biggr)(t-s)}\abs{\mathcal{F}(g(\cdot,s))(k)}ds.
	\end{align*}
	Choose \(\nu_{0}>0\) small enough that \eqref{nuthreshold} holds. Since \(\nu(t)-\nu(s) \leq \nu_{0}(t-s)\), we have
	\begin{align*}
		&e^{\nu(t)-\nu(s)}e^{\frac{A_{\gamma}}{2\pi R}\frac{1}{2}(\pi+\text{Im}(I_{1})(\lambda,R,1)+2\text{Re}(I_{2})(\lambda,R,1))(t-s)} \leq 1
	\end{align*}
	and for \(k \geq 2\),
	\begin{align*}
		&e^{(\nu(t)-\nu(s))\abs{k}}e^{-\frac{A_{\gamma}}{2\pi R}\abs{k}\biggl(\pi+\text{Im}(I_{1})(\lambda,R,k)+\frac{2\text{Re}(I_{2})(\lambda,R,k)}{\abs{k}}\biggr)(t-s)} \leq 1.
	\end{align*}
	It follows that
	\begin{align*}
		\norms{F(g)(\cdot,t)}{1}\leq&\int_{0}^{t}\norms{g(\cdot,s)}{1}ds.
	\end{align*}
	Similarly,
	\begin{align*}
		&\norms{F(g)(\cdot,\tau)}{2} \\
		=&2\int_{0}^{\tau}e^{\nu(\tau)-\nu(s)}e^{\frac{A_{\gamma}}{2\pi R}\frac{1}{2}(\pi+\text{Im}(I_{1})(\lambda,R,1)+2\text{Re}(I_{2})(\lambda,R,1))(\tau-s)}e^{\nu(s)}\abs{\mathcal{F}(g(\cdot,s))(1)}ds \\
		&+2\sum_{k \geq 2}\abs{k}^{2}\int_{0}^{\tau}e^{(\nu(\tau)-\nu(s)) \abs{k}}e^{-\frac{A_{\gamma}}{2\pi R}\abs{k}\biggl(\pi+\text{Im}(I_{1})(\lambda,R,k)+\frac{2\text{Re}(I_{2})(\lambda,R,k)}{\abs{k}}\biggr)(\tau-s)}e^{\nu(s)\abs{k}}\abs{\mathcal{F}(g(\cdot,s))(k)}ds.
	\end{align*}
	Since \(\nu(t)-\nu(s) \leq \nu_{0}(t-s)\), we have
	\begin{align*}
		&\int_{0}^{t}\norms{F(g)(\cdot,\tau)}{2}d\tau \\
		\leq&2\int_{0}^{t}\int_{0}^{\tau}e^{\nu(\tau)-\nu(s)}e^{\frac{A_{\gamma}}{2\pi R}\frac{1}{2}(\pi+\text{Im}(I_{1})(\lambda,R,1)+2\text{Re}(I_{2})(\lambda,R,1))(\tau-s)}e^{\nu(s)}\abs{\mathcal{F}(g(\cdot,s))(1)}dsd\tau \\
		&+2\sum_{k \geq 2}\int_{0}^{t}\abs{k}\int_{0}^{\tau}e^{(\nu(\tau)-\nu(s)) \abs{k}}e^{-\frac{A_{\gamma}}{2\pi R}\abs{k}\biggl(\pi+\text{Im}(I_{1})(\lambda,R,k)+\frac{2\text{Re}(I_{2})(\lambda,R,k)}{\abs{k}}\biggr)(\tau-s)}e^{\nu(s)\abs{k}}\abs{k}\abs{\mathcal{F}(g(\cdot,s))(k)}dsd\tau \\
		\leq&2\int_{0}^{t}e^{\nu(s)}\abs{\mathcal{F}(g(\cdot,s))(1)}\int_{s}^{t}e^{(\tau-s)\biggl(\nu_{0}+\frac{A_{\gamma}}{2\pi R}\frac{1}{2}(\pi+\text{Im}(I_{1})(\lambda,R,1)+2\text{Re}(I_{2})(\lambda,R,1))\biggr)}d\tau ds \\
		&+2\sum_{k \geq 2}\int_{0}^{t}e^{\nu(s)\abs{k}}\abs{k}\abs{\mathcal{F}(g(\cdot,s))(k)}\int_{s}^{t}\abs{k}e^{(\tau-s)\abs{k}\biggl(\nu_{0}-\frac{A_{\gamma}}{2\pi R}\biggl(\pi+\text{Im}(I_{1})(\lambda,R,k)+\frac{2\text{Re}(I_{2})(\lambda,R,k)}{\abs{k}}\biggr)\biggr)}d\tau ds.
	\end{align*}
	In view of
	\begin{align*}
		&\int_{s}^{t}e^{(\tau-s)\biggl(\nu_{0}+\frac{A_{\gamma}}{2\pi R}\frac{1}{2}\biggl(\pi+\text{Im}(I_{1})(\lambda,R,1)+2\text{Re}(I_{2})(\lambda,R,1)\biggr)\biggr)}d\tau \\
		=&\frac{1-e^{(t-s)\biggl(\nu_{0}+\frac{A_{\gamma}}{2\pi R}\frac{1}{2}\biggl(\pi+\text{Im}(I_{1})(\lambda,R,1)+2\text{Re}(I_{2})(\lambda,R,1)\biggr)\biggr)}}{-\nu_{0}-\frac{A_{\gamma}}{2\pi R}\frac{1}{2}\biggl(\pi+\text{Im}(I_{1})(\lambda,R,1)+2\text{Re}(I_{2})(\lambda,R,1)\biggr)}
	\end{align*}
	and
	\begin{align*}
		&\int_{s}^{t}\abs{k}e^{(\tau-s)\abs{k}\biggl(\nu_{0}-\frac{A_{\gamma}}{2\pi R}\biggl(\pi+\text{Im}(I_{1})(\lambda,R,k)+\frac{2\text{Re}(I_{2})(\lambda,R,k)}{\abs{k}}\biggr)\biggr)}d\tau \\
		=&\frac{1-e^{(t-s)\abs{k}\biggl(\nu_{0}-\frac{A_{\gamma}}{2\pi R}\biggl(\pi+\text{Im}(I_{1})(\lambda,R,k)+\frac{2\text{Re}(I_{2})(\lambda,R,k)}{\abs{k}}\biggr)\biggr)}}{\frac{A_{\gamma}}{2\pi R}\biggl(\pi+\text{Im}(I_{1})(\lambda,R,k)+\frac{2\text{Re}(I_{2})(\lambda,R,k)}{\abs{k}}\biggr)-\nu_{0}},
	\end{align*} 
	we obtain
	\begin{align*}
		&\int_{0}^{t}\norms{F(g)(\cdot,\tau)}{2}d\tau \\
		\leq&\frac{1}{-\nu_{0}-\frac{A_{\gamma}}{2\pi R}\frac{1}{2}\biggl(\pi+\text{Im}(I_{1})(\lambda,R,1)+2\text{Re}(I_{2})(\lambda,R,1)\biggr)}2\int_{0}^{t}e^{\nu(s)}\abs{\mathcal{F}(g(\cdot,s))(1)}ds \\
		&+\frac{1}{\inf_{k \geq 2}\biggl[\frac{A_{\gamma}}{2\pi R}\biggl(\pi+\text{Im}(I_{1})(\lambda,R,k)+\frac{2\text{Re}(I_{2})(\lambda,R,k)}{\abs{k}}\biggr)\biggr]-\nu_{0}}2\sum_{k \geq 2}\int_{0}^{t}e^{\nu(s)\abs{k}}\abs{k}\abs{\mathcal{F}(g(\cdot,s))(k)}ds \\
		\leq&\biggl(\frac{1}{-\nu_{0}-\frac{A_{\gamma}}{2\pi R}\frac{1}{2}\biggl(\pi+\text{Im}(I_{1})(\lambda,R,1)+2\text{Re}(I_{2})(\lambda,R,1)\biggr)} \\
		&+\frac{1}{\inf_{k \geq 2}\biggl[\frac{A_{\gamma}}{2\pi R}\biggl(\pi+\text{Im}(I_{1})(\lambda,R,k)+\frac{2\text{Re}(I_{2})(\lambda,R,k)}{\abs{k}}\biggr)\biggr]-\nu_{0}}\biggr)\int_{0}^{t}\norms{g(\cdot,s)}{1}ds.
	\end{align*}
\end{proof}
\begin{lemma} \label{estimates2}
	Fix \(T>0\), \(\nu_{0}>0\), \(\mu>0\), \(\gamma>0\), \(\lambda>0\), and \(R>0\). Suppose that \(f_{0}\) is a real-valued function in \(\dot{\mathcal{F}}^{1,1}(\mathbb{T})\). If \(\nu_{0}>0\) is small enough that \eqref{nuthreshold} holds, then
	\begin{align*}
		\int_{0}^{T}\norms{e^{t\mathcal{L}}f_{0}}{2}dt &\lesssim_{\nu_{0},A_{\gamma},\lambda,R}\norm{f_{0}}_{\dot{\mathcal{F}}^{1,1}}, \\
		\norms{e^{t\mathcal{L}}f_{0}}{1} &\leq \norm{f_{0}}_{\dot{\mathcal{F}}^{1,1}}, \quad 0 \leq t \leq T,
	\end{align*}
	where \(\lesssim_{\nu_{0},A_{\gamma},\lambda,R}\) denotes an inequality up to multiplication by a constant that depends on \(\nu_{0}\), \(A_{\gamma}\), \(\lambda\), and \(R\).
\end{lemma}
\begin{proof}
	Fix \(T>0\), \(\nu_{0}>0\), \(\mu>0\), \(\gamma>0\), \(\lambda>0\), and \(R>0\). Suppose that \(f_{0}\) is a real-valued function in \(\dot{\mathcal{F}}^{1,1}(\mathbb{T})\). Let \(\nu_{0}>0\) be small enough that \eqref{nuthreshold} holds. Then
	\begin{align*}
		\norms{e^{t\mathcal{L}}f_{0}}{2} \leq& 2e^{\nu(t)}e^{\frac{A_{\gamma}}{2\pi R}\frac{1}{2}(\pi+\text{Im}(I_{1})(\lambda,R,1)+2\text{Re}(I_{2})(\lambda,R,1))t}\abs{\mathcal{F}(f_{0})(1)} \\
		&+2\sum_{k \geq 2}e^{\nu(t)\abs{k}}\abs{k}^{2}e^{\frac{A_{\gamma}}{2\pi R}\biggl(-\pi\abs{k}-\text{Im}(I_{1})(\lambda,R,k)k-2\text{Re}(I_{2})(\lambda,R,k)\biggr)t}\abs{\mathcal{F}(f_{0})(k)}.
	\end{align*}
	It follows that
	\begin{align*}
		&\int_{0}^{T}\norms{e^{t\mathcal{L}}f_{0}}{2}dt \\
		\leq&2\abs{\mathcal{F}(f_{0})(1)}\frac{1}{\frac{A_{\gamma}}{2\pi R}\frac{1}{2}(\pi+\text{Im}(I_{1})(\lambda,R,1)+2\text{Re}(I_{2})(\lambda,R,1))} \\
		&\biggl(e^{\nu(t)}e^{\frac{A_{\gamma}}{2\pi R}\frac{1}{2}(\pi+\text{Im}(I_{1})(\lambda,R,1)+2\text{Re}(I_{2})(\lambda,R,1))t} \mid_{t=0}^{t=T} \\
        &-\int_{0}^{T}\frac{\partial}{\partial t}(e^{\nu(t)})e^{\frac{A_{\gamma}}{2\pi R}\frac{1}{2}(\pi+\text{Im}(I_{1})(\lambda,R,1)+2\text{Re}(I_{2})(\lambda,R,1))t}dt\biggr) \\
		&+2\sum_{k \geq 2}\abs{k}\abs{\mathcal{F}(f_{0})(k)}\frac{1}{-\frac{A_{\gamma}}{2\pi R}\biggl(\pi+\text{Im}(I_{1})(\lambda,R,k)+\frac{2\text{Re}(I_{2})(\lambda,R,k)}{\abs{k}}\biggr)} \\
		&\biggl(e^{\nu(t)\abs{k}}e^{-\frac{A_{\gamma}}{2\pi R}\abs{k}\biggl(\pi+\text{Im}(I_{1})(\lambda,R,k)+\frac{2\text{Re}(I_{2})(\lambda,R,k)}{\abs{k}}\biggr)t} \mid_{t=0}^{t=T} \\
        &-\int_{0}^{T}\frac{\partial}{\partial t}(e^{\nu(t)\abs{k}})e^{-\frac{A_{\gamma}}{2\pi R}\abs{k}\biggl(\pi+\text{Im}(I_{1})(\lambda,R,k)+\frac{2\text{Re}(I_{2})(\lambda,R,k)}{\abs{k}}\biggr)t}dt\biggr).
	\end{align*}
	Since
	\begin{align*}
		&\int_{0}^{T}\frac{\partial}{\partial t}(e^{\nu(t)})e^{\frac{A_{\gamma}}{2\pi R}\frac{1}{2}(\pi+\text{Im}(I_{1})(\lambda,R,1)+2\text{Re}(I_{2})(\lambda,R,1))t}dt \\
		=&\int_{0}^{T}\frac{\nu_{0}}{(1+t)^{2}}e^{\nu(t)}e^{\frac{A_{\gamma}}{2\pi R}\frac{1}{2}(\pi+\text{Im}(I_{1})(\lambda,R,1)+2\text{Re}(I_{2})(\lambda,R,1))t}dt \\
		\leq&\int_{0}^{T}\nu_{0}e^{\biggl(\nu_{0}+\frac{A_{\gamma}}{2\pi R}\frac{1}{2}(\pi+\text{Im}(I_{1})(\lambda,R,1)+2\text{Re}(I_{2})(\lambda,R,1))\biggr)t}dt \\
		\leq&\frac{\nu_{0}}{-\biggl(\nu_{0}+\frac{A_{\gamma}}{2\pi R}\frac{1}{2}(\pi+\text{Im}(I_{1})(\lambda,R,1)+2\text{Re}(I_{2})(\lambda,R,1))\biggr)}
	\end{align*}
	and
	\begin{align*}
		&\int_{0}^{T}\frac{\partial}{\partial t}(e^{\nu(t)\abs{k}})e^{-\frac{A_{\gamma}}{2\pi R}\abs{k}\biggl(\pi+\text{Im}(I_{1})(\lambda,R,k)+\frac{2\text{Re}(I_{2})(\lambda,R,k)}{\abs{k}}\biggr)t}dt \\
		=&\int_{0}^{T}\frac{\nu_{0}}{(1+t)^{2}}\abs{k}e^{\nu(t)\abs{k}}e^{-\frac{A_{\gamma}}{2\pi R}\abs{k}\biggl(\pi+\text{Im}(I_{1})(\lambda,R,k)+\frac{2\text{Re}(I_{2})(\lambda,R,k)}{\abs{k}}\biggr)t}dt \\
		\leq&\int_{0}^{T}\nu_{0}\abs{k}e^{\biggl(\nu_{0}-\frac{A_{\gamma}}{2\pi R}\biggl(\pi+\text{Im}(I_{1})(\lambda,R,k)+\frac{2\text{Re}(I_{2})(\lambda,R,k)}{\abs{k}}\biggr)\biggr)\abs{k}t}dt \\
		\leq&\frac{\nu_{0}}{\inf_{k \geq 2}\biggl[\frac{A_{\gamma}}{2\pi R}\biggl(\pi+\text{Im}(I_{1})(\lambda,R,k)+\frac{2\text{Re}(I_{2})(\lambda,R,k)}{\abs{k}}\biggr)\biggr]-\nu_{0}},
	\end{align*}
	we obtain
	\begin{align*}
		&\int_{0}^{T}\norms{e^{t\mathcal{L}}f_{0}}{2}dt \\
		\leq&\biggl(\frac{1}{-\frac{A_{\gamma}}{2\pi R}\frac{1}{2}(\pi+\text{Im}(I_{1})(\lambda,R,1)+2\text{Re}(I_{2})(\lambda,R,1))} \\
		&\frac{\nu_{0}}{-\biggl(\nu_{0}+\frac{A_{\gamma}}{2\pi R}\frac{1}{2}(\pi+\text{Im}(I_{1})(\lambda,R,1)+2\text{Re}(I_{2})(\lambda,R,1))\biggr)} \\
		&+\frac{1}{\inf_{k \geq 2}\biggl[\frac{A_{\gamma}}{2\pi R}\biggl(\pi+\text{Im}(I_{1})(\lambda,R,k)+\frac{2\text{Re}(I_{2})(\lambda,R,k)}{\abs{k}}\biggr)\biggr]} \\
		&\frac{\nu_{0}}{\inf_{k \geq 2}\biggl[\frac{A_{\gamma}}{2\pi R}\biggl(\pi+\text{Im}(I_{1})(\lambda,R,k)+\frac{2\text{Re}(I_{2})(\lambda,R,k)}{\abs{k}}\biggr)\biggr]-\nu_{0}} \\
		&+\frac{1}{-\frac{A_{\gamma}}{2\pi R}\frac{1}{2}(\pi+\text{Im}(I_{1})(\lambda,R,1)+2\text{Re}(I_{2})(\lambda,R,1))} \\
		&+\frac{1}{\inf_{k \geq 2}\biggl[\frac{A_{\gamma}}{2\pi R}\biggl(\pi+\text{Im}(I_{1})(\lambda,R,k)+\frac{2\text{Re}(I_{2})(\lambda,R,k)}{\abs{k}}\biggr)\biggr]}\biggr)\norms[]{f_{0}}{1}.
	\end{align*}
	Next, for \(0 \leq t \leq T\),
	\begin{align*}
		\norms{e^{t\mathcal{L}}f_{0}}{1} =& 2e^{\nu(t)}e^{\frac{A_{\gamma}}{2\pi R}\frac{1}{2}(\pi+\text{Im}(I_{1})(\lambda,R,1)+2\text{Re}(I_{2})(\lambda,R,1))t}\abs{\mathcal{F}(f_{0})(1)} \\
		&+2\sum_{k \geq 2}e^{\nu(t)\abs{k}}\abs{k}e^{-\frac{A_{\gamma}}{2\pi R}\abs{k}\biggl(\pi+\text{Im}(I_{1})(\lambda,R,k)+\frac{2\text{Re}(I_{2})(\lambda,R,k)}{\abs{k}}\biggr)t}\abs{\mathcal{F}(f_{0})(k)} \\
		\leq&2e^{\biggl(\nu_{0}+\frac{A_{\gamma}}{2\pi R}\frac{1}{2}(\pi+\text{Im}(I_{1})(\lambda,R,1)+2\text{Re}(I_{2})(\lambda,R,1))\biggr)t}\abs{\mathcal{F}(f_{0})(1)} \\
		&+2\sum_{k \geq 2}\abs{k}e^{-\biggl(\inf_{k \geq 2}\biggl[\frac{A_{\gamma}}{2\pi R}\biggl(\pi+\text{Im}(I_{1})(\lambda,R,k)+\frac{2\text{Re}(I_{2})(\lambda,R,k)}{\abs{k}}\biggr)\biggr]-\nu_{0}\biggr)\abs{k}t}\abs{\mathcal{F}(f_{0})(k)} \\
		\leq&\norms[]{f_{0}}{1}.
	\end{align*}
\end{proof}
The following result summarizes the contraction mapping argument.
\begin{theorem} \label{fixedptthm}
	Let \(\mu>0\), \(\gamma>0\), \(\lambda>0\), and \(R>0\). Consider the integral equation (\ref{ode}) with the initial condition (\ref{init}). 
	There exist \(\nu_{0}=\nu_{0}(\mu,\gamma,\lambda,R)>0\) and \(\epsilon=\epsilon(\mu,\gamma,\lambda,R)>0\) such that for any \(T>0\), if \(\norm{f_{0}}_{\dot{\mathcal{F}}^{1,1}}\leq\epsilon\), then \(\mathcal{T}\) is a contraction on the closed ball
	\begin{align*}
		B_{\epsilon,T,\nu}(e^{t\mathcal{L}}f_{0}) \eqdef \{g \in X_{T,\nu}(\mathbb{T}) : \norm{g-e^{t\mathcal{L}}f_{0}}_{X_{T,\nu}} \leq \epsilon\}.
	\end{align*}
\end{theorem}
\begin{proof}
	Let \(\mu>0\), \(\gamma>0\), \(\lambda>0\), and \(R>0\). We choose \(\nu_{0}>0\) sufficiently small that \eqref{nuthreshold} holds and let
	\begin{align*}
		\epsilon\eqdef&\min\biggl\{\epsilon_{1}(1+C_{1})^{-1},(C_{2} G(\epsilon_{1})(1+C_{1})^{2})^{-1}, \epsilon_{2}(1+C_{1})^{-1}, \frac{1}{2}(C_{3}H(\epsilon_{2},\epsilon_{2})2(1+C_{1}))^{-1}\biggr\},
	\end{align*}
	where \(\epsilon_{1}>0\) is the constant in Lemma \ref{nonlinearestimate} and \(\epsilon_{2}>0\) is the constant in Lemma \ref{nonlinearestimatescont}. We formally introduce the constants \(C_{1}=C_{1}(\nu_{0},A_{\gamma},R,\lambda)>0\), \(C_{2}=C_{2}(\nu_{0},A_{\gamma},R,\lambda)>0\), and \(C_{3}=C_{3}(\nu_{0},A_{\gamma},R,\lambda)>0\) later in the proof. Fix \(T>0\) and let \(f_{0}\in \mathcal{F}^{1,1}(\mathbb{T})\) such that \(\norm{f_{0}}_{\dot{\mathcal{F}}^{1,1}}\leq\epsilon\). First, we show that \(e^{t\mathcal{L}}f_{0} \in X_{T,\nu}\). In view of Lemma \ref{estimates2},
	\begin{align} \label{centerestimate}
		\norm{e^{t\mathcal{L}}f_{0}}_{X_{T,\nu}}=&\norm{e^{t\mathcal{L}}f_{0}}_{L^{\infty}(0,T;\wiener{1})} +\norm{e^{t\mathcal{L}}f_{0}}_{L^{1}(0,T;\wiener{2})} \\
		=&\sup_{0\leq t \leq T}\norms{e^{t\mathcal{L}}f_{0}}{1}+\int_{0}^{T}\norms{e^{t\mathcal{L}}f_{0}}{2}dt \leq C_{1}\norm{f_{0}}_{\dot{\mathcal{F}}^{1,1}} \nonumber
	\end{align}
	for some constant \(C_{1}=C_{1}(\nu_{0},A_{\gamma},\lambda,R)>0\). Next, we show that for any \(T>0\), \(\mathcal{T}\) maps the closed ball \(B_{\epsilon,T,\nu}(e^{t\mathcal{L}}f_{0})\) into itself. Let \(f \in B_{\epsilon,T,\nu}(e^{t\mathcal{L}}f_{0})\). Then
	\begin{align*}
		&\norm{\mathcal{T}(\mathbb{P}f)(t)-e^{t\mathcal{L}}f_{0}}_{X_{T,\nu}} \\
		=&\norm{\int_{0}^{t}e^{(t-s)\mathcal{L}}\mathbb{P}\tilde{N}(\mathbb{P}f)(\cdot,s)ds}_{L^{\infty}(0,T;\wiener{1})}+\norm{\int_{0}^{t}e^{(t-s)\mathcal{L}}\mathbb{P}\tilde{N}(\mathbb{P}f)(\cdot,s)ds}_{L^{1}(0,T;\wiener{2})} \\
		=&\sup_{0 \leq t \leq T}\norms{\int_{0}^{t}e^{(t-s)\mathcal{L}}\mathbb{P}\tilde{\mathcal{N}}(\mathbb{P}f)(\cdot,s)ds}{1}+\int_{0}^{T}\norms{\int_{0}^{t}e^{(t-s)\mathcal{L}}\mathbb{P}\tilde{\mathcal{N}}(\mathbb{P}f)(\cdot,s)ds}{2}dt \\
		=&\sup_{0 \leq t \leq T}\norms{F(\mathbb{P}\tilde{\mathcal{N}}(\mathbb{P}f))(\cdot,t)}{1}+\int_{0}^{T}\norms{F(\mathbb{P}\tilde{\mathcal{N}}(\mathbb{P}f))(\cdot,t)}{2}dt,
	\end{align*}
	where \(F\) is defined in (\ref{operf}). By Lemma \ref{estimates},
	\begin{align*}
		\norm{\mathcal{T}(\mathbb{P}f)(t)-e^{t\mathcal{L}}f_{0}}_{X_{T,\nu}}\leq&C_{2}\int_{0}^{T}\norms{\mathbb{P}\tilde{\mathcal{N}}(\mathbb{P}f)(\cdot,s)}{1}ds
	\end{align*}
	for some constant \(C_{2}=C_{2}(\nu_{0},A_{\gamma},\lambda,R)>0\). Since \(\tilde{\mathcal{N}}(\mathbb{P}f)=\mathcal{N}(f)\), we have
	\begin{align*}
		\norm{\mathcal{T}(\mathbb{P}f)(t)-e^{t\mathcal{L}}f_{0}}_{X_{T,\nu}}\leq&C_{2}\int_{0}^{T}\norms{\mathcal{N}(f)(\cdot,s)}{1}ds.
	\end{align*}
	Using that \(f \in B_{\epsilon,T,\nu}(e^{t\mathcal{L}}f_{0})\), \eqref{centerestimate}, and \(\norm{f_{0}}_{\dot{\mathcal{F}}^{1,1}}\leq\epsilon\), we obtain
	\begin{align*}
		\norm{f}_{X_{T,\nu}} \leq& \norm{f-e^{t\mathcal{L}}f_{0}}_{X_{T,\nu}}+\norm{e^{t\mathcal{L}}f_{0}}_{X_{T,\nu}} \leq \epsilon + C_{1}\norm{f_{0}}_{\dot{\mathcal{F}}^{1,1}} \leq (1+C_{1})\epsilon.
	\end{align*}
	Since \(\norms{f}{1} \leq \norm{f}_{X_{T,\nu}} \leq \epsilon_{1}\), by Lemma \ref{nonlinearestimate},
	\begin{align*}
		&\int_{0}^{T}\norms{\mathcal{N}(f)(\cdot,s)}{1}ds \leq G(\epsilon_{1})\int_{0}^{T}\norms{f(\cdot,s)}{2}\norms{f(\cdot,s)}{1}ds,
	\end{align*}
	where \(G=G(\nu_{0}, A_{\gamma}, R)\) is the function given in Lemma \ref{nonlinearestimate}. Moreover, since
	\begin{align*}
		\int_{0}^{T}\norms{f(\cdot,s)}{2}ds \leq \norm{f}_{X_{T,\nu}},
	\end{align*}
	we have
	\begin{align*}
		&\int_{0}^{T}\norms{f(\cdot,s)}{2}\norms{f(\cdot,s)}{1}ds \leq \norm{f}_{X_{T,\nu}}^{2} \leq (1+C_{1})^{2}\epsilon^{2}.
	\end{align*}
	Hence,
	\begin{align*}
		&\int_{0}^{T}\norms{\mathcal{N}(f)(\cdot,s)}{1}ds \leq G(\epsilon_{1})(1+C_{1})^{2}\epsilon^{2}.
	\end{align*}
	Therefore,
	\begin{align*}
		&\norm{\mathcal{T}(\mathbb{P}f)(t)-e^{t\mathcal{L}}f_{0}}_{X_{T,\nu}} \leq C_{2} G(\epsilon_{1})(1+C_{1})^{2}\epsilon^{2} \leq \epsilon.
	\end{align*}
	Lastly, we show that \(\mathcal{T}\) is a contraction on \(B_{\epsilon,T,\nu}(e^{t\mathcal{L}}f_{0})\). If \(f_{1}, f_{2} \in B_{\epsilon,T,\nu}(e^{t\mathcal{L}}f_{0})\), then
	\begin{align*}
		&\norm{\mathcal{T}(f_{1})-\mathcal{T}(f_{2})}_{X_{T,\nu}} \\
		=&\norm{\mathcal{T}(f_{1})-\mathcal{T}(f_{2})}_{L^{\infty}(0,T;\wiener{1})}+\norm{\mathcal{T}(f_{1})-\mathcal{T}(f_{2})}_{L^{1}(0,T;\wiener{2})} \\
		=&\sup_{0 \leq t \leq T}\norms{F\biggl(\mathbb{P}\tilde{N}(\mathbb{P}f_{1})(\cdot,t)-\mathbb{P}\tilde{N}(\mathbb{P}f_{2})(\cdot,t)\biggr)}{1}+\int_{0}^{T}\norms{F\biggl(\mathbb{P}\tilde{N}(\mathbb{P}f_{1})(\cdot,t)-\mathbb{P}\tilde{N}(\mathbb{P}f_{2})(\cdot,t)\biggr)}{2}dt,
	\end{align*}
	where \(F\) is defined in \eqref{operf}. By Lemma \ref{estimates},
	\begin{align*}
		\norm{\mathcal{T}(f_{1})-\mathcal{T}(f_{2})}_{X_{T,\nu}}\leq C_{3}&\int_{0}^{T}\norms{N(f_{1})(\cdot,s)-N(f_{2})(\cdot,s)}{1}ds
	\end{align*}
	for some constant \(C_{3}=C_{3}(\nu_{0},A_{\gamma},\lambda,R)>0\). For any \(i \in \{1,2\}\), \(\norms{f_{i}}{1} \leq \norm{f_{i}}_{X_{T, \nu}} \leq (1+C_{1})\epsilon \leq \epsilon_{2}\).	Hence, by Lemma \ref{nonlinearestimatescont},
	\begin{align*}
		&\norm{\mathcal{T}(f_{1})-\mathcal{T}(f_{2})}_{X_{T,\nu}} \\
		\leq&C_{3}H(\epsilon_{2},\epsilon_{2})\int_{0}^{T}\biggl((\norms{f_{1}(\cdot,s)}{1}+\norms{f_{2}(\cdot,s)}{1})\norms{f_{1}(\cdot,s)-f_{2}(\cdot,s)}{2} \\
		&+\norms{f_{1}(\cdot,s)-f_{2}(\cdot,s)}{1}(\norms{f_{1}(\cdot,s)}{2}+\norms{f_{2}(\cdot,s)}{2})\biggr)ds \\
		\leq&C_{3}H(\epsilon_{2},\epsilon_{2})\biggl(2(1+C_{1})\epsilon\int_{0}^{T}\norms{f_{1}(\cdot,s)-f_{2}(\cdot,s)}{2}ds+2(1+C_{1})\epsilon\sup_{0\leq s \leq T}\norms{f_{1}(\cdot,s)-f_{2}(\cdot,s)}{1}\biggr) \\
		\leq&C_{3}H(\epsilon_{2},\epsilon_{2})2(1+C_{1})\epsilon\norm{f_{1}-f_{2}}_{X_{T,\nu}}\leq\frac{1}{2}\norm{f_{1}-f_{2}}_{X_{T,\nu}},
	\end{align*}
	where \(H=H(\nu_{0}, A_{\gamma},R)\) is the function given in Lemma \ref{nonlinearestimatescont}. Therefore, \(\mathcal{T}\) is indeed a contraction on \(B_{\epsilon,T,\nu}(e^{t\mathcal{L}}f_{0})\).
\end{proof}
\begin{remark}
	Theorem \ref{fixedptthm} implies that for all \(t \in [0,\infty)\),
	\begin{align*}
		\norms{f(\cdot,t)}{1} \leq \norm{f_{0}}_{\dot{\mathcal{F}}^{1,1}}e^{-Ct},
	\end{align*}
	where \(C=C(\nu_{0},A_{\gamma}, \lambda,R,\norm{f_{0}}_{\dot{\mathcal{F}}^{1,1}})>0\).
\end{remark}

In summary, we have proved the following main theorem.
\begin{theorem}
	Let \(f_{0} \in \mathcal{F}^{1,1}(\mathbb{T})\). For any \(\mu>0\), \(\gamma>0\), \(\lambda>0\), and \(R>0\), there exist \(\nu_{0}(\mu,\gamma,\lambda, R)>0\) and \(\epsilon(\mu,\gamma,\lambda, R)>0\) such that for any \(T>0\), if \(\norm{f_{0}}_{\dot{\mathcal{F}}^{1,1}}\leq\epsilon\), then there exists a unique mild solution \(f \in C([0,T];\dot{Z}_{\nu}^{1,1}(\mathbb{T}))\cap L^{1}(0,T;\wiener{2})\) to (\ref{NEWequation}) as defined in Definition \ref{mildsoln}. Moreover, for any \(t \in [0,\infty)\),
	\begin{align*}
		\norms{f(\cdot,t)}{1} \leq \norm{f_{0}}_{\dot{\mathcal{F}}^{1,1}}e^{-Ct},
	\end{align*}
	where \(C=C(\nu_{0},A_{\gamma}, \lambda, R,\norm{f_{0}}_{\dot{\mathcal{F}}^{1,1}})>0\).
\end{theorem}

\section{The Nonlinear Estimates}
To close the contraction mapping argument detailed in Theorem \ref{fixedptthm}, we need to derive an appropriate estimate for the terms on the right hand side of \eqref{NEWequation} that are nonlinear in \(f\). The rest of this section is dedicated to proving the following lemmas.
\begin{lemma} \label{nonlinearestimate}
	Let \(T>0\), \(\nu_{0}>0\), \(\mu>0\), \(\gamma>0\), \(\lambda>0\), and \(R>0\). Given \(t \in [0,T]\), if \(f \in \wiener{1}\) satisfies \(\norms{f}{1}\leq\epsilon_{1}\) for some sufficiently small \(\epsilon_{1}>0\), then
	\begin{align*}
		\norms{\mathcal{N}(f)}{1} \lesssim_{\nu_{0},A_{\gamma}, \lambda, R}^{1}\norms{f_{\alpha\alpha}}{0}\norms{f_{\alpha}}{0},
	\end{align*}
	where \(\lesssim_{\nu_{0},A_{\gamma},\lambda,R}^{1}\) denotes an inequality up to multiplication by a monotone increasing function \(G=G(\nu_{0},A_{\gamma},\) \(R,\lambda): [0,\epsilon_{1}] \to [0,\infty)\) of \(\norms{f}{1}\) which depends on \(\nu_{0}\), \(A_{\gamma}\), \(\lambda\), and \(R\).
\end{lemma}
\begin{lemma} \label{nonlinearestimatescont}
	Let \(T>0\), \(\nu_{0}>0\), \(\mu>0\), \(\gamma>0\), \(\lambda>0\), and \(R>0\). Given \(t \in [0,T]\), if \(f_{1} \in \wieners{1}\) and \(f_{2} \in \wieners{1}\) satisfy \(\norms{f_{1}}{1}\leq\epsilon_{2}\) and \(\norms{f_{2}}{1}\leq\epsilon_{2}\) for some sufficiently small \(\epsilon_{2}>0\), then
	\begin{align*}
		\norms{\mathcal{N}(f_{1})-\mathcal{N}(f_{2})}{1}\lesssim_{\nu_{0},A_{\gamma},\lambda,R}^{1,2} &(\norms{f_{1}}{1}+\norms{f_{2}}{1})\norms{f_{1}-f_{2}}{2} \\
		&+\norms{f_{1}-f_{2}}{1}(\norms{f_{1}}{2}+\norms{f_{2}}{2}),
	\end{align*}
	where \(\lesssim_{\nu_{0},A_{\gamma},\lambda,R}^{1,2}\) denotes an inequality up to multiplication by a function \(H=H(\nu_{0},A_{\gamma},R,\lambda): [0,\epsilon_{2}] \times [0,\epsilon_{2}] \to [0,\infty)\) of \(\norms{f_{1}}{1}\) and \(\norms{f_{2}}{1}\), which is monotone increasing with respect to each argument.
\end{lemma}

We recall that
\begin{align*}
    \mathcal{N}(f)(\alpha)=&\biggl[\frac{1}{2\pi R} \int_{\T}\kerk(\alpha,\beta) \fmat(\alpha +\beta) d\beta-(1+2f(\alpha,t))^{\frac{1}{2}}\frac{\dot{\bm{C}}(t)\cdot \nrmv(\alpha)}{R}\biggr]_{N} \\
    =&\frac{1}{2\pi R}\int_{\mathbb{T}}\kerk_{0}(\alpha,\beta)\fmat_{N}(\alpha+\beta)+\kerk_{L}(\alpha,\beta)\fmat_{L}(\alpha+\beta)+\kerk_{L}(\alpha,\beta)\fmat_{N}(\alpha+\beta) \\
    &+\kerk_{N}(\alpha,\beta)\fmat_{0}(\alpha+\beta)+\kerk_{N}(\alpha,\beta)\fmat_{L}(\alpha+\beta)+\kerk_{N}(\alpha,\beta)\fmat_{N}(\alpha+\beta)d\beta \\
    &-\biggl(f(\alpha)\frac{\dot{\bm{C}}(t)\cdot \nrmv(\alpha)}{R}+f(\alpha)^{2}\mathcal{G}_{1}\frac{\dot{\bm{C}}(t)\cdot \nrmv(\alpha)}{R}\biggr).
\end{align*}

\subsection{The integral terms}
We can derive appropriate estimates for the six nonlinear integral terms at once by studying the integral
\begin{align*}
	&\int_{\mathbb{T}}\kerk(\alpha,\beta)\fmat(\alpha+\beta)d\beta =\int_{\mathbb{T}}\kerk(\alpha,\beta)2 A_\gamma \frac{d}{d\beta}\left(\frac{z_{\alpha}(\alpha+\beta)}{\abs{z_{\alpha}(\alpha+\beta)}}\right)d\beta =-2A_{\gamma}\int_{\mathbb{T}}\partial_{\beta}[\kerk(\alpha,\beta)]\frac{z_{\alpha}(\alpha+\beta)}{\abs{z_{\alpha}(\alpha+\beta)}}d\beta,
\end{align*}
where
\begin{align*}
	&\kerk(\alpha,\beta)(1+2f(\alpha))^{-\frac{1}{2}} \\
	=&\nrmv(\alpha) \cdot \biggl(\biggl(\frac{-1}{\lambda^{2}R^{2}\abs{2\sin(\beta/2)}^{2}\mathcal{Z}(f)}+\frac{K_{1}(\lambda R\abs{2\sin(\beta/2)}\mathcal{Z}(f)^{\frac{1}{2}})}{\lambda R\abs{2\sin(\beta/2)}\mathcal{Z}(f)^{\frac{1}{2}}}+K_{0}(\lambda R\abs{2\sin(\beta/2)}\mathcal{Z}(f)^{\frac{1}{2}})\biggr)I \\
	&+\biggl(\frac{1}{R^{2}\mathcal{Z}(f)}\frac{2}{\lambda^{2}R^{2}\abs{2\sin(\beta/2)}^{2}\mathcal{Z}(f)}-\frac{K_{2}(\lambda R\abs{2\sin(\beta/2)}\mathcal{Z}(f)^{\frac{1}{2}})}{R^{2}\mathcal{Z}(f)}\biggr) \\
	&\biggl(R^{2}\compt(f)^{2}\nrmv(\alpha) \otimes \nrmv(\alpha)+R^{2}\compt(f)\mompt(f)\nrmv(\alpha) \otimes \tgtv(\alpha) \\
	&+R^{2}\compt(f)\mompt(f)\tgtv(\alpha)\otimes \nrmv(\alpha)+R^{2}\mompt(f)^{2}\tgtv(\alpha) \otimes \tgtv(\alpha)\biggr)\biggr).
\end{align*}
We first derive an estimate for this integral and then discard the terms that do not correspond to the original nonlinear integral terms.

We observe that
\begin{align}
	&\norms{-2A_{\gamma}\int_{\mathbb{T}}\partial_{\beta}[\kerk(\alpha,\beta)]\frac{z_{\alpha}(\alpha+\beta)}{\abs{z_{\alpha}(\alpha+\beta)}}d\beta}{1} \label{integralnorm} \\
	=&\sum_{k \in \mathbb{Z} \setminus \{0\}}e^{\nu(t)\abs{k}}\abs{k}\abs{-2A_{\gamma}\int_{\mathbb{T}}\mathcal{F}\biggl(\partial_{\beta}[\kerk(\alpha,\beta)]\frac{z_{\alpha}(\alpha+\beta)}{\abs{z_{\alpha}(\alpha+\beta)}}\biggr)(k)d\beta}. \nonumber
\end{align}
Taking the derivative of \(\kerk(\alpha,\beta)(1+2f(\alpha))^{-\frac{1}{2}}\) with respect to \(\beta\), we get
\begin{align*}
	&\partial_{\beta}[\kerk(\alpha,\beta)](1+2f(\alpha))^{-\frac{1}{2}} \\
	=&\nrmv(\alpha) \cdot \biggl(\frac{\partial}{\partial\beta}\biggl(\frac{-1}{\lambda^{2}R^{2}\abs{2\sin(\beta/2)}^{2}\mathcal{Z}(f)}+\frac{K_{1}(\lambda R\abs{2\sin(\beta/2)}\mathcal{Z}(f)^{\frac{1}{2}})}{\lambda R\abs{2\sin(\beta/2)}\mathcal{Z}(f)^{\frac{1}{2}}}+K_{0}(\lambda R\abs{2\sin(\beta/2)}\mathcal{Z}(f)^{\frac{1}{2}})\biggr)I \\
	&+\frac{\partial}{\partial\beta}\biggl[\biggl(\frac{1}{R^{2}\mathcal{Z}(f)}\frac{2}{\lambda^{2}R^{2}\abs{2\sin(\beta/2)}^{2}\mathcal{Z}(f)}-\frac{K_{2}(\lambda R\abs{2\sin(\beta/2)}\mathcal{Z}(f)^{\frac{1}{2}})}{R^{2}\mathcal{Z}(f)}\biggr)R^{2}\compt(f)^{2}\biggr]\nrmv(\alpha) \otimes \nrmv(\alpha) \\
	&+\frac{\partial}{\partial\beta}\biggl[\biggl(\frac{1}{R^{2}\mathcal{Z}(f)}\frac{2}{\lambda^{2}R^{2}\abs{2\sin(\beta/2)}^{2}\mathcal{Z}(f)}-\frac{K_{2}(\lambda R\abs{2\sin(\beta/2)}\mathcal{Z}(f)^{\frac{1}{2}})}{R^{2}\mathcal{Z}(f)}\biggr)R^{2}\compt(f)\mompt(f)\biggr]\nrmv(\alpha) \otimes \tgtv(\alpha) \\
	&+\frac{\partial}{\partial\beta}\biggl[\biggl(\frac{1}{R^{2}\mathcal{Z}(f)}\frac{2}{\lambda^{2}R^{2}\abs{2\sin(\beta/2)}^{2}\mathcal{Z}(f)}-\frac{K_{2}(\lambda R\abs{2\sin(\beta/2)}\mathcal{Z}(f)^{\frac{1}{2}})}{R^{2}\mathcal{Z}(f)}\biggr)R^{2}\compt(f)\mompt(f)\biggr]\tgtv(\alpha)\otimes \nrmv(\alpha) \\
	&+\frac{\partial}{\partial\beta}\biggl[\biggl(\frac{1}{R^{2}\mathcal{Z}(f)}\frac{2}{\lambda^{2}R^{2}\abs{2\sin(\beta/2)}^{2}\mathcal{Z}(f)}-\frac{K_{2}(\lambda R\abs{2\sin(\beta/2)}\mathcal{Z}(f)^{\frac{1}{2}})}{R^{2}\mathcal{Z}(f)}\biggr)R^{2}\mompt(f)^{2}\biggr]\tgtv(\alpha) \otimes \tgtv(\alpha)\biggr).
\end{align*}
Using
\begin{align*}
	\frac{z_{\alpha}(\alpha+\beta)}{\abs{z_{\alpha}(\alpha+\beta)}}=\frac{f_{\alpha}(\alpha+\beta)}{(1+2f(\alpha+\beta))\biggl(\frac{f_{\alpha}(\alpha+\beta)^{2}}{(1+2f(\alpha+\beta))^{2}}+1\biggr)^{1/2}}\nrmv(\alpha+\beta)+\frac{1}{\biggl(\frac{f_{\alpha}(\alpha+\beta)^{2}}{(1+2f(\alpha+\beta))^{2}}+1\biggr)^{1/2}}\tgtv(\alpha+\beta),
\end{align*}
we obtain
\begin{align*}
	&\partial_{\beta}[\kerk(\alpha,\beta)]\frac{z_{\alpha}(\alpha+\beta)}{\abs{z_{\alpha}(\alpha+\beta)}}(1+2f(\alpha))^{-\frac{1}{2}} \\
	=&\biggl(\frac{\partial}{\partial\beta}\biggl(\frac{-1}{\lambda^{2}R^{2}\abs{2\sin(\beta/2)}^{2}\mathcal{Z}(f)}+\frac{K_{1}(\lambda R\abs{2\sin(\beta/2)}\mathcal{Z}(f)^{\frac{1}{2}})}{\lambda R\abs{2\sin(\beta/2)}\mathcal{Z}(f)^{\frac{1}{2}}}+K_{0}(\lambda R\abs{2\sin(\beta/2)}\mathcal{Z}(f)^{\frac{1}{2}})\biggr) \\
	&\cos\beta f_{\alpha}(\alpha+\beta)(1+2f(\alpha+\beta))^{-1}\biggl(\frac{f_{\alpha}(\alpha+\beta)^{2}}{(1+2f(\alpha+\beta))^{2}}+1\biggr)^{-1/2} \\
	&-\frac{\partial}{\partial\beta}\biggl(\frac{-1}{\lambda^{2}R^{2}\abs{2\sin(\beta/2)}^{2}\mathcal{Z}(f)}+\frac{K_{1}(\lambda R\abs{2\sin(\beta/2)}\mathcal{Z}(f)^{\frac{1}{2}})}{\lambda R\abs{2\sin(\beta/2)}\mathcal{Z}(f)^{\frac{1}{2}}}+K_{0}(\lambda R\abs{2\sin(\beta/2)}\mathcal{Z}(f)^{\frac{1}{2}})\biggr) \\
	&\sin\beta\biggl(\frac{f_{\alpha}(\alpha+\beta)^{2}}{(1+2f(\alpha+\beta))^{2}}+1\biggr)^{-1/2} \\
	&+\frac{\partial}{\partial\beta}\biggl[\biggl(\frac{1}{R^{2}\mathcal{Z}(f)}\frac{2}{\lambda^{2}R^{2}\abs{2\sin(\beta/2)}^{2}\mathcal{Z}(f)}-\frac{K_{2}(\lambda R\abs{2\sin(\beta/2)}\mathcal{Z}(f)^{\frac{1}{2}})}{R^{2}\mathcal{Z}(f)}\biggr)R^{2}\compt(f)^{2}\biggr] \\
	&\cos\beta f_{\alpha}(\alpha+\beta)(1+2f(\alpha+\beta))^{-1}\biggl(\frac{f_{\alpha}(\alpha+\beta)^{2}}{(1+2f(\alpha+\beta))^{2}}+1\biggr)^{-1/2} \\
	&+\frac{\partial}{\partial\beta}\biggl[\biggl(\frac{1}{R^{2}\mathcal{Z}(f)}\frac{2}{\lambda^{2}R^{2}\abs{2\sin(\beta/2)}^{2}\mathcal{Z}(f)}-\frac{K_{2}(\lambda R\abs{2\sin(\beta/2)}\mathcal{Z}(f)^{\frac{1}{2}})}{R^{2}\mathcal{Z}(f)}\biggr)R^{2}\compt(f)\mompt(f)\biggr] \\
	&\sin\beta f_{\alpha}(\alpha+\beta)(1+2f(\alpha+\beta))^{-1}\biggl(\frac{f_{\alpha}(\alpha+\beta)^{2}}{(1+2f(\alpha+\beta))^{2}}+1\biggr)^{-1/2} \\
	&+\frac{\partial}{\partial\beta}\biggl[\biggl(\frac{1}{R^{2}\mathcal{Z}(f)}\frac{2}{\lambda^{2}R^{2}\abs{2\sin(\beta/2)}^{2}\mathcal{Z}(f)}-\frac{K_{2}(\lambda R\abs{2\sin(\beta/2)}\mathcal{Z}(f)^{\frac{1}{2}})}{R^{2}\mathcal{Z}(f)}\biggr)R^{2}\compt(f)\mompt(f)\biggr] \\
	&\cos\beta\biggl(\frac{f_{\alpha}(\alpha+\beta)^{2}}{(1+2f(\alpha+\beta))^{2}}+1\biggr)^{-1/2}\biggr).
\end{align*}
In view of
\begin{align*}
	\biggl(\frac{f_{\alpha}^{2}(\alpha,t)}{(1+2f(\alpha,t))^{2}}+1\biggr)^{-1/2}=&\sum_{n_{1}=0}^{\infty}\sum_{n_{2}=0}^{\infty}\binom{-1/2}{n_{1}}\binom{-2n_{1}}{n_{2}}2^{n_{2}}f_{\alpha}(\alpha,t)^{2n_{1}}f(\alpha,t)^{n_{2}}, \\
	(1+2f(\alpha+\beta))^{-1}=& \sum_{n=0}^{\infty}\binom{-1}{n}2^{n}f(\alpha+\beta)^{n},
\end{align*}
we arrive at
\begin{align} \label{integrand}
	&\partial_{\beta}[\kerk(\alpha,\beta)]\frac{z_{\alpha}(\alpha+\beta)}{\abs{z_{\alpha}(\alpha+\beta)}} \\
	=&(1+2f(\alpha))^{\frac{1}{2}} \nonumber \\
	&\biggl(\frac{\partial}{\partial\beta}\biggl(\frac{-1}{\lambda^{2}R^{2}\abs{2\sin(\beta/2)}^{2}\mathcal{Z}(f)}+\frac{K_{1}(\lambda R\abs{2\sin(\beta/2)}\mathcal{Z}(f)^{\frac{1}{2}})}{\lambda R\abs{2\sin(\beta/2)}\mathcal{Z}(f)^{\frac{1}{2}}}+K_{0}(\lambda R\abs{2\sin(\beta/2)}\mathcal{Z}(f)^{\frac{1}{2}})\biggr) \nonumber \\
	&\cos\beta f_{\alpha}(\alpha+\beta)\sum_{n=0}^{\infty}\binom{-1}{n}2^{n}f(\alpha+\beta)^{n}\sum_{n_{1}=0}^{\infty}\sum_{n_{2}=0}^{\infty}\binom{-1/2}{n_{1}}\binom{-2n_{1}}{n_{2}}2^{n_{2}}f_{\alpha}(\alpha+\beta)^{2n_{1}}f(\alpha+\beta)^{n_{2}} \nonumber \\
	&-\frac{\partial}{\partial\beta}\biggl(\frac{-1}{\lambda^{2}R^{2}\abs{2\sin(\beta/2)}^{2}\mathcal{Z}(f)}+\frac{K_{1}(\lambda R\abs{2\sin(\beta/2)}\mathcal{Z}(f)^{\frac{1}{2}})}{\lambda R\abs{2\sin(\beta/2)}\mathcal{Z}(f)^{\frac{1}{2}}}+K_{0}(\lambda R\abs{2\sin(\beta/2)}\mathcal{Z}(f)^{\frac{1}{2}})\biggr) \nonumber \\
	&\sin\beta\sum_{n_{1}=0}^{\infty}\sum_{n_{2}=0}^{\infty}\binom{-1/2}{n_{1}}\binom{-2n_{1}}{n_{2}}2^{n_{2}}f_{\alpha}(\alpha+\beta)^{2n_{1}}f(\alpha+\beta)^{n_{2}} \nonumber \\
	&+\frac{\partial}{\partial\beta}\biggl[\frac{1}{\mathcal{Z}(f)}\biggl(\frac{2}{\lambda^{2}R^{2}\abs{2\sin(\beta/2)}^{2}\mathcal{Z}(f)}-K_{2}(\lambda R\abs{2\sin(\beta/2)}\mathcal{Z}(f)^{\frac{1}{2}})\biggr)\compt(f)^{2}\biggr] \nonumber \\
	&\cos\beta f_{\alpha}(\alpha+\beta)\sum_{n=0}^{\infty}\binom{-1}{n}2^{n}f(\alpha+\beta)^{n}\sum_{n_{1}=0}^{\infty}\sum_{n_{2}=0}^{\infty}\binom{-1/2}{n_{1}}\binom{-2n_{1}}{n_{2}}2^{n_{2}}f_{\alpha}(\alpha+\beta)^{2n_{1}}f(\alpha+\beta)^{n_{2}} \nonumber \\
	&+\frac{\partial}{\partial\beta}\biggl[\frac{1}{\mathcal{Z}(f)}\biggl(\frac{2}{\lambda^{2}R^{2}\abs{2\sin(\beta/2)}^{2}\mathcal{Z}(f)}-K_{2}(\lambda R\abs{2\sin(\beta/2)}\mathcal{Z}(f)^{\frac{1}{2}})\biggr)\compt(f)\mompt(f)\biggr] \nonumber \\
	&\sin\beta f_{\alpha}(\alpha+\beta)\sum_{n=0}^{\infty}\binom{-1}{n}2^{n}f(\alpha+\beta)^{n}\sum_{n_{1}=0}^{\infty}\sum_{n_{2}=0}^{\infty}\binom{-1/2}{n_{1}}\binom{-2n_{1}}{n_{2}}2^{n_{2}}f_{\alpha}(\alpha+\beta)^{2n_{1}}f(\alpha+\beta)^{n_{2}} \nonumber \\
	&+\frac{\partial}{\partial\beta}\biggl[\frac{1}{\mathcal{Z}(f)}\biggl(\frac{2}{\lambda^{2}R^{2}\abs{2\sin(\beta/2)}^{2}\mathcal{Z}(f)}-K_{2}(\lambda R\abs{2\sin(\beta/2)}\mathcal{Z}(f)^{\frac{1}{2}})\biggr)\compt(f)\mompt(f)\biggr] \nonumber \\
	&\cos\beta\sum_{n_{1}=0}^{\infty}\sum_{n_{2}=0}^{\infty}\binom{-1/2}{n_{1}}\binom{-2n_{1}}{n_{2}}2^{n_{2}}f_{\alpha}(\alpha+\beta)^{2n_{1}}f(\alpha+\beta)^{n_{2}}\biggr). \nonumber
\end{align}

To complete the derivation of an appropriate estimate for the norm \eqref{integralnorm}, we adapt the techniques from Section~5 of \cite{MR4679708} or Section~3 of \cite{MR4596732}. Specifically, the next step is to take the \(k\)th Fourier mode of the expression \eqref{integrand}. Unlike the Muskat and Peskin models, however, our system features modified Bessel functions of the second kind, which must be estimated carefully. In particular, we need to establish bounds for the Fourier modes of
\begin{align} \label{exp1}
	&\frac{\partial}{\partial\beta}\biggl(\frac{-1}{\lambda^{2}R^{2}\abs{2\sin(\beta/2)}^{2}\mathcal{Z}(f)}+\frac{K_{1}(\lambda R\abs{2\sin(\beta/2)}\mathcal{Z}(f)^{\frac{1}{2}})}{\lambda R\abs{2\sin(\beta/2)}\mathcal{Z}(f)^{\frac{1}{2}}}+K_{0}(\lambda R\abs{2\sin(\beta/2)}\mathcal{Z}(f)^{\frac{1}{2}})\biggr),
\end{align}
\begin{align} \label{exp2}
	&\frac{\partial}{\partial\beta}\biggl(\frac{2}{\lambda^{2}R^{2}\abs{2\sin(\beta/2)}^{2}\mathcal{Z}(f)}-K_{2}(\lambda R\abs{2\sin(\beta/2)}\mathcal{Z}(f)^{\frac{1}{2}})\biggr),
\end{align}
and
\begin{align} \label{exp3}
	&\frac{2}{\lambda^{2}R^{2}\abs{2\sin(\beta/2)}^{2}\mathcal{Z}(f)}-K_{2}(\lambda R\abs{2\sin(\beta/2)}\mathcal{Z}(f)^{\frac{1}{2}}).
\end{align}

For notational convenience, we let \(w(\alpha,\beta)\eqdef\lambda R \abs{2\sin(\beta/2)}\mathcal{Z}(f)^{\frac{1}{2}}\). In the case of \eqref{exp1}, we further define \(J(w)\eqdef -\frac{1}{w^{2}}+\frac{K_{1}(w)}{w}+K_{0}(w)\). We observe that
\begin{align*}
	\frac{\partial}{\partial\beta}[J(w)]=&J'(w)\frac{\partial w}{\partial\beta}=\biggl(-\frac{1}{2}+J_{1}(w)\biggr)\biggl(\frac{\cos(\beta/2)}{2\sin(\beta/2)}+\frac{1}{2\mathcal{Z}(f)}\frac{\partial}{\partial\beta}[\mathcal{Z}(f)]\biggr),
\end{align*}
where \(J_{1}(w)\eqdef\frac{1}{2}+\frac{2}{w^{2}}-K_{2}(w)-wK_{1}(w)\). Since there exists a uniform constant \(C_{J_{1}}>0\) such that \(\abs{J_{1}(y)} \leq C_{J_{1}}\abs{y}\) for all \(y \geq 0\), we obtain
\begin{align} \label{J1estimate}
	\abs{\mathcal{F}(J_{1}(w(\cdot,\beta)))(k)} \leq C_{J_{1}}\mathcal{F}(w(\cdot,\beta))(0)=C_{J_{1}}\lambda R \abs{2\sin(\beta/2)}\mathcal{F}(\mathcal{Z}(f)^{\frac{1}{2}})(0).
\end{align}
Therefore,
\begin{align} \label{derivJ}
	\frac{\partial}{\partial\beta}[J(w)]=&-\frac{1}{2}\biggl(\frac{\cos(\beta/2)}{2\sin(\beta/2)}+\frac{1}{2\mathcal{Z}(f)}\frac{\partial}{\partial\beta}[\mathcal{Z}(f)]\biggr)+J_{1}(w)\biggl(\frac{\cos(\beta/2)}{2\sin(\beta/2)}+\frac{1}{2\mathcal{Z}(f)}\frac{\partial}{\partial\beta}[\mathcal{Z}(f)]\biggr).
\end{align}
In \eqref{derivJ}, both the first and second terms can be estimated using the standard techniques from Section~5 of \cite{MR4679708} or Section~3 of \cite{MR4596732}, except that the second term requires an application of the estimate \eqref{J1estimate} to desingularize the integrand upon integration with respect to \(\beta\). The desingularization is made possible by the factor of \(\abs{2\sin(\beta/2)}\) in \eqref{J1estimate}.

For \eqref{exp2}, we introduce \(I(w)\eqdef \frac{2}{w^{2}}-K_{2}(w)\). Differentiating with respect to \(\beta\) yields
\begin{align} \label{II}
	\frac{\partial}{\partial\beta}[I(w)]=I'(w)\frac{\partial w}{\partial\beta}=I_{1}(w)\biggl(\frac{\cos(\beta/2)}{2\sin(\beta/2)}+\frac{1}{2\mathcal{Z}(f)}\frac{\partial}{\partial\beta}[\mathcal{Z}(f)]\biggr),
\end{align}
where \(I_{1}(w) \eqdef -\frac{4}{w^{2}}+\frac{w}{2}(K_{1}(w)+K_{3}(w))\). Since there exists a uniform constant \(C_{I_{1}}>0\) such that \(\abs{I_{1}(y)} \leq C_{I_{1}}\abs{y}\) for all \(y \geq 0\), it follows that
\begin{align} \label{Iestimate}
	\abs{\mathcal{F}(I_{1}(w(\cdot,\beta)))(k)} \leq C_{I_{1}}\mathcal{F}(w(\cdot,\beta))(0)=C_{I_{1}}\lambda R \abs{2\sin(\beta/2)}\mathcal{F}(\mathcal{Z}(f)^{\frac{1}{2}})(0).
\end{align}
The right hand side of \eqref{II} can be estimated via the techniques from Section~5 of \cite{MR4679708} or Section~3 of \cite{MR4596732}, except that we resolve the singularity of the integrand upon integration with respect to \(\beta\) by applying the estimate \eqref{Iestimate}.

Lastly, for \eqref{exp3} we define \(H(w)\eqdef \frac{2}{w^{2}}-K_{2}(w)=\frac{1}{2}+H_{1}(w)\), where \(H_{1}(w)\eqdef -\frac{1}{2}+ \frac{2}{w^{2}}-K_{2}(w)\). We note that there exists a uniform constant \(C_{H_{1}}>0\) such that \(\abs{H_{1}(y)} \leq C_{H_{1}}\abs{y}\) for all \(y \geq 0\). Then
\begin{align} \label{H1estimate}
	\abs{\mathcal{F}(H_{1}(w(\cdot,\beta)))(k)} \leq&C_{H_{1}}\mathcal{F}(w(\cdot,\beta))(0)=C_{H_{1}}\lambda R\abs{2\sin(\beta/2)}\mathcal{F}(\mathcal{Z}(f)^{\frac{1}{2}})(0).
\end{align}
The expression \(H\) can then be estimated using the techniques from Section~5 of \cite{MR4679708} or Section~3 of \cite{MR4596732}, in addition to \eqref{H1estimate}.

We complete the derivation of the remaining estimates in Lemmas \ref{nonlinearestimate} and \ref{nonlinearestimatescont} by combining our estimates for modified Bessel functions of the second kind with the techniques from Section~5 of \cite{MR4679708} or Section~3 of \cite{MR4596732} along with the following explicit formulas: 
\begin{align} \label{pvsinint}
	&\mbox{pv}\int_{\mathbb{T}}\frac{e^{-\frac{i\beta}{2}l}}{\sin(\beta/2)}d\beta=
	\begin{cases}
		-2i\pi &\mbox{if \(l\) is odd, \(l>1\)} \\
		4i\biggl(\frac{\pi}{2}-2\sum_{m=0}^{\frac{l}{2}-1}\frac{(-1)^{m}}{2m+1}\biggr)-2i\pi &\mbox{if \(l\) is even, \(l>1\)} \\
		0 &\mbox{if \(l=0\)},
	\end{cases}
\end{align}
\begin{align} \label{finitesum}
	\frac{\sin(\abs{k}\beta/2)}{\abs{k}\sin(\beta/2)}=\frac{1}{\abs{k}}\sum_{m=0}^{\abs{k}-1}e^{i(-2m+\abs{k}-1)\beta/2}, \quad \abs{k} \geq 1,
\end{align}
\begin{align} \label{FourierDBFA}
	\mathcal{F}(\DBTF)(k)=
	\begin{cases}
		\frac{\sin(k\beta/2)}{k\sin(\beta/2)}e^{\frac{ik\beta}{2}}\mathcal{F}(f_{\alpha})(k) &\mbox{if \(k \neq 0\)}, \\
		0 &\mbox{if \(k=0\)},
	\end{cases}
\end{align}
\begin{align}
	\mathcal{F}\biggl(\frac{\partial}{\partial\beta}[\DBTF]\biggr)(k)=e^{ik\beta}\frac{1-\frac{\sin(k\beta/2)}{k\tan(\beta/2)}e^{-ik\beta/2}}{2\sin(\beta/2)}\mathcal{F}(f_{\alpha})(k).
\end{align}

\appendix
\section{Proof of the claim in \texorpdfstring{\eqref{kgeq2}}{kgeq2}} \label{appendA}
We define \(x(\beta)\eqdef \lambda R \abs{2\sin(\beta/2)}\), which yields the derivative \(\partial_{\beta}x(\beta)=\frac{\lambda R \sin\beta}{\abs{2\sin(\beta/2)}}\). Setting \(p(x)\eqdef \frac{2}{x^{2}}-K_{2}(x)\) and applying the identity \(K_{2}'(x)=-K_{1}(x)-\frac{2}{x}K_{2}(x)\), we observe that \(p'(x)=-2\biggl(\frac{2}{x^{3}}-\frac{K_{2}(x)}{x}-K_{1}(x)\biggr)-K_{1}(x)\). Using the relation \(\frac{1}{4}(e^{-i(-k+1)\beta}+e^{-i(-k-1)\beta})=e^{ik\beta}\frac{1}{2}\cos\beta\), we find
\begin{align*}
	I_{11}\eqdef& -4\int_{\mathbb{T}}\partial_{\beta}\biggl[\biggl(\frac{2}{\lambda^{2}R^{2}\abs{2\sin(\beta/2)}^{2}}-K_{2}(\lambda R \abs{2\sin(\beta/2)})\biggr)\sin^{2}(\beta/2)\biggr]\frac{1}{4}(e^{-i(-k+1)\beta}+e^{-i(-k-1)\beta})d\beta \\
	=&\int_{\mathbb{T}}[-2p'(x)\partial_{\beta}x(\beta) \sin^{2}(\beta/2)\cos\beta-p(x)\sin\beta\cos\beta]e^{ik\beta}d\beta.
\end{align*}
Similarly, invoking \(\frac{1}{4i}(e^{i(k-1)\beta}-e^{i(k+1)\beta})=-e^{ik\beta}\frac{1}{2}\sin\beta\), we obtain
\begin{align*}
	I_{12}\eqdef&-4\int_{\mathbb{T}}\partial_{\beta}\biggl[\biggl(\frac{2}{\lambda^{2}R^{2}\abs{2\sin(\beta/2)}^{2}}-K_{2}(\lambda R\abs{2\sin(\beta/2)})\biggr)\sin(\beta/2)\cos(\beta/2)\biggr]\frac{1}{4i}(e^{i(k-1)\beta}-e^{i(k+1)\beta})d\beta \\
	=&\int_{\mathbb{T}}[p'(x)\partial_{\beta}x(\beta) \sin^{2}\beta+p(x)\sin\beta\cos\beta]e^{ik\beta}d\beta.
\end{align*}
Adding these two expressions yields
\begin{align*}
	I_{11}+I_{12}=\int_{\mathbb{T}}p'(x)\partial_{\beta}x(\beta)(1-\cos\beta)e^{ik\beta}d\beta.
\end{align*}
Since \(K_{1}'(x)=-\frac{1}{2}(K_{0}(x)+K_{2}(x))\) and \(K_{2}(x)=-K_{1}'(x)+\frac{K_{1}(x)}{x}\), we obtain
\begin{align*}
	&\frac{\lambda^{2}R^{2}4\sin(\beta/2)\cos(\beta/2)}{(\lambda^{2}R^{2}4\sin^{2}(\beta/2))^{2}}-\frac{1}{2}\biggl[K_{0}(\lambda R(4\sin^{2}(\beta/2))^{1/2})+K_{2}(\lambda R(4\sin^{2}(\beta/2))^{1/2})\biggr] \\
	&\frac{2\lambda R}{(4\sin^{2}(\beta/2))^{1/2}}\frac{\sin(\beta/2)\cos(\beta/2)}{\lambda R(4\sin^{2}(\beta/2))^{1/2}}-\frac{K_{1}(\lambda R(4\sin^{2}(\beta/2))^{1/2})}{(\lambda R(4\sin^{2}(\beta/2))^{1/2})^{2}}\frac{2\lambda R}{(4\sin^{2}(\beta/2))^{1/2}}\sin(\beta/2)\cos(\beta/2) \\
	&-K_{1}(\lambda R(4\sin^{2}(\beta/2))^{1/2})\frac{2\lambda R}{(4\sin^{2}(\beta/2))^{1/2}}\sin(\beta/2)\cos(\beta/2) \\
	=&\partial_{\beta}x(\beta)(p'(x)+K_{1}(x))\biggl(-\frac{1}{2}\biggr).
\end{align*}
It follows that
\begin{align*}
	I_{13}\eqdef& -\int_{\mathbb{T}}\biggl[\frac{1}{2\tan(\beta/2)}+2\cos\beta\partial_{\beta}x(\beta)(p'(x)+K_{1}(x))\biggl(-\frac{1}{2}\biggr)\biggr]e^{ik\beta}d\beta \\
	=&\int_{\mathbb{T}}\biggl(-\frac{1}{2\tan(\beta/2)}+\cos\beta\partial_{\beta}x(\beta)(p'(x)+K_{1}(x))\biggr)e^{ik\beta}d\beta.
\end{align*}
Therefore,
\begin{align*}
	I_{1}=I_{11}+I_{12}+I_{13} =&\int_{\mathbb{T}}\biggl(-\frac{1}{2\tan(\beta/2)}+\frac{\lambda R\sin\beta\cos\beta}{\abs{2\sin(\beta/2)}}K_{1}(\lambda R\abs{2\sin(\beta/2)}) \\
    &+\partial_{\beta}\biggl[\frac{2}{(\lambda R\abs{2\sin(\beta/2)})^{2}}-K_{2}(\lambda R\abs{2\sin(\beta/2)})\biggr]\biggr)e^{ik\beta}d\beta.
\end{align*}
Using the identity \(\frac{\partial_{\beta}x(\beta)}{x}=\frac{1}{2}\cot(\beta/2)\), we write
\begin{align*}
	\partial_{\beta}\biggl[\frac{2}{(\lambda R\abs{2\sin(\beta/2)})^{2}}-K_{2}(\lambda R\abs{2\sin(\beta/2)})\biggr]=\frac{\lambda R\sin\beta}{\abs{2\sin(\beta/2)}}K_{1}(x)+\frac{1}{2}\cot(\beta/2)\biggl(-\frac{4}{x^{2}}+2K_{2}(x)\biggr).
\end{align*}
Then
\begin{align*}
	I_{1}=\int_{\mathbb{T}}\biggl(\cot(\beta/2)\biggl(-\frac{2}{x^{2}}+K_{2}(x)-\frac{1}{2}\biggr)+\frac{\lambda R\sin\beta\cos\beta}{\abs{2\sin(\beta/2)}}K_{1}(\lambda R\abs{2\sin(\beta/2)})+\frac{\lambda R\sin\beta}{\abs{2\sin(\beta/2)}}K_{1}(x)\biggr)e^{ik\beta}d\beta.
\end{align*}
It follows that
\begin{align*}
	\text{Im}(I_{1})=&2\int_{0}^{\pi}\biggl[\cot(\beta/2)\biggl(-\frac{1}{2\lambda^{2}R^{2}\sin^{2}(\beta/2)}+K_{2}(2\lambda R\sin(\beta/2))-\frac{1}{2}\biggr) \\
	&+\frac{\lambda R\sin\beta(\cos\beta+1)}{2\sin(\beta/2)}K_{1}(2\lambda R \sin(\beta/2))\biggr]\sin(k\beta)d\beta.
\end{align*}
Now, we introduce the auxiliary functions \(A(x)\eqdef\frac{1}{x^{2}}-\frac{1}{2}\frac{K_{1}(x)}{x}-\frac{1}{4}K_{0}(x)-\frac{1}{4}K_{2}(x)-\frac{x}{2}K_{1}(x)\), \(B(x)\eqdef-\frac{4}{x^{2}}+K_{2}(x)+\frac{x}{4}(K_{1}(x)+K_{3}(x))\), and \(C(x)\eqdef\frac{2}{x^{2}}-K_{2}(x)\). By combining the identities \(-\frac{1}{4}K_{0}(x)-\frac{1}{4}K_{2}(x)=\frac{1}{2}K_{1}'(x)\) and \(K_{1}'(x)=-K_{2}(x)+\frac{1}{x}K_{1}(x)\), we deduce that \(A(x)=\frac{1}{2}C(x)-\frac{x}{2}K_{1}(x)\). Furthermore, the relation \(\frac{x}{4}(K_{1}(x)+K_{3}(x))=-\frac{x}{2}K_{2}'(x)\) implies \(C'(x)=-\frac{4}{x^{3}}+\frac{1}{2}(K_{1}(x)+K_{3}(x))\), from which we find \(B(x)=-C(x)+\frac{x}{2}C'(x)\). Since \(K_{2}'(x)=-K_{1}(x)-\frac{2}{x}K_{2}(x)\), we obtain \(C'(x)=K_{1}(x)-\frac{2}{x}C(x)\). Next, we define \(q(\beta)\eqdef -\frac{1}{2}\sin\beta C(x)\). Differentiating with respect to \(\beta\), we obtain \(q'(\beta)=\frac{1}{2}C(x)-\frac{x}{2}\cos^{2}(\beta/2)K_{1}(x)\). Therefore,
\begin{align*}
	&\int_{\mathbb{T}}\biggl(\frac{1}{\lambda^{2}R^{2}\abs{2\sin(\beta/2)}^{2}}-\frac{1}{2}\frac{K_{1}(\lambda R\abs{2\sin(\beta/2)})}{\lambda R\abs{2\sin(\beta/2)}}-\frac{1}{4}K_{0}(\lambda R\abs{2\sin(\beta/2)})-\frac{1}{4}K_{2}(\lambda R \abs{2\sin(\beta/2)}) \\
	&-\frac{\lambda R\abs{2\sin(\beta/2)}}{2}K_{1}(\lambda R\abs{2\sin(\beta/2)})\biggr)\cos\beta d\beta \\
	+&\int_{\mathbb{T}}-\sin^{2}(\beta/2)\biggl(-\frac{4}{\lambda^{2}R^{2}\abs{2\sin(\beta/2)}^{2}}+K_{2}(\lambda R\abs{2\sin(\beta/2)})+\frac{\lambda R\abs{2\sin(\beta/2)}}{4}(K_{1}(\lambda R\abs{2\sin(\beta/2)}) \\
	&+K_{3}(\lambda R\abs{2\sin(\beta/2)}))\biggr)d\beta+\int_{\mathbb{T}}-\sin^{2}(\beta/2)\biggl(\frac{2}{\lambda^{2}R^{2}\abs{2\sin(\beta/2)}^{2}}-K_{2}(\lambda R\abs{2\sin(\beta/2)})\biggr)d\beta \\
	=&\int_{\mathbb{T}}\frac{1}{2}C(x)+K_{1}(x)\frac{x}{2}(-\cos^{2}(\beta/2))d\beta=\int_{\mathbb{T}}q'(\beta)d\beta =0.
\end{align*}
Moreover, we observe that
\begin{align*}
	&\int_{\mathbb{T}}\biggl(\frac{1}{\lambda^{2}R^{2}\abs{2\sin(\beta/2)}^{2}}-\frac{1}{2}\frac{K_{1}(\lambda R\abs{2\sin(\beta/2)})}{\lambda R\abs{2\sin(\beta/2)}}-\frac{1}{4}K_{0}(\lambda R\abs{2\sin(\beta/2)})-\frac{1}{4}K_{2}(\lambda R \abs{2\sin(\beta/2)}) \\
	&-\frac{\lambda R\abs{2\sin(\beta/2)}}{2}K_{1}(\lambda R\abs{2\sin(\beta/2)})\biggr)\cos\beta e^{ik\beta}d\beta \\
	&+\int_{\mathbb{T}}-\sin^{2}(\beta/2)\biggl(-\frac{4}{\lambda^{2}R^{2}\abs{2\sin(\beta/2)}^{2}}+K_{2}(\lambda R\abs{2\sin(\beta/2)}) \\
	&+\frac{\lambda R\abs{2\sin(\beta/2)}}{4}(K_{1}(\lambda R\abs{2\sin(\beta/2)})+K_{3}(\lambda R\abs{2\sin(\beta/2)}))\biggr)e^{ik\beta}d\beta \\
	&+\int_{\mathbb{T}}-\sin^{2}(\beta/2)(2+\cos\beta)\biggl(\frac{2}{\lambda^{2}R^{2}\abs{2\sin(\beta/2)}^{2}}-K_{2}(\lambda R\abs{2\sin(\beta/2)})\biggr)e^{ik\beta}d\beta \\
	=&\int_{\mathbb{T}}(A(x)\cos\beta-\sin^{2}(\beta/2)B(x)-\sin^{2}(\beta/2)C(x))e^{ik\beta}d\beta+\int_{\mathbb{T}}-\frac{1}{2}\sin^{2}\beta C(x)e^{ik\beta}d\beta \\
	=&\int_{\mathbb{T}}q'(\beta)e^{ik\beta}d\beta+\int_{\mathbb{T}}-\frac{1}{2}\sin^{2}\beta C(x)e^{ik\beta}d\beta \\
	=&\int_{\mathbb{T}}\frac{1}{2}\sin\beta(ik-\sin\beta)\biggl(\frac{2}{\lambda^{2}R^{2}\abs{2\sin(\beta/2)}^{2}}-K_{2}(\lambda R\abs{2\sin(\beta/2)})\biggr)e^{ik\beta}d\beta.
\end{align*}
It follows that
\begin{align*}
	\text{Re}(I_{2})=-\int_{0}^{\pi}\sin\beta(k\sin(k\beta)+\sin\beta\cos(k\beta))\biggl(\frac{1}{2\lambda^{2}R^{2}\sin^{2}(\beta/2)}-K_{2}(2\lambda R \sin(\beta/2))\biggr)d\beta.
\end{align*}
We define
\begin{align*}
	f(\lambda,R,k)\eqdef& \int_{0}^{\pi}\biggl[2\biggl(-p(x)-\frac{1}{2}\biggr)\cot(\beta/2)+\frac{\lambda R\sin\beta(1+\cos\beta)}{\sin(\beta/2)}K_{1}(x)\biggr]\sin(k\beta)d\beta \\
	=&\int_{0}^{\pi}(4\lambda R\cos^{3}(\beta/2)p'(x)-\cot(\beta/2)+2p(x)\cot(\beta/2)\cos\beta)\sin(k\beta)d\beta.
\end{align*}
Since \(4\lambda R\cos^{3}(\beta/2)p'(x)=4\cos^{2}(\beta/2)\frac{\partial}{\partial\beta}[p(x)]\), integration by parts yields
\begin{align*}
	\int_{0}^{\pi}4\lambda R\cos^{3}(\beta/2)p'(x)\sin(k\beta)d\beta =\int_{0}^{\pi}p(x)(-4k\cos^{2}(\beta/2)\cos(k\beta)+4\cos(\beta/2)\sin(\beta/2)\sin(k\beta))d\beta.
\end{align*}
Moreover, we let
\begin{align*}
	g(\lambda,R,k)\eqdef& -\int_{0}^{\pi}\sin\beta(k\sin(k\beta)+\sin\beta\cos(k\beta))\biggl(\frac{1}{2\lambda^{2}R^{2}\sin^{2}(\beta/2)}-K_{2}(2\lambda R\sin(\beta/2))\biggr)d\beta \\
	=&-\int_{0}^{\pi}\sin\beta(k\sin(k\beta)+\sin\beta\cos(k\beta))p(x)d\beta.
\end{align*}
For \(k \neq 0\), it follows that
\begin{align*}
	f(\lambda,R,k)+\frac{2}{k}g(\lambda,R,k)=-\pi-\int_{0}^{\pi}p(x)\biggl(\frac{2\cos(k\beta)(k^{2}(1+\cos\beta)+\sin^{2}\beta)}{k}-2\cos\beta\cot(\beta/2)\sin(k\beta)\biggr)d\beta.
\end{align*}
We proceed to prove that for \(k>1\),
\begin{align*}
	-\int_{0}^{\pi}p(x)\biggl(\frac{2\cos(k\beta)(k^{2}(1+\cos\beta)+\sin^{2}\beta)}{k}-2\cos\beta\cot(\beta/2)\sin(k\beta)\biggr)d\beta>0.
\end{align*}
Let \(W(\beta)\eqdef p(x(\beta))\), where \(x(\beta)=2\lambda R\sin(\beta/2)\). We note that \(W(\beta)>0\), \(\lim_{\beta \to 0}W(\beta)=\frac{1}{2}\), and \(W'(\beta) \leq 0\). We also define the auxiliary function
\begin{align*}
	M_{k}(\beta)\eqdef \int_{0}^{\beta}\frac{2\cos(k\beta')(k^{2}(1+\cos\beta')+\sin^{2}\beta')}{k}-2\cos\beta'\cot(\beta'/2)\sin(k\beta')d\beta'.
\end{align*}
Integration by parts yields
\begin{align*}
	&-\int_{0}^{\pi}W(\beta)\biggl(\frac{2\cos(k\beta)(k^{2}(1+\cos\beta)+\sin^{2}\beta)}{k}-2\cos\beta\cot(\beta/2)\sin(k\beta)\biggr)d\beta \\
	=&-W(\pi)M_{k}(\pi)+\int_{0}^{\pi}W'(\beta)M_{k}(\beta)d\beta.
\end{align*}
Since \(W(\beta)>0\) and \(W'(\beta) \leq 0\), it suffices to show that \(M_{k}(\beta) \leq 0\) for all \(\beta \in (0,\pi)\) and \(M_{k}(\pi)<0\).

\subsection{The \texorpdfstring{\(k=2\)}{k=2} case}
A straightforward calculation shows that
\begin{align*}
	M_{2}(\beta)=&\int_{0}^{\beta}\sin^{2}\beta'(2\cos^{2}\beta'-4\cos\beta'-5)d\beta' =-\frac{9}{4}\beta+\frac{5}{4}\sin(2\beta)-\frac{1}{16}\sin(4\beta)-\frac{4}{3}\sin^{3}\beta.
\end{align*}
Then \(M_{2}(\pi)=-\frac{9\pi}{4}<0\). To show that \(M_{2} \leq 0\) on \([0,\pi]\), we divide the interval into two overlapping ones. First, observe that the term \(2\cos^{2}\beta'-4\cos\beta'-5=2(\cos\beta'-1)^{2}-7\) is strictly negative if and only if \(\cos\beta'>1-\sqrt{3.5}\), which corresponds to the condition \(\beta' < \arccos(1-\sqrt{3.5})\). Therefore, \(M_{2}'(\beta)=\sin^{2}\beta(2(\cos\beta-1)^{2}-7)<0\) on \((0,\arccos(1-\sqrt{3.5}))\). Since \(M_{2}(0)=0\), it follows that \(M_{2}<0\) on this interval. Now, if \(\beta \in (\frac{7}{12},\pi)\), then \(M_{2}(\beta) \leq -\frac{9}{4}\beta+\frac{5}{4}+\frac{1}{16}=-\frac{9}{4}\beta+\frac{21}{16}<0\). Therefore, \(M_{2}<0\) on this interval.

\subsection{The \texorpdfstring{\(k \geq 3\)}{kgeq3} case}
We first establish the sign of \(M_{k}\) on \((0,\frac{\pi}{k}]=(0,\frac{\pi}{2k})\cup[\frac{\pi}{2k},\frac{\pi}{k}]\). The integrand of \(M_{k}\) can be written as \(2(1+\cos\beta')\cos(k\beta')(k+\frac{1-\cos\beta'}{k}-\cot\beta'\tan(k\beta'))\). Since \(\tan(k\beta') \geq k\tan\beta'\) for any \(\beta' \in (0,\frac{\pi}{2k})\), the integrand of \(M_{k}\) is less than or equal to \(2(1+\cos\beta')\cos(k\beta')\frac{1-\cos\beta'}{k}\) on this interval, which is non-positive. It follows that \(M_{k} \leq 0\) on this interval. Since \(\cos(k\beta') \leq 0\) and \(\sin(k\beta')\geq 0\) for any \(\beta' \in [\frac{\pi}{2k},\frac{\pi}{k}]\), the integrand of \(M_{k}(\beta)\) is non-positive there. Therefore, \(M_{k} \leq 0\) on this interval.

Next, we establish the sign of \(M_{k}\) on \((\frac{\pi}{k},\pi]\). The integrand of \(M_{k}(\beta)\) can be written as
\begin{align*}
	&-2\cot(\beta'/2)\sin(k\beta')+\frac{\partial}{\partial\beta'}[2(1+\cos\beta')\sin(k\beta')]+4\sin\beta'\sin(k\beta')+\frac{2}{k}\sin^{2}\beta'\cos(k\beta').
\end{align*}
Then
\begin{align*}
	M_{k}(\beta)=&-2\int_{0}^{\beta}\cot(\beta'/2)\sin(k\beta')d\beta'+2(1+\cos\beta)\sin(k\beta)+\int_{0}^{\beta}4\sin\beta'\sin(k\beta')+\frac{2}{k}\sin^{2}\beta'\cos(k\beta')d\beta'.
\end{align*} 
Since \(\cot(\beta/2) \geq \frac{2}{\beta}-\frac{\beta}{4}\) for any \(\beta \in (0,\pi)\) and the minimum of \(\mbox{Si}(y)\) for \(y \geq \pi\) is attained at \(y=2\pi\), we obtain
\begin{align*}
	\int_{0}^{\beta}\cot(\beta'/2)\sin(k\beta')d\beta' \geq& \int_{0}^{\beta}\frac{2}{\beta'}\sin(k\beta')d\beta'-\int_{0}^{\beta}\frac{\beta'}{4}\sin(k\beta')d\beta'\geq 2\mbox{Si}(2\pi)-\frac{-k\beta\cos(k\beta)+\sin(k\beta)}{4k^{2}}.
\end{align*}
Then
\begin{align*}
	-2\int_{0}^{\beta}\cot(\beta'/2)\sin(k\beta')d\beta' \leq& -4\mbox{Si}(2\pi)+\frac{-k\beta\cos(k\beta)+\sin(k\beta)}{2k^{2}} \leq -4\mbox{Si}(2\pi)+\frac{k\pi+1}{2k^{2}}.
\end{align*}
Moreover, we note that
\begin{align*}
	&\int_{0}^{\beta}4\sin\beta'\sin(k\beta')+\frac{2}{k}\sin^{2}\beta'\cos(k\beta')d\beta' \\
	=&2\biggl(\frac{\sin((k-1)\beta)}{k-1}-\frac{\sin((k+1)\beta)}{k+1}\biggr)+\frac{1}{k}\biggl(\frac{\sin(k\beta)}{k}-\frac{\sin((k-2)\beta)}{2(k-2)}-\frac{\sin((k+2)\beta)}{2(k+2)}\biggr) \\
	\leq&2\biggl(\frac{1}{k-1}+\frac{1}{k+1}\biggr)+\frac{1}{k}\biggl(\frac{1}{k}+\frac{1}{2(k-2)}+\frac{1}{2(k+2)}\biggr)=\frac{4k}{k^{2}-1}+\frac{1}{k^{2}}+\frac{1}{k^{2}-4}.
\end{align*}
Hence,
\begin{align*}
	M_{k}(\beta) \leq& -4\mbox{Si}(2\pi)+\frac{k\pi+1}{2k^{2}}+2(1+\cos\beta)\sin(k\beta)+\frac{4k}{k^{2}-1}+\frac{1}{k^{2}}+\frac{1}{k^{2}-4} \\
	\leq& -4\mbox{Si}(2\pi)+\frac{k\pi+1}{2k^{2}}+2(1+\cos(\pi/k))+\frac{4k}{k^{2}-1}+\frac{1}{k^{2}}+\frac{1}{k^{2}-4}.
\end{align*}
This bound implies that \(M_{k} < 0\) on \((\frac{\pi}{k},\pi]\).

\section{Proof of the claim in \texorpdfstring{\eqref{keq1}}{keq1}} \label{appendB}
We begin by noting that
\begin{align*}
	M_{1}(\beta)=&\int_{0}^{\beta}2\sin^{2}\beta'\cos\beta'd\beta'=\frac{2}{3}\sin^{3}\beta.
\end{align*}
Integration by parts yields
\begin{align*}
	&-\int_{0}^{\pi}W(\beta)2\sin^{2}\beta\cos\beta d\beta =-W(\pi)M_{1}(\pi)+\int_{0}^{\pi}W'(\beta)M_{1}(\beta)d\beta=\int_{0}^{\pi}W'(\beta)M_{1}(\beta)d\beta<0
\end{align*}
since \(M_{1}> 0\) and \(W'(\beta)<0\) on \((0,\pi)\), as needed.

\nocite{*}
\bibliographystyle{abbrv}
\bibliography{bibliography.bib}

@article{Ehlers2021,
	author    = {Wolfgang Ehlers},
	title     = {Darcy, Forchheimer, Brinkman and Richards: classical hydromechanical equations and their significance in the light of the TPM},
	journal   = {Archive of Applied Mechanics},
	volume    = {91},
	number    = {2},
	pages     = {321--341},
	year      = {2021},
	doi       = {10.1007/s00419-020-01802-3},
	url       = {https://link.springer.com/article/10.1007/s00419-020-01802-3}
}

@article {MR2389403,
	AUTHOR = {Tsai, C. C.},
	TITLE = {Solutions of slow {B}rinkman flows using the method of
	fundamental solutions},
	JOURNAL = {Internat. J. Numer. Methods Fluids},
	FJOURNAL = {International Journal for Numerical Methods in Fluids},
	VOLUME = {56},
	YEAR = {2008},
	NUMBER = {7},
	PAGES = {927--940},
	ISSN = {0271-2091,1097-0363},
	MRCLASS = {76S05 (76M25)},
	MRNUMBER = {2389403},
	DOI = {10.1002/fld.1559},
	URL = {https://doi.org/10.1002/fld.1559},
}

@article {MR4596732,
	AUTHOR = {Garc\'ia-Ju\'arez, Eduardo and Mori, Yoichiro and Strain,
	Robert M.},
	TITLE = {The {P}eskin problem with viscosity contrast},
	JOURNAL = {Anal. PDE},
	FJOURNAL = {Analysis \& PDE},
	VOLUME = {16},
	YEAR = {2023},
	NUMBER = {3},
	PAGES = {785--838},
	ISSN = {2157-5045,1948-206X},
	MRCLASS = {35Q35 (35C10 35C15 35R11 35R35 76D07)},
	MRNUMBER = {4596732},
	DOI = {10.2140/apde.2023.16.785},
	URL = {https://doi.org/10.2140/apde.2023.16.785},
}

@article {MR4679708,
	AUTHOR = {Gancedo, Francisco and Garc\'ia-Ju\'arez, Eduardo and Patel,
	Neel and Strain, Robert M.},
	TITLE = {Global regularity for gravity unstable {M}uskat bubbles},
	JOURNAL = {Mem. Amer. Math. Soc.},
	FJOURNAL = {Memoirs of the American Mathematical Society},
	VOLUME = {292},
	YEAR = {2023},
	NUMBER = {1455},
	PAGES = {v+87},
	ISSN = {0065-9266,1947-6221},
	ISBN = {978-1-4704-6764-7; 978-1-4704-7700-4},
	MRCLASS = {35Q35 (35A01 35D30 35Q86 76T10)},
	MRNUMBER = {4679708},
	MRREVIEWER = {Huy\ Quang\ Nguyen},
	DOI = {10.1090/memo/1455},
	URL = {https://doi.org/10.1090/memo/1455},
}

@article {MR4756023,
	AUTHOR = {Kundu, Sahil and Maharana, Surya Narayan and Mishra,
	Manoranjan},
	TITLE = {Existence and uniqueness of solution to unsteady
	{D}arcy-{B}rinkman problem with {K}orteweg stress for
	modelling miscible porous media flow},
	JOURNAL = {J. Math. Anal. Appl.},
	FJOURNAL = {Journal of Mathematical Analysis and Applications},
	VOLUME = {539},
	YEAR = {2024},
	NUMBER = {2},
	PAGES = {Paper No. 128532, 20},
	ISSN = {0022-247X,1096-0813},
	MRCLASS = {35Q35 (35B30 76R50 76S05)},
	MRNUMBER = {4756023},
	DOI = {10.1016/j.jmaa.2024.128532},
	URL = {https://doi.org/10.1016/j.jmaa.2024.128532},
}

@article {MR4545960,
	AUTHOR = {Kumankat, Nisachon and Wuttanachamsri, Kanognudge},
	TITLE = {Well-posedness of generalized {S}tokes-{B}rinkman equations
	modeling moving solid phases},
	JOURNAL = {Electron. Res. Arch.},
	FJOURNAL = {Electronic Research Archive},
	VOLUME = {31},
	YEAR = {2023},
	NUMBER = {3},
	PAGES = {1641--1661},
	ISSN = {2688-1594},
	MRCLASS = {76S05 (47N20)},
	MRNUMBER = {4545960},
	MRREVIEWER = {Mirela\ Kohr},
	DOI = {10.3934/era.2023085},
	URL = {https://doi.org/10.3934/era.2023085},
}

@article {MR4314117,
	AUTHOR = {Mohan, Manil T.},
	TITLE = {{$\Bbb L^p$}-solutions of deterministic and stochastic
	convective {B}rinkman-{F}orchheimer equations},
	JOURNAL = {Anal. Math. Phys.},
	FJOURNAL = {Analysis and Mathematical Physics},
	VOLUME = {11},
	YEAR = {2021},
	NUMBER = {4},
	PAGES = {Paper No. 164, 33},
	ISSN = {1664-2368,1664-235X},
	MRCLASS = {76D06 (35Q30 47D06 76D03)},
	MRNUMBER = {4314117},
	MRREVIEWER = {Nikolai\ Vasilievich\ Chemetov},
	DOI = {10.1007/s13324-021-00595-0},
	URL = {https://doi.org/10.1007/s13324-021-00595-0},
}

@article {MR3633543,
	AUTHOR = {Howell, Jason S. and Neilan, Michael and Walkington, Noel J.},
	TITLE = {A dual-mixed finite element method for the {B}rinkman problem},
	JOURNAL = {SMAI J. Comput. Math.},
	FJOURNAL = {SMAI Journal of Computational Mathematics},
	VOLUME = {2},
	YEAR = {2016},
	PAGES = {1--17},
	ISSN = {2426-8399},
	MRCLASS = {65N30 (65N15 76M10 76S05)},
	MRNUMBER = {3633543},
	MRREVIEWER = {T.\ A.\ Angelov},
	DOI = {10.5802/smai-jcm.7},
	URL = {https://doi.org/10.5802/smai-jcm.7},
}

@article {MR4673875,
	AUTHOR = {Cameron, Stephen and Strain, Robert M.},
	TITLE = {Critical local well-posedness for the fully nonlinear {P}eskin
	problem},
	JOURNAL = {Comm. Pure Appl. Math.},
	FJOURNAL = {Communications on Pure and Applied Mathematics},
	VOLUME = {77},
	YEAR = {2024},
	NUMBER = {2},
	PAGES = {901--989},
	ISSN = {0010-3640,1097-0312},
	MRCLASS = {35Q35 (35C15 35R11 35R35 76D07)},
	MRNUMBER = {4673875},
	DOI = {10.1002/cpa.22139},
	URL = {https://doi.org/10.1002/cpa.22139},
}

\end{document}